\documentclass[11pt,psamsfonts]{amsart}
\usepackage{palatino}
\usepackage{euler,color}
\usepackage[normalem]{ulem}
\usepackage{enumitem,comment}
\usepackage{manfnt} % for \dbend
\usepackage{soul}

\usepackage[all]{xy}
\usepackage{amssymb,amscd,amsmath,amsthm,dsfont}
\usepackage{tikz-cd}
\usepackage{mathabx,epsfig}
\chardef\bslash=`\\ % p. 424, TeXbook
\newtheorem{theorem}[equation]{Theorem}
\newtheorem{cor}[equation]{Corollary}
\newtheorem{lemma}[equation]{Lemma}
\newtheorem{proposition}[equation]{Proposition}
\newtheorem{prop}[equation]{Proposition}

\theoremstyle{remark}
\newtheorem{remark}[equation]{Remark}

\theoremstyle{definition}
\newtheorem{definition}[equation]{Definition}

\numberwithin{equation}{section}

\newcommand{\thmref}[1]{Theorem~\ref{#1}}
\newcommand{\secref}[1]{Section~\ref{#1}}
\newcommand{\proref}[1]{Proposition~\ref{#1}}
\newcommand{\lemref}[1]{Lemma~\ref{#1}}
\newcommand{\corref}[1]{Corollary~\ref{#1}}
\newcommand{\remref}[1]{Remark~\ref{#1}}

\newcommand{\defref}[1]{Definition~\ref{#1}}

\newcommand{\clsp}{\overline{\operatorname{span}}}
\newcommand{\lsp}{\operatorname{span}}

 \newcommand{\Aut}{\operatorname{Aut}}

\newcommand{\cq}{\mathcal{C}_{\Q}}

\newcommand{\qd}{\widehat{\Q}_+^*}

\newcommand{\nxx}[1]{{\mathbb N^\times}\! \ltimes (\mathbb Z/ #1 \mathbb Z)}
\newcommand{\nxxf}[1]{{\mathbb N^\times}\! \ltimes ({\tfrac{1}{#1}}\Z/\Z)}
\newcommand{\nxxqz}{{\mathbb N^\times}\! \ltimes (\mathbb Q/ \mathbb Z)}
\newcommand{\tnxxqz}{\TT(\nx\!\ltimes (\mathbb{Q}/\mathbb{Z}))}

\newcommand{\tnxx}[1]{\TT(\nx\!\ltimes (\mathbb{Z}/#1\mathbb{Z}))}
\newcommand{\tnxxf}[1]{\TT(\nx\ltimes (\tfrac{1}{#1}\Z/\Z))}

\newcommand{\nx}{\mathbb N^{\times}}
\newcommand{\inv}{^{-1}}

\newcommand{\ve}{\varepsilon}

\newcommand{\Gal}{\mathrm{Gal}}

\newcommand{\N}{\mathbb N}

\newcommand{\tnxnx}{\mathcal T(\nxnx)}
\newcommand{\tzxnx}{\mathcal T(\zxnx)}

\newcommand{\Z}{\mathbb Z}
\newcommand{\Q}{\mathbb Q} 
\newcommand{\qx}{\mathbb Q^\times_+}

\newcommand{\C}{\mathbb C}
\newcommand{\R}{\mathbb R}
\newcommand{\T}{\mathbb T}

\newcommand{\TT}{\mathcal T}
\newcommand{\HH}{\mathcal H}
\newcommand{\F}{\mathcal F}

\newcommand{\QQ}{\mathcal{Q}}

\newcommand{\mfc}{\mathfrak{C}}

\newcommand{\primes}{\mathcal P}

\def\lcm{\operatorname{lcm}}

\newcommand{\ord}{\mathrm{ord}}
\newcommand{\supp}{\operatorname{supp}}
\newcommand{\subg}{\operatorname{\mathrm {Subg}}}

\newcommand{\euler}{\varphi}

\def\eps{\varepsilon}

\newcommand{\wt}[1]{\widetilde{#1}}
\newcommand{\im}{\mathrm{im}}

\newcommand{\nxxn}{{\mathbb N^\times}\! \ltimes \mathbb N}
\newcommand{\nxxz}{{\mathbb N^\times}\! \ltimes \mathbb Z}
\newcommand{\zxnx}{\mathbb{Z}\rtimes \nx}

\newcommand{\qxxq}{{\mathbb Q_+^\times}\!\! \ltimes \mathbb Q}
\newcommand{\nxnx}{{\mathbb N \rtimes \mathbb N^\times}}

\def\tnxxn{\TT(\nxxn)}
\def\tnxxz{\TT(\nxxz)}
\def\bcmu{\upsilon}

\def\mfd{\mathfrak D}
\def\mfb{\mathfrak B}
\def\Spec{\operatorname{Spec}}
\def\zed{\mathfrak z}
\newcommand{\Ad}{\operatorname{Ad}}

\newcommand{\<}[1]{\left<#1\right>}

\newcommand{\KMS}[1]{\text{$\mathrm{KMS}_{#1}$}}
\def\kmsb{$\mathrm{KMS}_\beta$ }
\def\kms1{$\mathrm{KMS}_1$}
\newcommand{\mt}{\mathcal{M}(\mathbb{T})}
\newcommand{\mz}[1]{\mathcal{M}(Z_{#1})}
\def\M{\mathcal M}

\newcommand{\III}[1]{\text{$\mathrm{III}_{#1}$}}
\newcommand{\wh}[1]{\widehat{#1}}

\RequirePackage{color}

\def\tcb{\textcolor{blue}}

\newcounter{listpause}

\title[Supercritical phase transition]{Supercritical phase transition on the  Toeplitz algebra of $\nxxz$}
\date{8 October 2025}
\thanks{This research was supported by the Natural Sciences and Engineering Research Council of Canada.}
\keywords{Toeplitz algebra; Boundary quotient; KMS states}

\author[Marcelo~Laca]{Marcelo Laca}
\address{ Department of Mathematics and Statistics, University of
Victoria, Victoria, BC V8W 3P4, Canada}
\email{laca@uvic.ca}
\author[Tyler~Schulz]{Tyler Schulz}
\email{tschulz@uvic.ca}

\begin{document}

%Remove (or greatly reduce) mentions of \nxxn and replace it with \nxxz.

%Only repeat results that are being changed. Otherwise, reference.

\begin{abstract} 
We study the  high-temperature equilibrium for the C*-algebra $\tnxxz$ recently considered by an Huef, Laca and Raeburn.  We show that the simplex of \kmsb states at each inverse temperature $\beta$ in the critical interval $(0,1]$ is a Bauer simplex whose space of extreme points is homeomorphic to  $\N \sqcup\{\infty\}$. This is in contrast to the uniqueness of equilibrium at high temperature observed in previously considered systems arising from number theory. We also show that quotients of our system exhibit spontaneous symmetry-breaking by finite cyclotomic Galois groups and establish their connection to the Bost-Connes phase transition.
\end{abstract}

%%%%%%%%%%%%%%%%%%%%%%%  
\dedicatory{To the memory of Iain Raeburn}
\maketitle
\section{Introduction}\label{sec:introduction}

The study of equilibrium states of C*-dynamical systems from number theory has been an increasingly active area of research since the seminal paper \cite{bos-con} in which  Bost and Connes exhibit a phase transition with spontaneous symmetry-breaking on a noncommutative Hecke C*-algebra. Their construction has been generalized in several ways, to  semigroup crossed products, to more general Hecke algebras, to groupoid C*-algebras,  to Toeplitz algebras of $ax+b$ monoids of algebraic integers, and to  C*-algebras associated to $K$-lattices \cite{bcalg,har-lei,LvF,LLN,ha-pau,CM2006,CMR,LR-advmath,CDL}. In a vast majority of the existing constructions there is a critical value $T_c = 1/\beta_c$ of the temperature above which the simplex of \kmsb states consists of a single point, but below which the nontrivial structure of the simplex  sheds light on the original structure used in the construction. Notably, for Bost--Connes type systems associated to number fields, the extremal equilibrium states at low temperature carry a free transitive action of the Galois group of the maximal abelian extension of the field, pointing to a tantalizing connection with concrete class field theory \cite{CMR,LNT,yal}.

Developed along similar lines in \cite{cun2,LR-advmath,CDL} are the Toeplitz-type systems for $ax+b$ semigroups of algebraic integers, which have played an important role  in the study of C*-algebras of general semigroups \cite{XL}. Furthermore, the phase transition observed at low temperature for these systems has brought about an important characterization of KMS states (and, in particular, traces) in terms of orbits and isotropy groups for groupoid C*-algebras \cite{nes2013}.
Another interesting avenue of research motivated by this is the study of phase transition of a system at low temperatures, and the analysis of the Toeplitz-type  systems suggest  that this `crystallization' process is related to the $K$-theory of the C*-algebra \cite{LNY}.

In recent work \cite{aHLR21}, an Huef, Laca, and Raeburn studied the structure of the Toeplitz C*-algebra $\tnxxn$ generated by the left regular representation of $\nxxn$ on $\ell^2(\nxxn)$. Here $\nxxn$ denotes the semidirect product of the nonzero natural numbers $\nx$ acting by multiplication on $\N$, where the operation is  $(a,m) (b,n) = (ab,bm+n)$ for $a,b \in \nx$ and $m,n \in \N$. They showed that $\tnxxn$ has a natural dynamics and that for large inverse temperatures ($\beta\in (1,\infty)$) the \kmsb states of the resulting Toeplitz system correspond to probability measures on the unit circle. Intriguingly, they pointed out that there are more than one \kmsb states at the critical inverse temperature $\beta =1$  \cite[Examples 9.1--9.3]{aHLR21}. %\tcr{This is in stark contrast to what happens for the Toeplitz system of the opposite monoid $\nxnx$, which has a unique equilibrium state at its critical inverse temperature \cite[Theorem 7.1]{LR-advmath}.} 
This unprecedented high-temperature phase transition motivates the present work, in which we advance the study of equilibrium for the Toeplitz system of $\nxxn$ by describing the simplex of \kmsb states in the supercritical temperature range $T= \beta\inv \geq 1$, that is, for inverse temperatures $\beta \in [0,1]$, see \thmref{thm:main} below.

\begin{comment}
Our work fits within the current growing interest of comparing the Toeplitz algebra of a monoid with that of its opposite, of which the pair $\nxnx$ and $\nxxn$ presents a case in point. 
The boundary quotient of $\tnxnx$ is the C*-algebra $\QQ_\N$ introduced by Cuntz, who showed it is simple and purely infinite \cite[Theorem 3.4]{cun2}, while the boundary quotient of $\tnxxn$ is the C*-algebra $C^*(\qx\ltimes \Q)$, which has uncountably many 1-dimensional representations \cite[Theorem 6.10]{aHLR21}. Despite this difference,  $\tnxnx$ and $\tnxxn$ are KK-equivalent, so in particular they have isomorphic K-theory \cite{CEL}. Similarly, we see agreement and discrepancy when we look at the equilibrium states of the Toeplitz systems of $\nxnx$ and of $\nxxn$, both with their natural dynamics. In both cases the \kmsb states for sufficiently large inverse temperature are parameterized by probability measures on $\T$, \cite{aHLR21, LR-advmath}. However, the partition function for $\tnxnx$ is $\zeta(\beta-1)$, while the partition function for $\tnxxn$ is $\zeta(\beta)$, so the critical inverse temperatures are $\beta=2$ and $\beta=1$, respectively. More significantly, Laca and Raeburn showed in \cite{LR-advmath} that $\tnxnx$ has a unique \kmsb state for each $\beta$ in the supercritical interval $[1,2]$ while, as we unveil here, the structure of equilibrium states of $\tnxxn$  for supercritical $\beta$ is quite intriguing.
\end{comment}

We choose to focus on the monoid $\nxxz$ from the onset because all \kmsb states of $\tnxxn$ factor through a surjective homomorphism to $\tnxxz$ (cf. \remref{rem:bound-quot}). The C*-algebra $\tnxxz$ is generated by a unitary $U$ and isometries $V_a$, $a\in \nx$ acting on $\ell^2(\nxxz)$ by
$$U\ve_{(b,n)} = \ve_{(b, n+b)},\qquad V_a \ve_{(b,n)} = \ve_{(ab,n)}.$$
We'll see  in \proref{pro:presentationtnxxz} that  the elements of the form $V_a U^m V_b^*$ span a dense *-subalgebra, so the dynamics and the \kmsb states are determined by their values on these elements.

Our formulas for the evaluation of \kmsb states are expressed in terms of elementary functions from number theory.
Recall that the Euler totient function $\euler$ counts the numbers between $1$ and a given positive integer $n$ that are relatively prime to $n$; equivalently, $\euler(n)$ is the order of the group $(\Z/n\Z)^*$ of invertible elements in the ring $\Z/n\Z$.   In terms of the prime factors of $n$,
$$\euler(n)  = n\prod_{p|n} (1-p^{-1}).$$
It will be convenient for us to introduce a {\em generalized totient function} $\euler_\beta:\nx\rightarrow \R$ that includes an additional inverse temperature parameter $\beta \geq 0$ and is given by
\[
\euler_\beta(n) := n^\beta \prod_{p|n} (1-p^{-\beta}).
\]
In particular, $\euler_1$ is Euler's function $\euler$, and $\euler_0 = \delta_1$.

We also make use of the M\"obius function $\mu: \nx \to \{-1, 0, 1\}$, which vanishes if $n$ is not square-free, and satisfies $\mu(n) = 1$ (respectively, $-1$) if  $n$ is square-free and  has an even (respectively, odd) number of distinct prime factors.

\begin{theorem}\label{thm:main}
Let $\sigma$ be the natural  dynamics on $\tnxxz$ determined by
\[
\sigma_t(V_a U^k V_b^*) = (a/b)^{it} V_a U^k V_b^* \qquad a,b \in \nx, \ k\in \Z,\  t\in \R.
\]  
Suppose $\beta \in(0, 1]$. Then
 
\begin{enumerate}
\item[\textup{(a)}] for each $n\in \nx $ there is an extremal  \kmsb state $\psi_{\beta,n}$ of type \III{1} determined by
\begin{equation}\label{eqn:finite-states}
\psi_{\beta,n} (V_a U^k V_b^*) = \delta_{a,b} a^{-\beta} \Big( \frac{n}{\gcd(n,k)}\Big)^{-\beta}\sum_{d|\frac{n}{\gcd(n,k)}}\mu\left(d\right) \frac{\euler_\beta(d)}{\euler(d)};
\end{equation}
  
 \medskip \item[\textup{(b)}]  for $n = \infty$ there is an extremal \kmsb state $\psi_{\beta,\infty}$ of type \III{} determined by
\begin{equation}\label{eqn:infinite-state}
\psi_{\beta,\infty} (V_a U^k V_b^*) =   \delta_{a,b} \delta_{k,0} a^{-\beta};
\end{equation}

\medskip \item[\textup{(c)}]  
the  simplex  $K_\beta$ of \kmsb states of $(\tnxxz,\sigma)$  is a Bauer simplex with extreme boundary 
\[
\partial_eK_\beta = \{\psi_{\beta,n} : n \in \nx\sqcup \{\infty\}\};
\]  
specifically, the map $n \mapsto \psi_{\beta,n}$ is a homeomorphism of the one-point compactification
$\nx\sqcup \{\infty\}$ onto the space $ \partial_eK_\beta $  with the weak-* topology.
\end{enumerate}
Suppose $\beta =0$. Then all the $\psi_{0,n}$ for finite $n$ coalesce into one and the system has exactly two extremal \KMS{0} states (i.e. invariant traces) $\psi_{0,1}$ and $\psi_{0,\infty}$; they are given by
\[
\psi_{0,1} (V_a U^k V_b^*) = \delta_{a,b} \quad \text{and} \quad \psi_{0,\infty} (V_a U^k V_b^*) =   \delta_{a,b} \delta_{k,0}.
\]
\end{theorem}

The proof of this  theorem will occupy most of the paper; for quick reference, the parametrization is achieved in
\proref{pro:kms-parametrized} and the classification of type is in \corref{cor:maintype}. Our methods address the question raised at the end of the introduction of \cite{NS19}  of how to classify \kmsb states of right LCM monoids in the region of divergence of the partition function and may provide useful insight for the groupoid approach.

A remarkable feature of the high-temperature phase transition is its spontaneous symmetry-breaking, which is related to the one observed for the Bost-Connes system in \cite{bos-con}. The formulas from \thmref{thm:main} are valid for all non-negative $\beta$, which indicates that the \kmsb states for $\beta \leq 1$ are linked by analytic continuation to \kmsb states for $\beta >1$. As the temperature decreases, the von Neumann type of the states from \eqref{eqn:finite-states} changes from factor of type \III{1} (for $\beta \leq 1$) to a uniform superposition 
\begin{equation}\label{eqn:subcrit-superposition}
\psi_{\beta,n} = \frac{1}{\euler(n)}\sum_{\xi\in Z_n^*} \phi_{\beta,\xi}
\end{equation}
of the type ${\rm I}_\infty$ factor states from \cite{aHLR21} corresponding to primitive $n^\mathrm{th}$ roots of unity (for $\beta>1$).

The symmetries of the system are expressed by an action of $\nx$ on $\tnxxz$ by injective endomorphisms $\kappa_q$  given by $\kappa_q(V_aU^kV_b^*) = V_aU^{qk}V_b^*$ for $q\in\nx$. These commute with the dynamics so they are symmetries in the sense of \cite{CM2006}; they resemble the Frobenius endomorphisms in finite characteristic. At the level of extremal \kmsb states, the endomorphism $\kappa_q$ acts as a {\em lowering operator} on the $\psi_{\beta,n}$ for finite $n$, effectively dividing $n$ by $\gcd(n,q)$. In particular, when $n$ and $q$ are relatively prime, $\psi_{\beta,n}$ is fixed by $\kappa_q$ and moreover, the GNS representation determines a quotient of $\tnxxz$ on which $\kappa_q$ becomes an automorphism. Symmetry is broken at low temperature because the $\kappa_q$ permute the factor states in \eqref{eqn:subcrit-superposition}. Specifically, the transformation is $\phi_{\beta,\xi}\circ \kappa_q = \phi_{\beta,\xi^q}$, resembling Artin's reciprocity law for the cyclotomic extension $\Q(\sqrt[n]{1})/\Q$. In \secref{sec:bc-symm}, we realize the GNS quotient of $\tnxxz$ as the fixed-point subalgebra of the Bost-Connes algebra for the symmetries $\Gal(\Q^{\rm cycl}/\Q(\sqrt[n]{1}))$, establishing a link between KMS states of our system and class field theory of $\Q$.

The values of extremal \kmsb states given
in \eqref{eqn:finite-states} are expressed in terms of basic arithmetic functions, cf. \cite[Remark 26]{bos-con}. These expressions are quite efficient but do not provide by themselves much insight on the underlying structure or  method of proof. To shed some light on this, we recall  that 
 by \cite[Proposition 7.2]{aHLR21} a state $\psi$ of $\tnxxz$ satisfies the \kmsb condition if and only if
 \begin{equation}\label{eqn:KMSfromrestriction}
 \psi(V_a U^k V_b^*) = \delta_{a,b} a^{-\beta} \psi(U^k), \qquad k\in \N, \  a,b\in \nx.
 \end{equation}
Hence, every \kmsb state $\psi$ is completely determined by its restriction to $C^*(U) \cong C(\T)$, or rather by the probability measure on $\T$ representing this restriction through the Riesz-Markov-Kakutani theorem. Thus, the extremal \kmsb states from  \thmref{thm:main} can also be characterized using probability measures on $\T$. The trouble is that not all such probability measures extend to states of $\tnxxz$ via \eqref{eqn:KMSfromrestriction};  the issue here is positivity of the extension, which depends on whether a given measure $\nu$ satisfies 
\[
- \sum_{1\neq d|n}\mu(d) d^{-\beta}  \int_\T f(z^d) \,d\nu(z) \leq  \int_\T f(z) \,d\nu(z)
 \qquad \forall f \in C(\T)_+,  \ \forall n\in \nx;
 \]
that is, on whether $\nu$ is {\em $\beta$-subconformal} in the sense of \defref{def:subconformal} below, cf. \cite{ALN20,lacaN}. This effectively reduces the problem of finding the \kmsb states of $\tnxxz$ to that of finding all the probability measures on $\T$ that are $\beta$-subconformal  for the transformations $z \mapsto z^d$ for $d\in \nx$.
Thus, our strategy to prove  \thmref{thm:main} is to first obtain  a characterization of extremal $\beta$-subconformal probability measures on $\T$.
This is summarized in the following theorem.
\begin{theorem} \label{thm:affinehomeom}\ 

\begin{enumerate}
\item For each state $\psi$ of $\tnxxz$ let $\nu_\psi$ be the probability measure representing the restriction of $\psi$ to $C^*(U) \cong C(\T)$.  The mapping $\psi \mapsto \nu_\psi$ is an affine weak-* homeomorphism of the $\sigma$-\kmsb states onto the $\beta$-subconformal probability measures on $\T$. 
%\tcb{\item Consider including the case $\beta\in (1,\infty]$ here and link it with the explanation in red below. Watch out for the change of item number}
\item For $\beta \in (0,1]$, the extremal $\beta$-subconformal probability measures are parametrized by $\nx \sqcup\{\infty\}$ and are given as follows.
For each $n\in \nx$ the atomic probability measure $\nu_{\beta,n}$ on $\T$ is given by
 \[
 \nu_{\beta,n}(\{z\}) :=  \begin{cases}n^{-\beta}\frac{\euler_\beta(\ord(z))}{\euler(\ord(z))} & \text{ if  } z^n=1,\\
 0 &
 \text{ otherwise;}
 \end{cases}
 \]
 and  $\nu_{\beta,\infty}$   normalized Lebesgue measure on $\T$.
 Moreover, the mapping $\psi_{\beta,n}  \mapsto  \nu_{\beta,n} $ is a weak*-homeomorphism of the extremal \kmsb states onto  $\{\nu_{\beta,n} : n = 1, 2, \ldots \}\cup \{\nu_{\beta,\infty}\}$. 
 \end{enumerate}
 \end{theorem}
Part (1) of the theorem is proved in \thmref{thm:KMScharact}. The case of  atomic measures in part (2) is proved in \thmref{thm:atomic} and uniqueness of the nonatomic conformal measure is obtained in \thmref{thm:leb-extreme}.

%\tcr{In order to verify that part (2) of \thmref{thm:affinehomeom} actually gives the extremal \kmsb states listed in \thmref{thm:main}, we notice that the Fourier coefficients of the probability measure representing the restriction of a \kmsb state to $C^*(U)$ are obtained setting $a=b=1$ in \eqref{eqn:finite-states} and \eqref{eqn:infinite-state}. The latter shows that Lebesgue measure $\lambda$ corresponds to $\psi_{\beta,\infty}$; once the Fourier coefficients of the measures $\nu_{\beta,n}$ are computed in the proof of \proref{pro:kms-parametrized} it follows that they  correspond to $\psi_{\beta,n}$.}

\begin{comment}
The formulas from \thmref{thm:main} are valid for all non-negative $\beta$. %\tcb{this follows from the observation that subconformality `improves' as $\beta$ increases}, see \cite[Corollary 9.5]{ALN20} (the case of product systems over right-angled Artin monoids suffices for our purposes).
This persistence indicates that the \kmsb states for $\beta \leq 1$ are linked by analytic continuation to \kmsb states for $\beta >1$. %but \tcb{only to those whose measures are evenly supported on primitive roots of unity}. 
As the temperature decreases, the von Neumann type of the \kmsb states from \eqref{eqn:finite-states} changes from factor of type \III{1} (for $\beta \leq 1$) to a uniform superposition 
$$\psi_{\beta,n} = \frac{1}{\euler(n)}\sum_{\xi\in Z_n^*} \phi_{\beta,\xi}$$
of the type ${\rm I}_\infty$ factor states from \cite{aHLR21} corresponding to primitive $n^\mathrm{th}$ roots of unity (for $\beta>1$).
\end{comment}

Next we describe the main contents section by section, highlighting the role of each section in the proof of the main results.
In \secref{sec:toepsyst} we give a presentation of $\tnxxz$ and discuss the basics of \kmsb states for the natural dynamics. In \secref{sec:specdiag} we dive into the structure of $\tnxxz$. We describe the fixed point algebra $\mfd$ of the gauge action of $\qd$ and in \proref{proj-limit-present} we realize its spectrum as a projective limit over $a\in \nx$ of copies of the unit circle indexed by the divisors of $a$. This result is instrumental for the passage from subconformal measures on $\T$ to \kmsb states. 

\secref{sec:subconformal} is about $\beta$-subconformal measures on $\T$ with respect to the semigroup of `wrap-around' transformations $z\mapsto z^n$. We work through the projective limit realization of the spectrum of the diagonal subalgebra $\mfd$ to show that $\beta$-subconformal probability measures on $\T$
extend to states \kmsb of $\tnxxz$. The main result of this section, \thmref{thm:KMScharact}, establishes that this correspondence is an affine isomorphism of simplices. We also observe that atomic and nonatomic measures can be studied separately. In \secref{sec:atomic} we focus on atomic $\beta$-subconformal measures,  giving a complete description in \thmref{thm:atomic}. \secref{app:appendixB} is purely about obtaining a number theoretic estimate for partial sums over $\nx$, \proref{prop:summation}, which is crucial to analyze the nonatomic case. The main result of \secref{sec:nonatomic}, \thmref{thm:leb-extreme}, is that the only nonatomic $\beta$-subconformal probability measure on $\T$ is normalized Lebesgue measure. The argument follows the strategy used by Neshveyev in \cite{nes} to prove uniqueness of the \kmsb state of the Bost-Connes system on the critical interval. This relies on a multiplicative version of Wiener's lemma, \proref{pro:leb-extreme} obtained through the estimate from \secref{app:appendixB}. In \secref{sec:proofofmainthm} we collect the results of the preceding sections and prove \thmref{thm:main} without the type assertion.

In preparation for the type classification, in \secref{sec:eqv-quot} we introduce a sequence of  equivariant quotients of $\tnxxz$.  We realize them in a natural way as Toeplitz algebras of monoids of affine transformations associated to arithmetic modulo $n$, and characterize their equilibrium states in \thmref{thm:finite-kms}. In \secref{sec:bc-symm} we show that these modular quotients have natural homomorphisms to the Bost-Connes  C*-algebra $\cq$, \proref{pro:bc-hom}. This allows us to import the type classification from the known results for the Bost--Connes system, which we do in \corref{cor:maintype}. In \secref{sec:nxxqz} we observe that these modular quotients can also be assembled together to form  another natural Toeplitz C*-algebra, namely $\tnxxqz$, for which we give a presentation. The main result here is the associated phase transition of $\tnxxqz$ described in \thmref{thm:limit-kms} in terms of subgroups of $\Q/\Z$.

\section{The Toeplitz system of $\nxxz$}\label{sec:toepsyst}

%The right Toeplitz algebra $\mathcal{T}_{rt}(\zxnx)$ of the monoid $\zxnx$ is the C*-algebra generated by the right-regular (anti?)representation of the $ax+b$ monoid $\Z\rtimes \nx$ by (co?) isometries. The literature on isometric semigroup representations places an emphasis on left-regular representations, so we follow the strategy of \cite{aHLR21} and work with the presentation of $\mathcal{T}_{rt}(\zxnx)$ as the C*-algebra generated by the left-regular representation of the opposite monoid $\nxxz$. See the discussion in the Introduction of \cite{aHLR21}. 
%Strictly speaking, \cite{aHLR21} begins with a discussion of the monoid $\nxxn$, but since here we are mostly interested in KMS states, and as shown in \cite[Proposition 7.1]{aHLR21}, these factor through the additive boundary quotient, we focus on the semigroup $\nxxz$ from the onset.
%
Let  $\tnxxz$ be the C*-subalgebra  generated by the operators $T_{(a,m)}$ on  $\ell^2(\nxxz)$ defined on the canonical orthonormal basis by
\[
T_{(a,m)}\ve_{(b,n)} = \ve_{(ab, bm+n)}\qquad (a,m),\ (b,n)\in\nxxz.
\]
Then $T_{(a,m)} = V_a U^m$, where $V_a = T_{(a,0)}$ is an isometry for each $a$ and $U = T_{(1,1)}$ is a unitary. 
Next we give a presentation of $\tnxxz$ and use it to show that $\tnxxz$ is the additive boundary quotient of the C*-algebra  $\tnxxn$ studied in \cite{aHLR21}.

\begin{prop} \label{pro:presentationtnxxz} The generating elements $\{V_a: a\in \nx\}$ and $U$ of $\tnxxz$ satisfy %$V_a^*V_a =1 = UU^*=U^*U$ and 
\begin{enumerate}
\item[\textup{(AB0)}] $V_a^*V_a =1 =U^*U= UU^*$ 
\item[\textup{(AB1)}] $UV_a = V_a U^a$;%\\[-5pt]
\smallskip\item[\textup{(AB2)}] $V_aV_b = V_{ab}$;%\\[-5pt]
\smallskip\item[\textup{(AB3)}] $V_a^* V_b = V_bV_a^*$ when $\gcd(a,b) = 1$.%\\[-5pt]
%\smallskip\item $U^* V_a = V_a U^{*a}$.%\\[-5pt]
\end{enumerate}
Moreover, the relations \textup{(AB0)--(AB3)} constitute a presentation of $\tnxxz$ and imply 
\begin{enumerate}
\smallskip\item[\textup{(AB4)}]$U^* V_a = V_a U^{*a}$.%\\[-5pt]
\end{enumerate}
The C*-algebra  $\tnxxz$ is  canonically isomorphic to the additive boundary quotient $\partial_{add}\tnxxn$ and
 \[
\tnxxz = \clsp \{V_a U^m V_b^*\ :\ a,b\in\nx,\ m\in \Z\}. 
\]
 \end{prop}
\begin{proof} 

That the $V_a$ are isometries and $U$ is a unitary is obvious. That they satisfy relations (AB1)--(AB3) was verified in \cite[Example 3.9]{aHLR21}, while (4) is obtained on
multiplying (AB1)  by $U^*$ on the left and by $U^{*a}$ on the right, see the proof of  \cite[Proposition 3.8]{aHLR21}. 

Let  $C^*(u,v_a: a\in \nx)$ be the universal C*-algebra generated by isometries $\{ v_a : a\in\nx\}$ and a unitary $u$ satisfying the lowercase-analogues of the relations (AB1)--(AB3). By the preceding considerations, there is a canonical surjective homomorphism $C^*(u,v_a: a\in\nx)\rightarrow \tnxxz$, which we will show is an isomorphism.

Recall from \cite{aHLR21} that the monoid $ \nxxz $ is right LCM; indeed, the smallest common upper bounds of $(a,m)$ and  $(b,n) $ are the elements $(\lcm(a,b),k)$ for  $k\in \Z$ (so we may take, e.g. $(\lcm(a,b),0)$). Since $\nxxz$ embeds in  $\qxxq$, we have that $\tnxxz$ is universal for Nica covariant representations of $\nxxz$ by \cite[Corollary 5.6.45]{CELY}. 

The elements $w_{(a,n)} = v_a u^n$ form an isometric representation of $\nxxz$ in $C^*(u, v_a:a\in\nx)$ by (AB0)--(AB2). For $a,b\in \nx$, let $a' = a/\gcd(a,b)$ and $b' = b/\gcd(a,b)$. Then (AB3) implies
$$w_{(a,m)} w_{(a,m)}^* w_{(b,n)} w_{(b,n)}^* = v_a v_a^* v_b v_b^* = v_{ab'} v_{b'a}^* = w_{(\lcm(a,b),0)}w_{(\lcm(a,b),0)}.$$
This shows that $w$ is Nica covariant. Therefore, there is a canoncial surjective homomorphism $\tnxxz\rightarrow C^*(u,v_a: a\in\nx)$, which is the inverse to $C^*(u,v_a: a\in\nx)\rightarrow \tnxxz$.

\begin{comment}
Suppose that  $W$ is an isometric representation of $\nxxz$; then \tcb{the elements 
$V_a = W_{(a,0)}$ and $U = W_{(1,1)}$} satisfy relations (AB0)--(AB2). If W is also Nica covariant  and $\gcd(a,b) = 1$, then 
\[
V_a V_a^* V_b V_b^* = W_{((\lcm(a,b),0)}W_{((\lcm(a,b),0)}^* = W_{(ab,0)}W_{(ab,0)}^* = V_{ab}V_{ab}^* ,
\]
from which relation (AB3) is derived  multiplying both sides on the left by $V_a^*$ and on the right by $V_b$. 

\tcb{Let $v_a$ and $u$ be the generators of $C^*(u,v_a: a\in \nx) \to \tnxxz$ and define  $W{(a,m)}:= v_a u^m$.} 
Then $W$  an isometric representation of $\nxxz$ because of relations (AB0)--(AB2). Take $(a,m)$ and $(c,n)$ in $\nxxz$ and let $a' = a/\gcd(a,c)$ and $c' =c/\gcd(a,c)$, so that  $ac' = ca' = \lcm(a,c)$ and $\gcd(a',c') =1$. Then (AB3) implies 
\[
W_{(a,m)}W_{(a,m)}^* W_{(c,n)}W_{(c,n)}^*= v_a  v_a^* v_c v_c^* = v_{ac'} v_{ca'}^* =W_{(\lcm(a,c),0)}W_{(\lcm(a,c),0)}^* .
\]
This shows that $W$ is Nica covariant. Thus $C^*(u,v_a: a\in \nx)$ is canonically isomorphic to the universal C*-algebra for Nica covariant representations of $\nxxz$.

Since $\nxxz$ embeds in  $\qxxq$,  the canonical map $C^*(u,v_a: a\in \nx) \to \tnxxz$
is an isomorphism by  \cite[Corollary 5.6.45]{CELY}, and 
the isomorphism $\tnxxz \cong \partial_{add}\tnxxn$ now follows from \cite[Proposition 3.8]{aHLR21}. 
\end{comment}

The collection $\{V_a U^m V_b^*\ :\ a,b\in\nx,\ m\in \Z\}$ is obviously closed under taking adjoints, and 
 \begin{equation}\label{eqn:monomial-product}
(V_a U^m V_b^* )(V_c U^n V_d^*) = V_{ac'} U^{mc' + nb'} V_{b' d}^*,
\end{equation}
where $c' = \frac{c}{\gcd(b,c)}$ and $b' = \frac{b}{\gcd(b,c)}$, %or equivalently $bc' = b'c = \lcm(b,c)$. 
so this collection is also closed under multiplication. Hence its linear span is a self-adjoint subalgebra of $\tnxxz$, which is dense because  it contains  the generating elements $V_a$ for $a\in \nx$ and $U$.
\end{proof}

\begin{remark}\label{rem:bound-quot}
The presentation of $\tnxxz$ in \proref{pro:presentationtnxxz} agrees with that of $\partial_{\mathrm{add}} \tnxxn$ in \cite[Proposition 3.8]{aHLR21}, which implies that these C*-algebras are isomorphic. Moreover, by \cite[Proposition 7.1]{aHLR21}, the KMS states of $\tnxxn$ factor through the additive boundary quotient. This gives a 1-to-1 correspondence between KMS states of $\tnxxn$ and KMS states of $\tnxxz$.
\end{remark}

\begin{prop}\label{pro:condexptheta}
There exists a strongly continuous (gauge) action $\theta$ of the compact group $\wh{\qx}$ by  automorphisms of $\tnxxz$ such that 
\[
\theta\chi (V_a U^m V_b^*) =  \chi(a/b)V_a U^m V_b^*.
\]
The fixed point algebra $\mfd := \tnxxz^{\theta}$   is a commutative unital C*-algebra and there is a faithful conditional expectation 
$E: \tnxxz \longrightarrow \mfd $ determined by
\[
E( V_a U^m V_b^*) = \int_{ \wh{\qx}} \theta_\chi (V_a U^m V_b^*) d\chi = \delta_{a,b} V_a U^m V_b^*,
\]
with range $E(\tnxxz) = \mfd = \clsp\{V_a U^m V_a^*: a\in\nx,\ m\in\Z\}$.
\begin{proof}
The proof is by a standard argument and is almost entirely analogous to that of \cite[Proposition 8.2]{aHLR21},  the only difference being the computation of the product at the end. Here the additive generator $U$ is a unitary operator and thus, once we verify that 
\[
E( V_a U^m V_b^*) = \int_{ \wh{\qx}} \theta_\chi (V_a U^m V_b^*) d\chi = \int_{ \wh{\qx}} \chi(a/b) (V_a U^m V_b^*)d\chi  = \delta_{a,b} V_a U^m V_b^*,
\]
where $\delta_{a,b}$ is the Kronecker delta function, we may conclude that 
$ \mfd = \clsp\{V_a U^m V_a^*: a\in\nx,\ m\in\Z\}$. Setting $a=b$ and $d=c$ in \eqref{eqn:monomial-product}, we get the product
\begin{equation}\label{eqn:D-prod}
(V_b U^m V_b^*)(V_c U^n V_c^*) = V_{\lcm(b,c)} U^{mc' + nb'} V_{\lcm(b,c)}^* = V_{\lcm(b,c)} U^{\lcm(b,c) (\frac{m}{b}+ \frac{n}{c})} V_{\lcm(b,c)}^*,
\end{equation}
which shows that $\mfd$ is commutative.
\end{proof}
\end{prop}
%\tcr{No need to say this here} In the next section we will realize $\Spec \mfd$ as a projective limit of circles, by virtue of the Fourier transform applied to $C^*(U)$ and a lattice grading of $\mfd$ by $\nx$, in the sense of \cite[????]{ALN20}.

We are interested in the C*-dynamical system $(\tnxxz, \sigma)$ in which $\sigma $ is the dynamics  determined by
\[
\sigma_t(V_a U^m V_b^*) = \Big(\frac{a}{b}\Big)^{it} \,V_a U^m V_b^*.
\]
The study of equilibrium on $(\tnxxz, \sigma)$ was initiated in \cite{aHLR21}, where it was shown that \kmsb states of $(\tnxxn, \sigma)$ factor through the additive boundary quotient. We briefly recall next the basic definitions and the key results needed in our analysis.

Whenever $\sigma$ is a time evolution, or dynamics,  on a C*-algebra $A$ (this means that $\sigma$ is a strongly continuous $\R$-action by automorphisms of $A$),  there is a dense *-subalgebra $A^\infty$ of {\em analytic elements} of $A$, consisting of elements $x\in A$ for which the function $F_x(t)=\sigma_t(x)$ for $t\in\R$  can be analytically continued to an entire function on $\C$. For $\beta\in[0,\infty)$, a state $\phi$ on $A$ satisfies the $\sigma$-\KMS{\beta} condition (or simply the \kmsb condition, when $\sigma$ is clear) if 
$$\phi(xy) = \phi(y \sigma_{i\beta} (x))\qquad \text{ for $x \in A^\infty$ and $y \in A$}$$
 in fact, because of bilinearity and continuity, it suffices to show that equality holds for $x$ and $y$ in a subset of $A^\infty$ whose linear span is $\sigma$-invariant and dense in $A$ \cite[Proposition 8.12.3]{ped}.
Every $\sigma$-\kmsb state is also $\sigma$-invariant (for $\beta = 0$ this is part of the definition, so that $\sigma$-KMS$_0$ states are $\sigma$-invariant traces). The set $K_\beta$ of $\sigma$-\kmsb states of $A$, endowed with the weak* topology, is a Choquet simplex, and hence is affinely isomorphic to the simplex of probability measures on the set $\partial K_\beta$ of its extreme points. 
%This isomorphism is realized by the map $$\mu\mapsto \left(x\mapsto \int_{\Omega_\beta} \phi(x) d\mu(\phi)\right).$$
We refer to Chapter 5 of \cite{bra-rob} and to Chapter 8 of \cite{ped} for  further details and background.

When we consider our system $(\tnxxz,\sigma)$ it is easy to see that the monomials $V_a U^m V_b^*$ are analytic for $\sigma$ because $\sigma_z(V_a U^m V_b^*) = (a/b)^{iz}V_a U^m V_b^*$. Moreover, the dynamics $\sigma$ is obtained by composing the continuous one-parameter subgroup of characters $\chi_t(r) = r^{it}$  of $\qx$ with the gauge action $\theta$, and  
the fixed-point subalgebra of $\sigma$ agrees with the fixed-point subalgebra of $\theta$. Hence, the $\sigma$-invariant states on $\tnxxz$, in particular the \kmsb states, are induced through the conditional expectation $E$ from traces on $\mfd$, or, equivalently, from measures on $X = \Spec \mfd$.

For later reference, we record the following result analogous to \cite[Proposition 8.3]{aHLR21}.
\begin{prop}\label{pro:semigroupaction}
 There exists an action $\alpha$ of $\nx$ by injective endomorphisms 
 $\alpha_a:\mfd\rightarrow\mfd$ defined by $\alpha_a(x) = V_a x V_a^*$  for each $a\in\nx$.   Each $\alpha_a$ has a  left-inverse given by $\gamma_a(x) = V_a^* x V_a$.  
Moreover, there is a semigroup crossed product decomposition 
 \begin{equation}\label{eqn:semiXprod}
 \tnxxz\cong \nx\ltimes_{\alpha} \mfd.
 \end{equation}
 \end{prop}
Following \cite[Section 4]{aHLR21} we use the `backwards'  notation for the above semigroup crossed product because it is more compatible with the semigroup operation in $\nxxz$. 
 
The action of $\nx$ on $\mfd$ respects the lattice structure, see \cite[Definition 3]{diri}, thus  \cite[Theorem 12]{diri} implies that the map $\tau\mapsto \tau\circ E$ is a one-to-one correspondence between the tracial states $\tau$ on $\mfd$ satisfying
\begin{equation}\label{eqn:rescale}
\tau(V_a x V_a^*) = a^{-\beta} \tau(x)\qquad\forall a\in\nx,\ x\in \mfd
\end{equation}
and the \kmsb states of $(\tnxxz,\sigma)$ for $\beta \in (0,\infty)$. %\tcr{\st{For a similar argument see the proof of }\cite[Theorem 8.1]{aHLR21} \st{which is about $\tnxxn$, but from which the present case can be derived by taking quotients.}}

For $\beta>1$, the tracial states satisfying \eqref{eqn:rescale}  were explicitly computed in \cite[Theorem 8.1]{aHLR21} using \cite[Theorem 20]{diri}; they are in one-to-one correspondence with the probability measures $\eta$ on $\T$ via the formula
$$\tau_{\eta,\beta}(V_a U^n V_a^*) = \frac{a^{-\beta}}{\zeta(\beta)} \sum_{c=1}^\infty c^{-\beta} \int_\T z^{nc} d\eta,$$
and the corresponding \kmsb state, obtained through the conditional expectation $E$, is given by
\begin{equation}\label{eqn:low-temp}
\phi_{\eta,\beta}(V_a U^n V_b^*) = \delta_{a,b} \frac{a^{-\beta}}{\zeta(\beta)} \sum_{c=1}^\infty c^{-\beta} \int_\T z^{nc} d\eta.
\end{equation}
\begin{remark}Obviously these formulas break down for $\beta\leq 1$. Partly because of this, the question of equilibrium for  $\beta\in (0,1]$ was left open in \cite{aHLR21}, except for three KMS$_1$ states exhibited as limits for $\beta\to 1^+$ in \cite[Section 9]{aHLR21}.
Those three states are recovered by the parametrization \eqref{eqn:finite-states} of \thmref{thm:main}. Indeed,  setting $\beta=1$ in \eqref{eqn:infinite-state} recovers the KMS$_1$ state obtained in
\cite[Example 9.1]{aHLR21} from Lebesgue measure on $\T$. Similarly, an easy computation shows that our $\psi_{1,1} $ is the state obtained in \cite[Example 9.2]{aHLR21} from the point mass at $1\in \T$, and a slightly more involved computation shows that 
\[\psi_{\beta,2}(V_a U^k V_b^*) = \begin{cases} \delta_{a,b} a^{-\beta}  & \text{if }  k \text { is  even} \\
 \delta_{a,b} a^{-\beta} (2^{1-\beta} - 1)&  \text{if }  k \text{ is odd,}
\end{cases}
\]
so that our $\psi_{1,2}$ is the state from \cite[Example 9.3]{aHLR21}. 
It is also clear from the parametrization that all these states  `persist' as the inverse temperature drops below  critical.
\end{remark}

\section{The diagonal and its spectrum}\label{sec:specdiag}
In this section we  provide a detailed description of the fixed point algebra $\mfd$ of the gauge action and its spectrum. We begin by outlining a unitarily equivalent copy of $\tnxxz$ obtained via the Fourier transform on the second coordinate of $\nxxz$. To be precise, we 
let  $\zed: z \mapsto z$ be the inclusion $\T\subseteq \C$ (viewed as a complex-valued function on $\T$), and we normalize Haar measure on $\T$ so that  the collection $\{\zed^k: k\in \Z\}$ of characters  is an orthonormal basis of  $L^2(\T)$. Then there is a unitary transformation 
 \[\F :   \ell^2(\nxxz) \to  \ell^2(\nx)\otimes L^2(\T) \qquad \quad \F (\ve_{(b,k)}) = \delta_b \otimes \zed^k, \quad (b,k) \in \nxxz.\]
 When we conjugate the  generators $U=T_{(1,1)}$ and $V_a = T_{(a,0)}$ of $\tnxxz$  by $\F$ we get operators
 $\Ad_\F(U): =  \F U \F\inv$ and $\Ad_\F(V_a): = \F V_a \F\inv $  on $\ell^2(\nx)\otimes L^2(\T)$, and when we compute these on the standard orthonormal basis $\{\delta_b \otimes \zed^m: b\in \nx, \ m\in \Z\}$ of $ \ell^2(\nx)\otimes L^2(\T)$ we get
\begin{equation}\label{eqn:adU}
\Ad_\F(U)(\delta_b\otimes\zed^k) =  \F U \F\inv (\delta_b\otimes\zed^k) =  \F U \ve_{(b,k)} =  \F   \ve_{(b,b+k)} = \delta_b\otimes \zed^b\zed^k
\end{equation}
and 
\begin{equation}\label{eqn:adV_a}
 \Ad_\F(V_a) (\delta_b\otimes\zed^k) = \F V_a \F\inv (\delta_b\otimes\zed^k) =  \F V_a \ve_{(b,k)} =  \F \ve_{(ab,k)} =  \delta_{ab}\otimes\zed^k.
\end{equation}
\begin{lemma}\label{lem:FourierTransform} 
For each $a\in \nx$ define a map\[
\omega_a:\T\rightarrow \T, \qquad  z\mapsto z^a.
\] 
 wrapping the circle $a$-times around itself. Denote 
by $\pi:C(\T) \to \mathcal B(\ell^2(\nx)\otimes L^2(\T))$ the representation of $C(\T)$ generated by the unitary $u := \Ad_\F(U)$ and let  $v_a := \Ad_\F(V_a )$.
Then 
\[ 
\pi(f) (\delta_b\otimes g) := \delta_b\otimes (f\circ \omega_b) g \qquad b\in \nx \ \ f,g \in C(\T),
\]
and the image of the fixed point algebra $\mfd$ under the isomorphism $\Ad_\F: \tnxxz \cong C^*(\pi,v ) $  
is   
\begin{equation}\label{oiyasdfasqw}
\Ad_\F(\mfd) = {\clsp}\{v_a \pi(f) v_a^*: f\in C(\T), a\in\nx\}.
\end{equation}

\end{lemma}

\begin{proof}
It is easy to show using \eqref{eqn:adU} that the first assertion holds for $f = \zed^m$ and $g=\zed^k$, and the general case follows from this
because the characters $\{\zed^m: m\in \Z\}$ span a dense subalgebra of $C(\T)$.
Since we already know that $\mfd = \clsp\{V_a U^m V_a^*: a\in\nx,\ m\in\Z\}$, the second assertion is also a direct consequence of this.
\end{proof}

To simplify the notation from now on we will write $V_a f V_a^*$ for  the element of $\mfd$ corresponding to $v_a\pi(f)v_a^*$, so that, e.g.   $V_a \zed^m V_a^* =V_a U^m V_a^* $. For each fixed $a\in \nx$ we also define
\[
\mfd_a := \clsp\{V_d f V_d^*: d|a,\ f\in C(\T)\}.
\]
This is a closed subspace which is closed under adjoints; it is also closed under multiplication because \eqref{eqn:D-prod} implies that
\begin{equation}\label{eqn:productxy}
V_c f V_c^* \, V_d gV_d^* = V_{c\vee d} (f\circ \omega_{d'})(g\circ \omega_{c'}) V_{c\vee d}^*, \qquad c,d \in \nx, \ g \in C(\T),
\end{equation}
where we have written $c\vee d$ for $\lcm(c,d)$, with $c' = \frac{c}{\gcd(b,c)}$ and $b' = \frac{b}{\gcd(b,c)}$ to streamline the notation. Hence $\mfd_a$ is  a C*-subalgebra of $\mfd$, 
and the inclusions $\iota_{a,b}:\mfd_a\hookrightarrow \mfd_b$ for $a|b$ give an injective system $(\mfd_a, \iota_{a,b})_{a\in \nx}$ of C*-algebras  whose union is dense in $\mfd$ by \lemref{lem:FourierTransform}, making  $\mfd$ the direct limit of the system. 

\begin{lemma}\label{lem:complete-projections}
For $a\in \nx$ define $e_a = \prod_{p|a} (1-V_p V_p^*)$ and for $b|a$ let $e_{a,b} = \alpha_b(e_{\frac{a}{b}}) = V_b e_{\frac{a}{b}} V_b^*$. Then $e_{a,b}$ is a projection and
\begin{enumerate}
 \smallskip\item $e_{a,b} = \sum_{d|\frac{a}{b}} \mu(d) V_{bd} V_{bd}^*$, and thus belongs to $\mfd_a$;
 \smallskip\item $\sum_{d|\frac{a}{b}} e_{a,bd} = V_b V_b^*$;
 \smallskip\item the map
 \begin{equation}\label{eqn:compression-iso}
 \gamma_{a,b}:C(\T)\rightarrow e_{a,b}\mfd_a e_{a,b},\qquad f\mapsto e_{a,b}V_b f V_b^* = \sum_{d|\frac{a}{b}} \mu(d) V_{bd} (f\circ \omega_d) V_{bd}^*
 \end{equation}
 is an isomorphism.

\end{enumerate}
\end{lemma}

\begin{proof}
For part (1), the case $b = 1$ follows from the usual inclusion-exclusion formula and $V_c V_c^* V_d V_d^* = V_{cd} V_{cd}^*$ for relatively prime divisors $c,d|a$. For more general $b$ we have
$$\alpha_b(e_{\frac{a}{b},1}) = V_b\left(\sum_{d|\frac{a}{b}} \mu(d) V_d V_d^*\right) V_b^* = \sum_{d|\frac{a}{b}} \mu(d) V_{bd} V_{bd}^*.$$%For $c,d | \frac{a}{b}$, \eqref{eqn:D-prod} implies that 
%$$V_{cb} V_{cb}^* V_{db} V_{db}^* = V_{(c\vee d) b} V_{(c\vee d) b}^*,$$
%so the usual inclusion-exclusion formula gives part (1) when $\frac{a}{b}$ is square-free. Since the M\"obius function vanishes for non-square-free integers, it follows that part (1) holds for all $b|a$. 

%Part (2) follows from part (1), as $e_b$ is a product of pairwise commuting projections.

For part (2), we have
$$\sum_{d|\frac{a}{b}} e_{a,bd} = \sum_{d|\frac{a}{b}} \sum_{c|\frac{a}{bd}} \mu(c) V_{bcd} V_{bcd}^* = \sum_{e|\frac{a}{b}} V_{be} V_{be}^* \sum_{c|e} \mu(c) = V_b V_b^*,$$
where we have made use of the substitution $e = cd$ and the classical identity $\sum_{c| e} \mu(c) = \delta_{e,1}$.

For part (3), since $e_{a,b}$ is a projection in the commutative algebra $\mfd_a$, the map $\kappa_b:\mfd_a\rightarrow e_{a,b} \mfd_a e_{a,b}$, $x\mapsto  e_{a,b} x $ is a homomorphism. By part (1) and \eqref{eqn:productxy}, composing $\kappa_b$ with the map $f\mapsto V_b f V_b^*$ yields
$$e_{a,b} V_b f V_b^* = \sum_{d|\frac{a}{b}} \mu(d)V_{bd} V_{bd}^* V_b f V_b^* = \sum_{d|\frac{a}{b}} \mu(d) V_{bd} (f\circ \omega_d) V_{bd}^*$$
which is \eqref{eqn:compression-iso}, so this is a homomorphism from $C(\T)$ to $e_{a,b} \mfd_a e_{a,b}$. 

In order to show that \eqref{eqn:compression-iso} is surjective, consider $V_c f V_c^*$ for $c|a$ and $f\in C(\T)$. By part (2) with $b = 1$, the projections $\{e_{a,d}: d|a\}$ are mutually orthogonal. Again by part (2), for $c|a$,
$$V_c f V_c^* = V_cV_c^*V_c f V_c^*  = \sum_{d|\frac{a}{c}} e_{a,cd} V_c f V_c^*;$$
hence if $c\nmid b$, then $e_{a,b} V_c f V_c^* = 0$. If $c| b$, then using part (1) and \eqref{eqn:productxy},
$$e_{a,b} V_c f V_c^* = \sum_{d|\frac{a}{b}} \mu(d) V_{bd} V_{bd}^* V_c f V_c^* = \sum_{d|\frac{a}{b}} \mu(d) V_{bd} (f\circ \omega_{bd/c}) V_{bd}^*,$$
which is the image of $f\circ \omega_{b/c}$ under \eqref{eqn:compression-iso}. Since the elements $V_c f V_c^*$ span a dense subspace of $\mfd_a$, \eqref{eqn:compression-iso} is surjective.

Lastly, we show that \eqref{eqn:compression-iso} is faithful. If $e_{a,b} V_b f V_b^* = 0$, then by part (2), it follows that
\begin{align*}
V_b f V_b^* &= \sum_{c|\frac{a}{b}, c\neq 1} e_{a,bc} V_b f V_b^* = \sum_{c|\frac{a}{b},c\neq 1} \sum_{d|\frac{a}{bc}} \mu(d) V_{bcd} V_{bcd} V_b f V_b^*\\[5pt]
&= \sum_{c|\frac{a}{b},c\neq 1} \sum_{d|\frac{a}{bc}} \mu(d) V_{bcd} (f\circ\omega_{cd}) V_{bcd}^*.
\end{align*}
The condition $c\neq 1$ (hence $bcd\nmid b$) implies that the vector $\delta_b\otimes 1_{\T}$ belongs to the kernel of $\Ad_\F(V_{bcd}^*)$, where $\Ad_\F$ is the isomorphism of \lemref{lem:FourierTransform}; moreover, $\Ad_{\F}(V_b f V_b^*)(\delta_b \otimes 1_{\T}) = \delta_b\otimes f$, so $f = 0$.
\end{proof}

\begin{cor}\label{cor:system-spectra}
For each $a\in \nx $ let $\Delta_a := \{b\in\nx: b|a\}$  be the set of divisors of $a$ and define a space 
\[
 X_a := \T\times \Delta_a.
 \]
For each $f\in C(X_a)$ and  $b\in \Delta_a$ let 
$f|_b(z) = f(z,b)$, $z\in \T$. Then the map 
\[
\Gamma_a : f\mapsto \sum_{b|a} e_{a,b} V_b f|_bV_b^* = \sum_{b|a}\sum_{d|\frac{a}{b}} \mu(d) V_{bd}(f|_b\circ \omega_d)V_{bd}^*
\]
%(we use $\gamma_n$ to denote which copy of $C(\T)$ each function belongs to)
is an isomorphism of C*-algebras $\Gamma_a: C(X_a) \stackrel{\cong}{\longrightarrow} \mfd_a$. The inverse is determined by the formula
\begin{equation}\label{eqn:gamma-inv}
\Gamma_a^{-1}( V_b f V_b^*)(z,d) = \left\{\begin{array}{ll} f\circ\omega_{\frac{d}{b}}(z)&\text{if $b|d$,}\\
0&\text{otherwise.}
\end{array}\right.
\end{equation}
\end{cor}

\begin{proof}
The algebra $C(X_a)$ is naturally identified with $\bigoplus_{b|a} C(\T)$, where a function $f\in C(X_a)$ corresponds to the tuple $(f|_b)_{b|a}$. Under this identification, $\Gamma_a$ becomes
$$\Gamma_a = \bigoplus_{b|a}\gamma_{a,b}:\bigoplus_{b|a} C(\T)\rightarrow \mfd_a,$$
where $\gamma_{a,b}:C(\T)\rightarrow \mfd_a$ is the isomorphism onto $e_{a,b}\mfd_a e_{a,b}$ from \lemref{lem:complete-projections} (3). Since the projections $\{e_{a,b}:b|a\}$ are mutually orthogonal and sum to the identity, $\Gamma_a$ is an isomorphism onto $\mfd_a$.

For $f\in C(\T)$ and $b\in \Delta_a$, let $\wt{f}\in C(X)$ denote the function in \eqref{eqn:gamma-inv}. Then, by the M\"obius inversion formula, we have
\begin{align*}
\Gamma_a (\wt{f}) &= \sum_{c|a} \sum_{d|\frac{a}{c}} \mu(d) V_{cd} (\wt{f}|_c\circ \omega_d) V_{cd}^* \\[5pt]
&= \sum_{c|\frac{a}{b}} \sum_{d|\frac{a}{c}}\ \mu(d) V_{bcd} (f\circ \omega_{cd}) V_{bcd}^*\\[5pt]
&= \sum_{c|\frac{a}{b}} V_{bc} (f\circ \omega_c) V_{bc}^*\left(\sum_{d|c} \mu(d) \right) = V_b f V_b^*.
\end{align*}
Since $\Gamma_a$ is an isomorphism, we conclude that $\wt{f} = \Gamma^{-1}_a(V_bf V_b^*)$. 
\end{proof}

\begin{proposition}\label{proj-limit-present}
For each $a\in \nx$ and for $a|b$, define $\Psi_{a,b} :X_b \to X_a$  by
\[
\Psi_{a,b} (z,d) := (z^{d/\gcd(a,d)}, \gcd(a,d)) \qquad \text{ for each } (z,d) \in X_b.
\]
Then $\Psi = (X_a,\Psi_{a,b})_{a\in \nx}$ is a projective system that is topologically conjugate to the projective system $\iota^* =(\Spec\mfd_a, \iota_{a,b}^*)_{a\in \nx}$ under the 
transformations $\Gamma_a^*: \Spec\mfd_a  \to X_a$, and this gives a homeomorphism 
\[\projlim_a(X_a,\Psi_{a,})_{a\in \nx}  \cong \Spec \mfd.\]
\end{proposition}

\begin{proof}
It suffices to show that $\Gamma_a^*\circ\iota_{a,b}^*=\Psi_{a,b} \circ \Gamma_b^*$ on $\Spec \mfd_b$, or dually, that $\Gamma_b^{-1}\circ \iota_{a,b} = \Psi_{a,b}^*\circ\Gamma_a^{-1}$ on $\mfd_a$. For $f\in C(\T)$ and $c|a$, applying the first homomorphism to $V_c f V_c^*$ and evaluating at $(z,d)\in \Delta_b$ gives
$$\Gamma_b^{-1}\circ \iota_{a,b}(V_c f V_c^*)(z,d) = \Gamma_b^{-1}(V_c f V_c^*)(z,d) =
\left\{\begin{array}{ll}
f\circ \omega_{\frac{d}{c}}(z)&\text{if $c|d$,}\\
0&\text{otherwise.}
\end{array}\right.$$
Applying the second homomorphism and evaluating at $(z,d)$ gives
$$\Psi_{a,b}^*\circ \Gamma_a^{-1}(V_c f V_c^*) = \Gamma_a^{-1}(V_c f V_c^*)(z^{\frac{d}{\gcd(a,d)}}, \gcd(a,d)) = \left\{\begin{array}{ll}
f\circ \omega_{\frac{\gcd(a,d)}{c}}(z^{\frac{d}{\gcd(a,d)}})&\text{if $c|\gcd(a,d)$,}\\
0&\text{otherwise.}
\end{array}\right.$$
Since $c|a$, the conditions $c|d$ and $c|\gcd(a,d)$ are equivalent, in which case both formulas agree.
\end{proof}

\section{KMS states and subconformal measures}\label{sec:subconformal}
According to \cite[Proposition 7.2]{aHLR21}, a state $\psi$ on $\tnxxz$ satisfies the \kmsb condition for the dynamics $\sigma$ on $\tnxxz$ if and only if 
\begin{equation}\label{eqn:KMScharact}
\psi(V_a U^kV_b^*) = \delta_{a,b} a^{-\beta} \psi(U^k)\quad\text{for all $a,b\in \nx$ and $k\in\Z$}.
\end{equation} 
Thus each \kmsb state is determined by its restriction to  $C^*(U) = V_1 C(\T) V_1^* \cong C(\T)$; we denote by  $\nu_\psi$ the probability measure on $\T$ representing this restriction, so that
\begin{equation}\label{eqn:nufromvarphi}
\int_\T f d\nu_\psi = \psi(V_1 f V_1^*) \qquad f\in C(\T).
\end{equation}
The map $\psi \mapsto \nu_\psi$ of \kmsb states to probability measures is injective but, as discussed in the Introduction after \thmref{thm:main}, it is not surjective. In order to determine its range we introduce the following  condition, cf. \cite[Equation (2.1)]{ALN20}.
\begin{definition} \label{def:subconformal}
A measure $\nu$ on $\T$ is {\em $\beta$-subconformal} if  it satisfies
\begin{equation}\label{eqn:positivitycondition}
\sum_{d|n} \mu(d)d^{-\beta} \omega_{d*}(\nu)\geq 0 \qquad \forall n\in \nx,
\end{equation}
or, more explicitly,
\[
\sum_{d|n} \mu(d)d^{-\beta}  \int_\T f(z^d) \,d\nu(z) \geq 0 
 \qquad \forall f \in C(\T)_+  \text{ and } \forall n\in \nx.
 \]
\end{definition}

It will be useful to formulate subconformality  in terms of a family of operators on the space $\mt$ of complex Borel measures on $\T$. 
\begin{lemma} 
\label{lem:eulerproduct} 
For each  $\beta\in[0,\infty)$ and $n\in \nx$  define an  operator $A_{\beta,n}: \mt \to \mt$ on the Banach space of complex Borel measures on $\T$  using the left hand side of \eqref{eqn:positivitycondition},
\begin{equation}\label{eqn:anbdefinition}
A_{\beta,n} (\nu):=\sum_{d|n} \mu(d) d^{-\beta} \omega_{d*} \nu, \qquad \text{ i.e. \ }  \int_\T f dA_{\beta,n}(\nu) = \sum_{d|n} \mu(d) d^{-\beta} \int_\T f(z^d) d \nu
\end{equation}
for each $\nu \in \mt$ and $f\in C(\T)$. Then  
 
 \begin{enumerate}
 \smallskip\item $A_{\beta,m} A_{\beta,n}  = A_{\beta,mn} $ whenever $\gcd(m,n) =1$;
\smallskip \item $ A_{\beta,n} = \prod_{p |n}(1-p^{-\beta} \omega_{p*})$;

\smallskip\item if $\beta\in(0,\infty)$, then $A_{\beta,n}$ has a positive inverse, which for prime $n=p$ is given by the norm-convergent series
\[ A_{\beta,p}\inv = (1-p^{-\beta} \omega_{p*})\inv = \sum_{n=0}^\infty p^{-\beta n} \omega_{p^n*},\]
moreover $\prod_{p\mid n} (1-p^{-\beta}) A_{\beta,n}\inv \nu $ is a probability measure whenever $\nu$ is a probability measure;
\item  $A_{\beta,m}^{-1} \mt^+\supseteq A_{\beta,n}^{-1}\mt^+$ whenever $m|n$.
\end{enumerate}
\end{lemma}
\begin{proof}
A function on $\nx$ satisfying part  (1) is said to be {\em number-theoretic multiplicative}; notice that this will follow easily from part (2), which we prove next.
When $n=p$ is prime,  formula \eqref{eqn:anbdefinition} becomes  $A_{\beta,p} = 1-p^{-\beta} \omega_{p*}$. 
Since the operators $1-p^{-\beta} \omega_{p*}$ commute with each other, the usual inclusion-exclusion formula for the expansion of the product $ \prod_{p|n} (1-p^{-\beta} \omega_{p*})$ gives the formula in part (2) for square-free $n$. This suffices because  the M\"obius function eliminates the terms in which $d$ has repeated prime factors,  so that $ A_{\beta,n} =  A_{\beta,\prod_{p|n}p}$, where $ \prod_{p|n}p$ is square-free.  This proves part (2). %\begin{lemma}\label{lem:invertible}
%For every $n\in \nx$ and  every $\beta>0$ the operator $A_{\beta,n}$ has a positive inverse. When $n=p$ is prime the inverse of $A_{\beta,p} = 1-p^{-\beta} \omega_{p*}$ is given by the norm-convergent series  $\sum_{n=0}^\infty p^{-\beta n} \omega_{p^n*}$.
%\end{lemma}

%\begin{proof}
For part (3) first notice that if $\nu$ is a positive measure, then
\[
\int_\T f(z) \,d\omega_{n*} \nu (z)= \int_\T f (z^n)\,d \nu(z) \geq 0 \qquad \forall f \in C(\T)^+,
\]
hence the operator $\omega_{n*}$ is positive, and setting $f=1$ shows that
\[
\|\omega_{n*} \nu\| = (\omega_{n*}\nu)(\T) = \nu(\omega_n^{-1}(\T)) = \nu(\T) = \|\nu\|, \qquad \nu\in \mt^+.
\]
For a general measure $\nu\in\mt$, write $\nu = \nu_+ - \nu_- + i\nu_i - i\nu_{-i}$ for the complex Hahn-Jordan decomposition of $\nu$, so that 
$\omega_{n*}\nu = \omega_{n*}\nu_+ - \omega_{n*}\nu_- + i\omega_{n*}\nu_i - i\omega_{n*}\nu_{-i}$ is a decomposition for $\omega_{n*}\nu$ and thus, by minimality,  $\omega_{n*}(\nu_+) \geq (\omega_{n*}\nu)_+$ and so on,  hence
\[
\|\nu\|= \|\nu_+\| + \|\nu_-\| + \|\nu_i\| + \|\nu_{-i}\| = \|\omega_{n*}\nu_+\| + \|\omega_{n*}\nu_-\| + \|\omega_{n*}\nu_i\| + \|\omega_{n*}\nu_{-i}\| \geq \|\omega_{n*}\nu\|.
\]
 Assume now that $n$ is a prime number $p$. Then $ \omega_{p*} ^k = \omega_{p^k*}$, and since $\| p^{-\beta} \omega_{p*}\| = p^{-\beta} <1$   the well-known  Neumann series $\sum_{k=0}^\infty p^{-\beta k} \omega_{p^k*}$ of positive operators converges in the Banach algebra $\mathcal B(\mt)$  and gives the formula for $A_{\beta,p}\inv$  in part (3). Taking inverses in part (2), we conclude that  
\[
 A_{\beta,n}\inv  = \prod_{p |n }\sum_{k=0}^\infty p^{-\beta k} \omega_{p^k*}
 \]
is a positive operator. The last assertion follows from normalizing $A_{\beta,p}\inv $ with the factor $\prod_{p |n} (1-\omega_{p*})$.

In order to prove (4), suppose $m|n$ and let $k$ be the product of the primes that divide $n$ but not $m$. 
By part (2) $A_{\beta,n} = A_{k,\beta}A_{\beta,m}$, and hence 
$A_{\beta,n}^{-1} = A_{\beta,m}^{-1}A_{\beta,k}^{-1}$.  Since the operator $A_{\beta, k}^{-1}$ is positive,  
 $
 A_{\beta,n}^{-1}\mt^+ = A_{\beta,m}^{-1}(A_{\beta,k}^{-1}\mt^+) \subseteq A_{\beta,m}^{-1}\mt^+$.
 \end{proof}
 Motivated by \lemref{lem:eulerproduct}(2), we extend the notation to include finite subsets $F\Subset \primes$ and define an operator  $A_{\beta,F}$ on the space $\mt$ of complex measures on $\T$ by 
\begin{equation}\label{eqn:definitionabn}
A_{\beta,F}\nu := \prod_{p \in F }(1-p^{-\beta} \omega_{p*}) \nu = \sum_{d\in\nx_F} \mu(d)d^{-\beta} \omega_{d*}\nu , \qquad \nu \in \mt,
\end{equation}
 where $\nx_F$ is the set of all natural numbers whose prime factors are in $F$. Thus, if $F$ is the set of prime divisors of a given $n \in \nx$, then $A_{\beta,F} = A_{\beta,n}$ .
\begin{prop} \label{pro:bsubcequivalence}
The following are equivalent for $\nu \in \mt$:
\begin{enumerate}

\smallskip\item $\nu $ is $\beta$-subconformal;

\smallskip\item $A_{\beta,n} \nu \geq 0$ for every $n\in \nx$;

\smallskip\item $A_{\beta,F} \nu\geq 0 $ for all    $F \Subset \primes$;
\smallskip\item $\nu \geq \sum_{ \emptyset \neq  A\subset F } (-1)^{|A|+1} \prod_{p \in A }(p^{-\beta} \omega_{p*}) \nu $ for all  finite  $F \Subset \primes$;
\smallskip\item the atomic part and the nonatomic part of $\nu$ are $\beta$-subconformal.
\setcounter{listpause}{\value{enumi}}
\end{enumerate}
If in addition $\beta\in(0,\infty)$, then these are also equivalent to:
\begin{enumerate}
\setcounter{enumi}{\value{listpause}}
\smallskip\item $\nu \in \bigcap_n A_{\beta,n}^{-1} \mt^+$;
\smallskip\item $\nu \in \bigcap_{F\Subset \primes} \prod_{p\in F} A_{\beta,F}\inv \mt^+$.
\end{enumerate}
\end{prop}

\begin{proof} The equivalence of properties (1) through (4) is clear from \lemref{lem:eulerproduct}, and so is the equivalence between (6) and (7), using $F =\{p\in \primes: p\mid n\}$.
Let $\nu = \nu_a + \nu_{c}$ be the decomposition of   $\nu$ into its atomic and nonatomic parts. 
Observe that  $(A_{\beta,n}\nu)_a = A_{\beta,n}(\nu_a)$ and $(A_{\beta,n}\nu)_{c} = A_{\beta,n}(\nu_{c})$ because for each $d$, the map $\omega_d$ is $d$-to-1. Since a measure is positive if and only if its atomic and non-atomic parts are positive, (5) is equivalent to (2). If $\beta\in(0,\infty)$, then since the set of measures satisfying  \eqref{eqn:positivitycondition} for a given $n\in \nx$ is $A_{\beta,n}^{-1} \mt^+$, we also see that (6) is equivalent to  (1).
\end{proof}

\begin{remark}\label{rem:low-temp}
For $\beta\in (1,\infty)$, the series $T_\beta = \frac{1}{\zeta(\beta)}\sum_{c= 1}^\infty c^{-\beta}\omega_{c*}$ defines a bounded linear transformation on $\mathcal{M}(\T)$.
Combining \cite[Theorem 8.1]{aHLR21} and \thmref{thm:KMScharact}, we see that $T_\beta$ is an affine isomorphism between the simplex of all probability measures on $\T$ and (the simplex of) $\beta$-subconformal probability measures on $\T$  (in the low-temperature range). 
\end{remark}

\begin{lemma}\label{lem:Xdecomposition}
For each finite  $F\Subset \primes$ define  $e_F :=\prod_{p\in F} (1- V_p V_p^*)$ and let $\alpha_a$, $a\in \nx$ be the endomorphisms from \proref{pro:semigroupaction}. 
Then, for $\beta\in(0,\infty)$ and $\psi$ a \kmsb state,
\smallskip
\begin{enumerate}
\item  \ $e_F  = \sum_{d\in \nx_F} \mu(d) V_d V_d^*$;

\medskip
\item \ 
$\alpha_a(e_F) \alpha_b(e_F) = \begin{cases} \alpha_a(e_F) & a=b \\  0 & a\neq b
\end{cases} \qquad a,b \in\nx_F;
$

\medskip \item \ $\psi(e_F) = {\zeta_F(\beta)}\inv$;

\medskip
\item \  $\sum_{a\in \nx_F}\psi(\alpha_a(e_F)) = 1$.

\end{enumerate}
\end{lemma}

\begin{proof} Let $n_F = \prod_{p\in F} p$; then $e_F$ is the projection $e_{n_F}$ of \lemref{lem:complete-projections}, so part (1) follows from \lemref{lem:complete-projections} (1). Similarly for part (2), if $a,b\in \nx_F$, we have $\alpha_a(e_F) = e_{ab n_F, a}$ and $\alpha_b(e_F) = e_{abn_F, b}$, and \lemref{lem:complete-projections} (2) implies that these projections are mutually orthogonal. 

Since $\psi(V_dV_d^*) = d^{-\beta}$, (3) follows from (1) and the M\"obius inversion formula,
\[
\psi(e_F) = \sum_{d\in\nx_F} \mu(d) d^{-\beta} = \frac{1}{\zeta_F(\beta)}.
\]

For (4), use (3) to compute 
\[
\sum_{a\in\nx_F}\psi(\alpha_a(e_F)) 
=\sum_{a\in\nx_F} a^{-\beta} \psi(e_F)
= \zeta_F(\beta) \frac{1}{\zeta_F(\beta)}=1.\qedhere
\]

\end{proof}

\begin{lemma}\label{lem:necessary} 
Suppose $\psi$ is a \kmsb state of $(\tnxxz,\sigma)$ and let $\nu_\psi$ be the  probability measure on $\T$ representing the restriction of $\psi$ to $C(\T)$ as in \eqref{eqn:nufromvarphi}. Then $\nu_\psi$ is $\beta$-subconformal and
\[\psi(e_F\,V_1fV_1^*\,e_F) = \int_\T f \,dA_{\beta,F} \nu_\psi\qquad\forall f\in C(\T),\ F\Subset \primes.\]

\begin{proof} Suppose $f\in C(\T)$ and $F\Subset \primes$. Then
\begin{equation}\label{eqn:claimforpositivity}
e_FV_1fV_1^*e_F = V_1fV_1^*e_F  = \sum_{d\in\nx_F}\mu(d)  V_1fV_1^*V_dV_d^* =  \sum_{d\in\nx_F} \mu(d) V_d (f\circ \omega_d) V_d^*,
\end{equation}
where the second equality follows from equation \eqref{eqn:productxy}.

Let $\nu_\psi$ be the probability measure on $\T$ representing the restriction of a \kmsb state $\psi$, and assume $f\geq 0$. Then
\begin{align*}
\int_\T f dA_{\beta,F} \nu_\psi &= \int_\T f d \Big(\sum_{d\in\nx_F} \mu(d) d^{-\beta} \omega_{d*} \nu_\psi\Big) 
= \sum_{d\in\nx_F} \mu(d) d^{-\beta} \int_\T (f\circ \omega_d)  d\nu_\psi\\[5pt]
&= \sum_{d\in\nx_F} \mu(d) d^{-\beta} \psi(V_1 (f\circ \omega_d) V_1^*)
= \sum_{d\in\nx_F} \mu(d) \psi(V_d (f\circ \omega_d) V_d^*)\\[5pt]
& = \psi\Big(e_F V_1fV_1^*\,e_F\Big) 
\geq 0,
\end{align*}
where the first three equalities are obvious, the fourth one holds because of the \kmsb condition, and the fifth one holds by  \eqref{eqn:claimforpositivity}.
\end{proof}
 
\end{lemma}

Next we see that every measure on $\T$ gives rise to a linear functional on $\mfd_a$ via \eqref{eqn:KMScharact},
but only the $\beta$-subconformal ones extend to positive linear functionals on $\mfd = \lim \mfd_a$.

\begin{lemma}\label{lem:prime-positivity} Suppose 
$\nu$ is a  finite measure on $\T$ and let $\beta\in[0,\infty)$.  For each $a\in \nx$ there exists a unique linear functional $\psi_{\beta,\nu,a}$ on 
$\mfd_a := \lsp \{V_b f V_b^*: b|a, \ f\in C(\T)\}$ such that
\begin{equation}\label{eqn:extensionformula}
\psi_{\beta,\nu,a} (V_b f V_b^*) := b^{-\beta} \int_\T f d\nu  \qquad b |a, \ f\in C(\T),
\end{equation}
and  $(\psi_{\beta,\nu,a} )_{ a\in \nx}$ is a coherent family for the inductive system $(\mfd_a, \iota_{a,b})_{a\in \nx}$.

If in addition $\nu$  is $\beta$-subconformal,
 then  $\psi_{\beta,\nu,a}\geq 0$ for every $a$ and there is a unique positive linear functional $ \lim_a \psi_{\beta,\nu,a}$ on $\mfd = \lim \mfd_a$ 
 extending  $\psi_{\beta,\nu,a}$. If $\nu(\T) =1$, the gauge-invariant extension of the  limit functional is a \kmsb state of $\tnxxz$ given by

 \begin{equation}\label{eqn:KMScharactBoundQuot}
\psi_{\beta,\nu}(V_bfV_c^*) := (\lim_a \psi_{\beta,\nu,a}) \circ E(V_bfV_c^*) = \delta_{b,c} b^{-\beta} \int_\T fd\nu \qquad b,c \in \nx, \ f\in C(\T)
\end{equation}
where $E$ is the conditional expectation of \proref{pro:condexptheta}.
\end{lemma}

\begin{proof} 
From the proof of \corref{cor:system-spectra} we know that $\mfd_a$ is the linear space direct sum of the subspaces $V_b C(\T) V_b^*$ over the divisors $b$ of  $a$, and hence \eqref{eqn:extensionformula} defines a unique linear functional on $\mfd_a$. The resulting family $(\psi_{\beta,\nu,a})_{a\in \nx}$ of linear functionals is coherent with respect to inclusion because the right hand side does not depend on $a$ explicitly.

Suppose now that  $\nu$ is $\beta$-subconformal and notice, 
by setting   $n=1$ in \eqref{eqn:positivitycondition}, that $\nu$ is  positive, so we may as well assume without loss of generality that $\nu$ is a probability measure. We will show next that 
$\psi_{\beta,\nu,a}$ is a state of $\mfd_a$ for each $a \in \nx$. 
The isomorphism $\Gamma_a: C(X_a) \cong \mfd_a$ from \corref{cor:system-spectra} establishes a bijection between positive cones. For $f\in C(\T)$ and $b|a$, let $f^b\in C(X_a)$ be the function $f^b(z,d) = \delta_{b,d} f(z)$. Since the positive cone of $C(X_a) = C(\bigsqcup_{b|a} (\T\times \{b\}))$ is the direct sum of positive cones of the $C(\T\times \{b\})$, 
 the functional $\psi_{\beta,\nu,a}$ is positive if and only if 
 $\psi_{\beta,\nu,a}(\Gamma_a(f^b) ) \geq 0$ for every $b|a$ and every $f\in C(\T)^+$. We verify the latter condition by the following  direct computation using \corref{cor:system-spectra}:
\begin{align*}
\psi_{\beta,\nu,a}(\Gamma_a (f^b) ) 
&=  \psi_{\beta,\nu,a} \Big( \sum_{d|\frac{a}{b}} \mu(d) V_{bd}(f\circ \omega_d)V_{bd}^*\Big) \\
&= \sum_{d|\frac{a}{b}} \mu(d) (bd)^{-\beta} \int_{\T} (f\circ \omega_d ) d\nu \\ 
&= b^{-\beta} \int_\T f d\Big(\sum_{d|\frac{a}{b}} \mu(d) d^{-\beta} \omega_{d*} \nu\Big).
\end{align*}
This shows that  $\psi_{\beta,\nu,a}$ is positive as a linear functional on $\mfd_a$ if and only if condition \eqref{eqn:positivitycondition} holds for all divisors of $a$. Computing at the identity shows that $\psi_{\beta,\nu,a}$ is a state of $\mfd_a$. 
We have thus shown that $\{\psi_{\beta,\nu,a}\}_{a\in \nx}$ is a coherent system of states for the inductive system 
$(\mfd_a, \iota_{a,b})_{a\in \nx}$ and this uniquely defines a state $\lim_a \psi_{\beta,\nu,a}$ on the direct limit $\mfd$, which is  given by \eqref{eqn:KMScharactBoundQuot} with $b=c$.

Now let $\psi_{\beta,\nu} : = \lim_a \psi_{\beta,\nu,a} \circ E$ be the gauge-invariant extension induced via the conditional expectation of the gauge action. On the spanning monomials,  this extension is given by
\begin{equation}\label{eqn:psiformula}
\psi_{\beta,\nu}(V_b f V_c^*) =  \delta_{b,c} b^{-\beta} \int_\T f d\nu, \qquad b,c \in \nx, \ f\in C(\T),
\end{equation}
which obviously satisfies \eqref{eqn:KMScharact} and is thus a \kmsb state.
\end{proof}

We can now prove the first part of \thmref{thm:affinehomeom}.
\begin{theorem}[\thmref{thm:affinehomeom}(1)] 
\label{thm:KMScharact}
For each $\beta\in[0,\infty)$, the map that sends a \kmsb state $\psi$ to the measure $\nu_\psi$ on $\T$ representing the restriction of $\psi$ to $C^*(U)$, as in \eqref{eqn:nufromvarphi}, is an affine homeomorphism of the simplex of \kmsb states of $(\tnxxz, \sigma)$ onto the $\beta$-subconformal probability measures $\nu$ on $\T$.
The inverse map $\nu\mapsto \psi_{\beta,\nu}$ is given by \eqref{eqn:psiformula}.

\begin{proof}
By \lemref{lem:necessary}, the `restriction map' $\psi\mapsto \nu_\psi$ sends \kmsb states to $\beta$-subconformal probability measures. This map is clearly affine, weak* continuous, and also injective because of \eqref{eqn:KMScharact}, as noticed before. 
Suppose $\nu$ is a $\beta$-subconformal probability measure on $\T$ and let  $\psi_{\beta,\nu}$ be the \kmsb state constructed in  \lemref{lem:prime-positivity}. Setting $b=c=1$ in \eqref{eqn:psiformula} shows that the restriction of  $\psi_{\beta,\nu}$ to $C(\T) \cong V_1 C(\T) V_1^*$ is  $\nu$ again, proving at once that the map  $\psi\mapsto \nu_\psi$ is surjective and that its  inverse is $\nu \mapsto \psi_{\beta,\nu}$.  

Clearly the $\beta$-subconformal probability measures form a weak*-compact subset  of $\mt$, and,
being a continuous bijection of compact spaces, the map $\psi \mapsto \nu_\psi$ is a homeomorphism. Its image is a Choquet simplex in  $\mt$ because the \kmsb states form a Choquet simplex.
 \end{proof}
\end{theorem}

\begin{remark}
It is possible to realize  $\tnxxz$ as the C*-algebra of a finitely aligned product system of correspondences over $\nx$, which makes \eqref{def:subconformal} a particular case of the general positivity condition from \cite[Theorem 2.1]{ALN20}. Further, the reduction of positivity  to generators  from \cite[Theorem 9.1]{ALN20} applies here too, because  $\nx$ can be viewed as the right-angled Artin monoid corresponding to the full graph with vertices on the prime numbers, see \lemref{lem:eulerproduct}(2).
\end{remark}

 \section{Atomic subconformal measures on $\T$.} \label{sec:atomic}

In this section, we produce the list of $\beta$-subconformal measures for $\beta\in(0,1]$ that appear in \thmref{thm:affinehomeom}.
 The main result is \thmref{thm:atomic}, where we compute the decomposition of an arbitrary atomic $\beta$-subconformal probability measure in terms of the extremal ones.  We verify directly that Haar measure $\lambda$ on $\T$ satisfies $A_{\beta,B} \lambda = \prod_{p\in B} (1-p^{-\beta}) \lambda \geq 0$, and conclude that it is $\beta$-subconformal by \proref{pro:bsubcequivalence}; this extends the case $\beta =1$, which was already exhibited in \cite{aHLR21}.  The proof that $\lambda$ is the unique nonatomic $\beta$-subconformal probability measure for $\beta \in (0,1]$ is more involved and is given in the following section.

%The following notation will be used throughout the section. 
For each $k \in \nx$  the set of $k^{\mathrm th}$ roots of unity will be denoted by $Z_k$ and the primitive $k^{\mathrm th}$ roots of unity will be denoted by $Z_k^*$.  Also, $\mz{k}$ denotes the space of measures on  $Z_k$, viewed as a subspace of $\mt$, with positive cone $\mz{k}^+$.

\begin{proposition}\label{pro:finitesupport}
For $\beta \in [0,1]$ every $\beta$-subconformal atomic probability measure on $\T$  is supported on the roots of unity. Moreover, the only $0$-subconformal atomic probability measure is $\delta_1$.
\end{proposition}
\begin{proof}
Suppose that $\nu$ is a finite  $\beta$-subconformal measure on $\T$  for $\beta \leq 1$ and $\nu(\{z\}) >0$; we will show that $z$ is a root of unity.
 For prime $n=p$, the definition of subconformality  \eqref{eqn:positivitycondition}  reads
%\[%\sum_{d|p} \mu(d)d^{-\beta} \omega_{d*}(\nu)  = ( 1 - p^{-\beta}\omega_{p*} )\nu\geq 0.\]Thus 
$\nu \geq  p^{-\beta}\omega_{p*} \nu$. In particular, for each $a\in \nx$,
\[
\nu(\{z^{ap}\})\ \geq\ p^{-\beta} \nu(\omega_p^{-1}(\{z^{ap}\}))\ =\ p^{-\beta} \sum_{s:s^p = (z^{a})^p} \nu(\{s\})\ \geq\ p^{-\beta}  \nu(\{z^a\}).
\]
Iterating this procedure we see that $\nu(\{z^{ap^k}\}) \geq p^{-k\beta}  \nu(\{z^a\})$ for every $k\in \N$, and, more generally, using the prime factorization $n =  \prod_{p|n} p^{e_p(n)}$, we conclude that
\[
\nu(\{z^n\})  \geq n^{-\beta} \nu(\{z\}).
\]
Since $\beta\leq 1$, the series $\sum_n n^{-\beta}$ diverges. Hence the map $n \mapsto z^n$ cannot be injective, for otherwise $\nu(\{z^n:n\in \nx\})$ would be infinite by $\sigma$-additivity. Hence there  exist $n_1 \neq n_2$ such that $z^{n_1} =z^{n_2}$
and $z$ is an $(n_1-n_2)^\mathrm{th}$ root of unity.

Now suppose that $\beta = 0$ and $z\neq 1$ is a $k^\mathrm{th}$ root of unity. Then
\[\nu(\{1\}) \geq \nu(\omega_k^{-1}(\{1\})) \geq \nu(\{1\}) + \nu(\{z\}),\]
so that $\nu(\{z\}) = 0$. Therefore, $\nu = \delta_1$.
\end{proof}

\begin{lemma}\label{lem:restrictions}  
For each  $k\in \nx$  let $Z_k$ denote the set of $k^\mathrm{th}$ roots of unity, and 
for each measure $\nu$ on $\T$ denote by  $\nu|_k$ its restriction to $Z_k$, that is, $\nu|_k(A) := \nu(Z_k\cap A)$ for   measurable $A\subset \T$.
If  $\nu$ is $\beta$-subconformal, then so is  $\nu|_k$. Moreover, $\nu|_k$ converges to the atomic part of $\nu$  in the weak-* topology as $k \nearrow$ in $\nx$.
\end{lemma}

\begin{proof}
Fix $k\in \nx$ and suppose $z \in Z_k$ has primitive order $r| k$. For any prime $p$, there are three mutually exclusive and complementary possible cases for the set $\omega_p^{-1}(\{z\})\cap Z_k$:
\begin{equation}\label{eqn:pull-back-cases}
\omega_p^{-1}(\{z\}) \cap Z_k = \left\{\begin{array}{ll}
\emptyset & \text{if $p\nmid \frac{k}{r}$ and $p|r$}\\[2pt]
\{z^{1/p}\} & \text{if $p\nmid \frac{k}{r}$ and $p\nmid r$}\\[2pt]
\omega_p^{-1}(\{z\}) & \text{if $p | \frac{k}{r}$},
\end{array}\right.
\end{equation}
in the second case we have written $z^{1/p}$ for the unique element of $Z_k$ satisfying $(z^{1/p})^p = z$. 
Each square-free integer $n\in\nx$ factors uniquely as $n = n_1 n_2 n_3$, where $n_i$ is the product of prime factors corresponding to the $i^\mathrm{th}$ case of \eqref{eqn:pull-back-cases}. For $d|n_i$, since $n$ is square-free, it follows by induction on the number of prime factors of $d$ that \eqref{eqn:pull-back-cases} remains valid with $d$ in place of $p$. Consequently, $A_{\beta, n_1} \nu|_k(\{z\}) = \nu|_k(\{z\})$. By \lemref{lem:eulerproduct}(1) there is a (commuting) factorization yielding
$$A_{\beta,n} \nu|_k (\{z\})= A_{\beta, n_1} A_{\beta,n_2} A_{\beta,n_3} \nu|_k (\{z\})= A_{\beta,n_2} A_{\beta,n_3} \nu|_k (\{z\}).$$
For any $d_2|n_2$, the root $z^{1/d_2}$ is also a primitive $r^\mathrm{th}$ root of unity, so case (3) implies that $\omega_{d_3}^{-1}(\{z^{1/d_2}\})\cap Z_k = \omega_{d_3}^{-1}(\{z^{1/d_2}\})$ for any $d_3|n_3$. Hence:
\begin{align*}
A_{\beta,n_2} A_{\beta,n_3} \nu|_k (\{z\}) &= \sum_{d|n_2} \mu(d) d^{-\beta} A_{\beta,n_3} \nu|_k(\{z^{1/d}\})\\[5pt]
&=  \sum_{d|n_2} \mu(d) d^{-\beta} A_{\beta,n_3} \nu(\{z^{1/d}\})\\[5pt]
&= A_{\beta,n_2} A_{\beta,n_3} \nu(\{z\}) - \sum_{1\neq d | n_2} \mu(d) d^{-\beta} A_{\beta,n_3} \nu(\omega_d^{-1}(\{z\})\backslash \{z^{1/d}\}).
\end{align*}
Since $\nu$ is $\beta$-subconformal, \proref{pro:bsubcequivalence}(2) implies that the first term is positive; we will argue by induction that
\begin{equation}\label{eqn:positive-remainder}
-\sum_{1\neq d|n_2} \mu(d) d^{-\beta} A_{\beta,m} \nu(\omega_d^{-1}(\{z\})\backslash \{z^{1/d}\}) \geq 0
\end{equation}
when $m$ is relatively prime to $n_2$, from which it follows that the second term is also positive.

If $p$ is a prime dividing $n_2$, then we can write 
\begin{align*}
-&\sum_{1\neq d|n_2}\mu(d) d^{-\beta} A_{\beta,m}\nu(\omega_d^{-1}(z) \backslash \{z^{1/d}\}) \\[5pt]
&= p^{-\beta} A_{\beta,m}\nu(\omega_p^{-1}(z) \backslash \{z^{1/p}\})\\[5pt]
&\qquad\qquad- \sum_{1\neq d|\frac{n_2}{p}} \mu(d) d^{-\beta} \left[A_{\beta,m}\nu(\omega_d^{-1}(z) \backslash \{z^{1/d}\}) - p^{-\beta} A_{\beta,m}\nu(\omega_{pd}^{-1}(z) \backslash \{z^{1/pd}\})\right]\\[5pt]
&= p^{-\beta} A_{\beta,m}\nu(\omega_p^{-1}(z) \backslash \{z^{1/p}\}) - \sum_{1\neq d|\frac{n_2}{p}} \mu(d) d^{-\beta}A_{\beta,p}A_{\beta,m}\nu(\omega_d^{-1}(z) \backslash \{z^{1/d}\}).
\end{align*}
  The first term is positive since $\nu$ is $\beta$-subconformal. Since $p$ is relatively prime to $m$, \lemref{lem:eulerproduct}(1) says that $A_{\beta,p}A_{\beta,m} = A_{\beta,pm}$, so the second term is \eqref{eqn:positive-remainder} with $\frac{n_2}{p}$ and $pm$ in place of $n_2$ and $m$. Positivity then follows by induction on the number of prime factors of $n_2$. 

The final claim about the weak* limit is immediate because the atomic part of $\nu$  is supported on $\bigcup_k Z_k$ by \proref{pro:finitesupport}.
\end{proof}

As usual, we will write $\ord(z)$ for the order of $z$ in the group $\T$. That is, $\ord(z)$ is the primitive order if $z$ is a root of unity, and $\ord(z) = \infty$ if $z$ is not a root of unity. 
Recall the Euler totient function $\euler$ and  its  generalization $\euler_\beta$  defined in \secref{sec:introduction} by $\euler_\beta(n) := n^\beta \prod_{p|n} (1-p^{-\beta})$, where the product is over the primes that divide $n$.

\begin{lemma}\label{lem:atomiclist}
For  each $n \in \nx$ let $\eps_n$ be the atomic probability measure on $\T$ defined by 
\[
\eps_n(\{z\}) = \begin{cases} \frac{1}{\euler(n)} & \text{ if } \ord(z) =n,\\   0& \text{ otherwise;}
\end{cases}
\]
so that $\eps_n$ is evenly supported on the set $Z_n^*$ of primitive roots of unity of order $n$.
Define a measure
\begin{equation}\label{eqn:definitionnubetan}
\nu_{\beta,n} := \prod_{p|n} (1- p^{-\beta})  A_{\beta,n}\inv   \eps_n = \prod_{p|n} (1- p^{-\beta}) (1- p^{-\beta} \omega_{p*})\inv \eps_n
\end{equation}
for each $\beta \in (0,\infty)$. Then $\nu_{\beta,n}$ is a $\beta$-subconformal  atomic probability measure on $\T$  supported on $Z_n$ such that
\begin{equation}\label{eqn:nunbdefinition} 
\nu_{\beta,n}(\{z\}) = \left\{\begin{array}{ll}
n^{-\beta}\frac{\euler_\beta(\ord(z))}{\euler(\ord(z))},&\text{if }\ord(z) | n,\\
0,&\text{otherwise}.\end{array}\right.
\end{equation}
\end{lemma}
\begin{proof} Let $m\in \nx$ and write $m = a b$ in such a way that $(b,n) =1$ and all prime factors of $a$ divide $n$. Then   $A_{\beta,b}$ commutes with $A_{\beta,n}\inv = \prod_{p|n}  (1- p^{-\beta} \omega_{p*})\inv$ and 
$A_{\beta,b} \eps_n = \prod_{q|b} (1- q^{-\beta}) \eps_n$ because  $\omega_{q*} \eps_n = \eps_n$ whenever $\gcd(n,q) = 1$. Hence
\begin{align*}
A_{\beta,m} \nu_{\beta,n} &= A_{\beta,a} A_{\beta,b}\prod_{p|n} (1- p^{-\beta}) (1- p^{-\beta} \omega_{p*})\inv \eps_n \\ 
&=\prod_{p|n} (1- p^{-\beta}) A_{\beta,a} \prod_{p|n} (1- p^{-\beta} \omega_{p*})\inv A_{\beta,b}\eps_n \\
&= \prod_{p|n} (1- p^{-\beta}) \prod_{q\mid b}(1- q^{-\beta})\prod_{p|n, p\nmid a} (1- p^{-\beta} \omega_{p*})\inv \eps_n.
\end{align*}
 Since the last expression is $\geq 0$ by \lemref{lem:eulerproduct}(3), we conclude that $\nu_{\beta,n}$ is $\beta$-subconformal. By \lemref{lem:eulerproduct}(3) $\nu_{\beta,n}$ is a probability measure.

Before proving \eqref{eqn:nunbdefinition}, we point out that
\begin{equation}\label{eqn:epnpushforward}
\omega_{k*} \ve_n = \ve_{\frac{n}{\gcd(n,k)}}.
\end{equation}
This is computed directly:
\begin{align*}
\omega_{k*} \ve_n &= \omega_{k*} \left(\frac{1}{\euler(n)}\sum_{z\in\T, \ord(z) = n} \delta_z\right) \\[5pt]
&= \frac{1}{\euler(n)} \sum_{z\in\T, \ord(z) = n} \delta_{z^k}\\[5pt]
&= \frac{1}{\euler\left(\frac{n}{\gcd(n,k)}\right)} \sum_{z\in\T, \ord(z) = \frac{n}{\gcd(n,k)}} \delta_z\\[5pt]
&= \ve_{\frac{n}{\gcd(n,k)}},
\end{align*}
since the set of $z\in Z_n^*$ with $z^k = w$ contains $\euler(n)\euler\left(\frac{n}{\gcd(n,k)}\right)^{-1}$ elements for each $w \in Z_{\frac{n}{\gcd(n,k)}}^*$.

Now recall from \lemref{lem:eulerproduct}(3) that the operator $A_{\beta,p}^{-1}$ can be expressed as
$$A_{\beta,p}^{-1} = \sum_{m=0}^\infty p^{-\beta m} \omega_{p^m*}.$$
Letting $e=e_p(n)$ be the largest integer such that $p^e|n$, we have by \eqref{eqn:epnpushforward}:
$$(1-p^{-\beta})A_{\beta,p}^{-1} \ve_n = \sum_{m=0}^{e-1} (1-p^{-\beta})p^{-\beta m} \ve_{n/p^m} +p^{-\beta e} \ve_{n/p^e}.$$
Applying this to each prime divisor of $n$ gives the formula
$$\nu_{\beta,n} = \sum_{d|n} \left(\frac{n}{d}\right)^{-\beta} \left(\prod_{p|d} 1-p^{-\beta} \right)\ \ve_d= \sum_{d|n} n^{-\beta} \euler_\beta(d) \ve_d,$$
which proves \eqref{eqn:nunbdefinition}.
 \end{proof}

\begin{theorem}\label{thm:atomic} 
For $\beta \in (0,1]$ each atomic $\beta$-subconformal probability measure $\nu \in \mt$ can be written uniquely as a (possibly infinite) convex linear combination $\nu = \sum_n \lambda_n \nu_{\beta,n}$ with coefficients 
\[
\lambda_n = n^\beta \sum_{d\in\nx} \mu(d) \frac{1}{\euler_\beta(nd)} \nu(Z_{nd}^*) \qquad n\in \nx.
\]
In particular, $\{\nu_{\beta,n}: n \in \nx\}$ are the extremal  atomic $\beta$-subconformal probability measures.
\end{theorem}

In order to prove the theorem we need to establish a few properties of the measures $\nu_{\beta,n}$ first.
\begin{lemma}\label{lem:div}
For $\beta \in(0,\infty)$ and every $n,k \in \nx$,  
$
\omega_{k*} \nu_{\beta,n} = \nu_{\beta,\frac{n}{\gcd(n,k)}}
$.
\end{lemma}
\begin{proof} It is clear from its definition that $A_{\beta,n}^{-1}$ commutes with $\omega_{k*}$, so that $\omega_{k*} \nu_{\beta,n} = A_{\beta,n}^{-1} \omega_{k*} \ve_n = A_{\beta,n}^{-1} \ve_{\frac{n}{\gcd(n,k)}}$ by \eqref{eqn:definitionnubetan} and \eqref{eqn:epnpushforward}.
Using \lemref{lem:eulerproduct}(2), we have
\begin{align*}
\prod_{p|n}( 1-p^{-\beta})A_{\beta,n}^{-1} \ve_{\frac{n}{\gcd(n,k)}} &= \prod_{p|n}( 1-p^{-\beta})A_{\beta,\frac{n}{\gcd(n,k)}}^{-1} \prod_{p|n,p\nmid \frac{n}{\gcd(n,k)}} A_{\beta,p}^{-1} \ve_{\frac{n}{\gcd(n,k)}} \\[5pt]
&= \prod_{p|\frac{n}{\gcd(n,k)}}( 1-p^{-\beta})A_{\beta,\frac{n}{\gcd(n,k)}}^{-1} \ve_{\frac{n}{\gcd(n,k)}}\\[5pt]
&= \nu_{\beta,\frac{n}{\gcd(n,k)}}.\qedhere
\end{align*}

\end{proof}

The following is a simple consequence of Dirichlet's density theorem.

\begin{lemma}\label{lem:Dirichletcharacterthm}
    Suppose that $k\in \nx$, that $F$ is a finite subset of $\primes$  containing all the prime factors of $k$, and that $\beta \in (0,1] $. If $\chi$ is a nontrivial Dirichlet character modulo $k$, then 
    \[
\prod_{q\in A} \frac{1-q^{-\beta}}{ 1- \chi(q)q^{-\beta}} \underset{A\nearrow \primes \setminus F}{\longrightarrow} 0,
    \]
where the limit is taken  over finite sets of primes disjoint from $F$.
\end{lemma}

\begin{proof}
For every $q\in\primes$, one has $1-q^{-\beta}< |1-\chi(q)q^{-\beta}|$. If $\Re(\chi(q))<0$, then $1<|1-\chi(q)q^{-\beta}|$, which implies that $1-q^{-\beta}>\left|\frac{1-q^{-\beta}}{1-\chi(q)q^{-\beta}}\right|$. It follows that
$$\prod_{q\in A} \left|\frac{1-q^{-\beta}}{1-\chi(q) q^{-\beta}}\right| < \prod_{q\in A,\Re(\chi(q))<0} 1-q^{-\beta}.$$
Taking the logarithm of the right hand side, we have
$$-\log\left(\prod_{q\in A, \Re(\chi(q))<0} 1-q^{-\beta}\right) = \sum_{q\in A, \Re(\chi(q))<0} -\log(1-q^{-\beta}) > \sum_{q\in A, \Re(\chi(q))<0} q^{-\beta}.$$
The last series diverges for $\beta\in (0,1]$ as $A\nearrow\primes\setminus F$ by Dirichlet's density theorem \cite[Chapter IV, Section 4, Theorem 2]{ser}, whence the result follows.
\end{proof}

\begin{prop}\label{pro:finite-subspace}
Let $\beta \in (0,1]$ and fix $k\in \nx$. For each finite set of primes $L$ define an operator 
\[
P_{\beta,L} := \prod_{q\in L} (1-q^{-\beta}) A_{\beta,q}\inv : \M(Z_k) \to \M(Z_k).
\]
Then $P_{\beta,L}$ converges as $L \nearrow \primes$, and for each $\eta \in \M(Z_k)$, the limit $\lim_L P_{\beta,L}\eta $ is in $\lsp \{\nu_{\beta, d}: d\mid k\}$.   Moreover, 
if $\eta$ is a probability measure, then 
\begin{equation}\label{eqn:limitPbeta}
\lim_{L \nearrow \primes } \prod_{q\in L} (1-q^{-\beta}) A_{\beta,q}\inv \eta = \sum_{d\mid k} \lambda_d \nu_{\beta,d},
\end{equation}
with $\lambda_d\geq 0$  and $\sum_{d\mid k} \lambda_d =1$.
The limit is unchanged if one leaves out of the product an arbitrary finite subset  of primes that do not divide $k$.
\end{prop}
\begin{proof}%[Proof of \thmref{thm:atomic}]
The set $Z_k$ of $k^{\mathrm{th}
}$ roots of unity can be decomposed according to primitive order as a disjoint union $Z_k = \bigsqcup_{d\mid k} Z_d^*$ , and this gives a direct sum decomposition $\M(Z_k) = \bigoplus_{d\mid k} \M(Z_d^*)$ of  measure spaces
 (since $Z_k$ is finite we view measures as represented by their density functions). Let  $d$ be a divisor of $k$, and for each character $\chi \in \widehat{(\Z/d\Z)^*}$  consider
the vector $\tilde\chi \in \M(Z_d^*)$  obtained from $\chi$  through the identification $ (\Z / d\Z)^* \cong Z_d^*$ that sends the invertible element $u \in (\Z/d\Z)^*$ to the primitive $d^{\mathrm{th}}$ root of unity $\exp(2\pi i u / d) = (\xi_d)^u$; specifically,
\[
\tilde\chi((\xi_d)^u) = \chi(u), \qquad \chi \in \widehat{(\Z/d\Z)^*}
.\]
Then $\{\tilde\chi : \chi \in \widehat{(\Z/d\Z)^*}\}$ is a linear basis of $\M(Z_d^*) \cong \C^{\varphi(d)}$.
Suppose now that $q$ is a prime number that does not divide $k$ and let $\chi \in \widehat{(\Z/d\Z)^*}$.  Since
\[
(\omega_{q^m*} \tilde\chi) (\xi_d^u) = \tilde\chi (\xi_d^{uq^m}) = \chi(uq^m) = \chi(q)^m \chi(u) = \chi(q)^m\tilde\chi(\xi_d^u)\]
for every $m\geq 0$, \lemref{lem:eulerproduct}(3) shows that $\tilde\chi$ is an eigenvector of  $A_{\beta,q}\inv$,
\[
 A_{\beta,q}\inv \tilde\chi 
 = \sum_{m\geq 0} q^{-\beta m} (\omega_{q^m*} \tilde \chi)  
 = \sum_{m\geq 0} (q^{-\beta})^m  \chi(q)^m \tilde\chi =
  (1- \chi(q)q^{-\beta})\inv  \tilde\chi.
\]
Each $\chi \in \widehat{(\Z/d\Z)^*}$ can be extended to a Dirichlet character modulo $d$,  also denoted by $\chi$ and given by
 \[ 
 \chi(u) = \begin{cases} \tilde\chi(\xi_d^u) & \text{ if } \gcd(u,d) =1\\
 0 & \text{ if } \gcd(u,d) \neq 1
 \end{cases} \qquad u\in \Z.
 \]
Let $F$ be a fixed finite subset of primes not dividing $k$ and denote by $F\vee k$ the union of $F$ and the set of prime divisors of $k$. Suppose $\beta \in (0, 1]$. Then \lemref{lem:Dirichletcharacterthm} gives the following limit as $L \nearrow \primes $, with $1_d$  the trivial character in $\widehat{(\Z/d\Z)^*}$, 
\begin{equation}\label{eqn:Loverzeta}
 \Big(\prod_{\substack{q\in L \\ q\notin F\vee k}}(1-q^{-\beta}) A_{\beta,q}^{-1}\Big) \ \tilde\chi
 = \Big(\prod_{\substack{q \in L \\ q\notin F\vee k}} \frac{1-q^{-\beta}}{ 1- \chi(q)q^{-\beta}} \Big) \ \tilde\chi
 \ \underset{L\nearrow \primes}{\longrightarrow} \ \begin{cases} \tilde\chi &\text{ if } \chi = 1_d \\ 
			 0 &\text{ if }  \chi \in \widehat{(\Z/d\Z)^*}\setminus\{1_d\}.
\end{cases}
 \end{equation}

Suppose now $\eta \in \M(Z_k) $ and
combine all the bases of the $\M(Z_d^*)$ into a basis of $\M(Z_k)$, so $\eta$ can be written uniquely as  
$\eta = \sum_{d\mid k} \sum_{\chi \in \widehat{(\Z/d\Z)^*}} a_{d,\chi}
\tilde\chi$. 
Notice that the measure $\eps_d$ defined in \lemref{lem:atomiclist} is just 
$ \eps_d = \frac{1}{\varphi(d)} \tilde 1_d$.
Then  
\begin{equation}\label{eqn:limitaseps_d}
\lim_{L\nearrow \primes}  \Big(\prod_{\substack{q\in L \\ q\notin F \vee k}}(1-q^{-\beta}) A_{\beta,q}^{-1}\Big) \ \eta 
= \sum_{d\mid k}  a_{d,1_d} \varphi(d) \eps_d,
\end{equation}
because the contribution of the nontrivial characters vanishes in the limit by \eqref{eqn:Loverzeta}. By \lemref{lem:eulerproduct}(3) the measure above is positive and thus $\lambda_d:= a_{d,1_d} \varphi(d)\geq 0$ because the $\eps_d$ have disjoint support.

To finish the proof  simply apply the linear operator $\prod_{p|k} (1-p^{-\beta}) A_{\beta,p}^{-1}$ to both sides of \eqref{eqn:limitaseps_d}, using continuity on the left and the definition of $\nu_{\beta, d}$ on the right.
\end{proof}

\begin{lemma} \label{lem:subconformaliffPbeta}
Let $\beta \in (0,1]$.
A probability measure $\nu \in \mz{k}^+$ is $\beta$-subconformal if and only if 
\[
\nu = \lim_{L \nearrow \primes } \prod_{q\in L} (1-q^{-\beta}) A_{\beta,q}\inv \eta 
\]
for some probability $\eta \in \mz{k}^+$.
\end{lemma}

\begin{proof} 
Let $P_\beta$ denote the linear operator on $\mz{k}$  defined  by $P_\beta \eta = \lim_{L \nearrow \primes } \prod_{p\in L} (1-p^{-\beta}) A_{\beta,p}^{-1} \eta$. It suffices to show that 
\[
 P_\beta \mz{k}^+ = \bigcap_{n\in\nx}A_{\beta,n}^{-1}\mz{k}^+ 
\]
because the right hand side is the set of $\beta$-subconformal measures on $Z_k$ by \proref{pro:bsubcequivalence}(6).

Let $P_{\beta,n} := \prod_{p|n} (1-p^{-\beta})A_{\beta,p}^{-1}|_{\mz{k}}$.
That $P_\beta \mz{k}^+  \subseteq \bigcap_{n\in\nx}A_{\beta,n}^{-1}\mz{k}^+$ follows from \proref{pro:finite-subspace} because for each $n\in \nx$ 
\[
 P_\beta \mz{k}^+  = P_{\beta,n}  \prod_{p\in \primes, \ p\nmid n} (1-p^{-\beta}) A_{\beta,p}^{-1}   \mz{k}^+   \subseteq A_{\beta,n}^{-1}\mz{k}^+.
\]
It remains to  show that $\bigcap_{n\in\nx}A_{\beta,n}^{-1}\mz{k}^+ \subseteq P_\beta \mz{k}^+$.  Note that
\[
\|P_{n,\beta}\nu\| = (P_{n,\beta}\nu)(X) = \nu(X) = \|\nu\| \qquad \nu\in \mz{k}^+.
 \]
This shows that $\|P_{\beta,n}^{-1}y\| = \|y\|$ for every $y\in \bigcap_{n\in\nx} A_{\beta,n}^{-1} \mz{k}^+$; since $\M(Z_k)$ is finite-dimensional, the net $P_{\beta,n}^{-1}y$ has a subnet  $(P_{\beta,n_j}^{-1}y)$  converging to $x$. For $\ve>0$, choose some $K$ such that $\|P_{\beta,n_j}- P_\beta\|<\frac{\ve}{2\|y\|}$ and $\|P_{\beta,n_j}^{-1}y - x\| <\frac{\ve}{2\|P_\beta\|}$ for all $j>K$. Hence
\begin{align*}
\|y - P_\beta x\| &\leq \|y - P_\beta P_{\beta,n_j}^{-1}y\| + \|P_\beta P_{\beta,n_j}^{-1}y - P_\beta x\| \\[5pt]
&\leq \|P_{\beta,n_j} - P_\beta\|\cdot\|P_{\beta,n_j}^{-1}y\| + \|P_\beta\|\cdot\|P_{\beta,n_j}^{-1}y - x\| < \ve.
\end{align*}
Since $\ve$ is arbitrary, it follows that $P_\beta x = y$.
\end{proof}

\begin{proof}[Proof of \thmref{thm:atomic}]
Let $\nu$ be  an arbitrary $\beta$-subconformal atomic probability measure and fix $k\in \nx$.  By \lemref{lem:restrictions}, the restriction $\nu|_k := \nu(\cdot \cap Z_k)$ is $\beta$-subconformal, and hence  decomposes uniquely as $\sum_{n|k} \lambda_{k,n} \nu_{\beta,n},$ with $\lambda_{k,n}\geq 0$ and $\sum_{n\mid k} \lambda_{k,n}= \nu (Z_k)$, by \proref{pro:finite-subspace} and \lemref{lem:subconformaliffPbeta}. 
For each  $n|k$
\[
\nu(Z_n^*) = \sum_{d|\frac{k}{n}} \lambda_{k,nd} \nu_{ \beta,nd}(Z_n^*) = \sum_{d|\frac{k}{n}} \lambda_{k,nd}(nd)^{-\beta} \euler_\beta(n) = \sum_{d|\frac{k}{n}} \lambda_{k,\left(\frac{nd}{k}\right)k} \left(\frac{nd}{k}\right)^{-\beta} k^{-\beta} \euler_\beta(n).
\]
Reindexing the sum using the permutation $d\mapsto \frac{k}{nd}$  of  divisors of $\frac{k}{n}$ yields
$$\frac{1}{\euler_\beta(n)} \nu(Z_n^*)= \sum_{d|\frac{k}{n}} \lambda_{k,k/d} \left(\frac{k}{d}\right)^{-\beta}.$$
This relates the function $n\mapsto \frac{1}{\euler(k/n)} \nu(Z_{k/n}^*)$ to a summation of the function $d\mapsto \lambda_{k,k/d} \left(\frac{k}{d}\right)^{-\beta}$ over divisors of $n$. The M\"obius inversion formula then implies that
\begin{equation}\label{eqn:partialcoefficient}
\lambda_{k,n} =  n^\beta\sum_{d|\frac{k}{n}} \mu(d)\frac{1}{\euler_\beta(nd)} \nu(Z_{nd}^*).
\end{equation}
As $k$ increases in the directed set $\nx$ this gives rise to an absolutely convergent series 
because 
$\sum_{d|\frac{k}{n}} \left| \mu(d)\frac{1}{\euler_\beta(nd)} \nu(Z_{nd}^*)\right| \leq 
\nu(\bigsqcup_{d|\frac{k}{n}} Z_{nd}^*) \leq 1$. Thus we may define
\[
\lambda_n := n^\beta \sum_{d\in \nx} \mu(d) \frac{1}{\euler_\beta(nd)} \nu(Z_{nd}^*).
\]

It only remains to verify that $\nu = \sum_n \lambda_n\nu_{\beta,n}$. If $z$ is a primitive $k^\mathrm{th}$ root of unity, then 
\begin{align*}
\left(\sum_{n\in\nx} \lambda_n \nu_{\beta,n}\right)(\{z\}) &= \sum_{\substack{n\in\nx\\ k|n}} \left( n^\beta \sum_{d\in\nx} \mu(d) \frac{1}{\euler_\beta(nd)} \nu(Z_{nd}^*)\right)\left( n^{-\beta} \frac{\euler_\beta(k)}{\euler(k)}\right)\\[5pt]
&= \frac{\euler_\beta(k)}{\euler(k)} \sum_{\substack{n,d\in\nx\\k|n}} \mu(d) \frac{1}{\euler_\beta(nd)} \nu(Z_{nd}^*)\\[5pt]
&= \frac{\euler_\beta(k)}{\euler(k)} \sum_{\substack{m\in\nx\\ k|m}} \frac{1}{\euler_\beta(m)} \nu(Z_m^*) \sum_{d|\frac{m}{k}} \mu(d)\\[5pt]&= \frac{\euler_\beta(k)}{\euler(k)} \sum_{\substack{m\in\nx\\ k|m}} \frac{1}{\euler_\beta(m)} \nu(Z_m^*) 
\delta_{m,k}\\[5pt]
&= \frac{1}{\euler(k)} \nu(Z_k^*)\\[5pt]
&= \nu(\{z\}).
\end{align*}
The fourth equality holds because the M\"obius function satisfies $\sum_{d|\frac{m}{k}} \mu(d) = \delta_{m,k}$ and the last one holds because the value $\nu(\{z\}) = \nu|_k(\{z\})$ depends only on the order of $z$ and $|Z_k^*| = \euler(k)$.
\end{proof}

\begin{remark}
If we write  $Z_k^* = \bigcap_{p|k} Z_k \backslash Z_{k/a}$ and use the inclusion-exclusion principle, we get
$$\nu(Z_k^*) = \prod_{p|k}\Big(\omega_{k*} - \omega_{k/p*} \Big)\nu(\{1\}) = \sum_{a|k} \mu(a) \omega_{k/a*}\nu(\{1\}).$$
Which gives 
\begin{align*}
\lambda_n &= n^\beta \sum_{d\in\nx}\sum_{a|nd}  \frac{\mu(d)\mu(a)}{\euler_\beta(nd)}
  \omega_{nd/a*}\nu(\{1\})\\
  &= n^{\beta} \sum_{m\in \nx} \left( \sum_{a\in \nx} \frac{\mu(n' a) \mu(m' a)}{\euler_\beta((n\vee m) a)}\right)\omega_{m*}\nu(\{1\}),
\end{align*}
as an alternative expression  for $\lambda_n$  in terms of  the wrap-around maps $\omega_{d*}$ applied to $\nu$, where $n'm = m'n = n\vee m$. 
\end{remark}

\section{Asymptotic estimates for partial sums}
\label{app:appendixB}

Here we prove an asymptotic estimate for partial summation over $\nx$ using a partial order based on prime factorization. The multiplicative partial order plays an important role in the structure of the Toeplitz algebra and its KMS states, as shown by the subconformal condition. Our motivation is the application of \proref{prop:summation} in proving a multiplicative version of Wiener's lemma (cf. \proref{pro:leb-extreme}), but since its statement and proof rely solely on classical results from analytic number theory, we gather them in a separate section.

For each $n\geq 1$, let $p_n$ be the $n^{th}$ prime number and let $\primes_n = \{2,3,5, \ldots p_n\}$ be the set consisting of the first $n$ primes. We denote by $\nx_n$ the submonoid  of $ \nx$ generated by $\primes_n$, that is, $\nx_n$ consists of all natural numbers with no prime factors greater than $p_n$.
\begin{prop}\label{prop:summation}
Let $(a_n)_{n=1}^\infty$ be a bounded sequence of non-negative real numbers such that
\begin{equation}\label{eqn:propassumption}
\lim_{n\rightarrow\infty} \frac{1}{\log(n)}\sum_{m=1}^n \frac{a_m}{m} = 0.
\end{equation}
Then
\begin{equation}\label{eqn:mainapp}
\lim_{n\rightarrow\infty} \Big(\prod_{p\in \primes_n} (1- p^{-1})\Big)\sum_{m\in\nx_n} \frac{a_m}{m} = 0.
\end{equation}
\end{prop}

For the proof we need to gather a few tools from analytic number theory. As usual, 
when  $\lim_{n\to\infty} {f(n)}/{g(n)} = 1$, we say that  $f$ and $g$ are {\em asymptotically equal} and we write $f(n) \sim g(n)$.
 
 Mertens' Third Theorem states that 
 \[
\lim_{n\rightarrow\infty} \log (n) \prod_{p \leq n }(1 - p\inv) = e^{-\gamma},
\]
where $\gamma = 0.57721566...$ is Euler's constant. If we replace first $n$ by $p_n$  in the formula above, and then use the prime number theorem $p_n \sim n\log(n)$ to change the factor $\log p_n$ back to $\log n \sim \log (n\log n)$, we obtain 
\begin{equation*}
 \lim_{n\rightarrow\infty} \log(n) \prod_{p \leq p_n} (1 - p^{-1}) = e^{-\gamma}.
\end{equation*}
If we now take inverses and use the Euler product formula for the monoid $\nx_n$, we see that
\begin{equation}\label{eqn:mertens3new}
 \lim_{n\rightarrow\infty} \frac{1}{\log(n)} \sum_{m\in \nx_n} \frac{1}{m}  = e^\gamma.
\end{equation}

Abel's summation formula states that
if  $(a_n)_{n=1}^\infty$ is a sequence in $\C$ and  $A(x) := \sum_{1\leq m \leq x}
a_m$ for each $x\geq 1$, then
\begin{equation}\label{eqn:abelsum}
\sum_{1\leq m \leq x} a_m f(m) = -\int_1^x A(t) f'(t) dt + A(x) f(x)
\end{equation}
for every continuously differentiable function $f$ on $[1,\infty)$.

 For real $x>0$ and $y\geq 2$ let $\Psi(x,y)$ be the number of positive integers less than $x$ that have no prime divisors greater than $y$; this $\Psi$  is often called the {\em de Bruijn function}. Improving on earlier work of \cite{Buch,dic,CV,ram}, %Dickman \cite{dic}, Buchstab \cite{Buch}, Chowla--Vijayaraghavan \cite{CV}, and Ramaswani \cite{ram}
de Bruijn showed that the asymptotic estimate
$$ \frac{\Psi(x^u, x)}{x^u} \sim \rho(u)$$
is uniform for $1\leq u\leq (\log x)^{3/8 - \ve}$,  for any fixed $\ve>0$, see \cite{dB51} and also \cite{Hi86} and the references thereof.
Here $\rho(u)$ denotes the Dickman function, usually defined as the continuous solution to the delay differential equation
\[
u\rho'(u) + \rho(u-1) = 0
\]
with initial conditions $\rho(u) = 1$ for $0\leq u \leq 1$. In addition, de Bruijn further showed in \cite{dB66} that the  Dickman function has total mass $e^\gamma$, that is, 
\[
\int_0^{\infty} \rho(u) du = e^\gamma;
\]
we refer to \cite[Theorem 3.5.1]{Lag13}  for the details.
To make the uniform approximation precise, we borrow the statement of Hildebrand's improvement of de Bruijn's result, with $x^u$ in place of $x$.

\begin{prop}\label{prop:hildebrand}\cite[Theorem 1]{Hi86}
Let $\ve >0$. Then the estimate
$$\frac{\Psi(x^u, x)}{x^u} -\rho(u) = O_\ve \left(\frac{\rho(u) \log(u+1)}{\log x}\right)$$
holds uniformly in the range $x\geq 3$ and $1\leq u \leq \log x/ (\log\log x)^{5/3 + \ve}$.
\end{prop}

Next define a function $\delta:[1,\infty)\rightarrow \R$ by
\[
\delta(u) := \limsup_x \int_{u}^\infty \frac{\Psi(x^s, x)}{x^s} ds= \limsup_x \frac{1}{\log(x)}\int_{x^u}^\infty \frac{\Psi(t, x)}{t^2} dt. 
\]
%where the last expression is obtained via the change of variables $t=x^s$.
We will require the following properties of $\delta(u)$.

\begin{lemma}
The function $\delta$ is differentiable with $\delta'(u) = -\rho(u)$ and $\delta(1) = e^\gamma - 1$. Moreover, $\lim_{u\rightarrow\infty} \delta(u) = 0$. 
\end{lemma}

\begin{proof}
Fix $\ve>0$. By  Proposition \ref{prop:hildebrand} there exists a constant $C_\ve>0$  such that 
$$\left|\frac{\Psi(x^s, x)}{x^s} - \rho(s)\right| \leq C_\ve \cdot \frac{\log(s+1)}{\log(x)}$$
for $1\leq s \leq \log(x)/(\log\log(x))^{5/3 + \ve}$  (we may drop the factor $\rho(u) \leq 1$ from the r.h.s.). Then 
$$\int_u^{u+h} \left|\frac{\Psi(x^s, x)}{x^s} - \rho(s) \right|ds \leq \frac{C_\ve}{\log(x)} \int_u^{u+h} \log(s+1) ds$$
for $h > 0$ and sufficiently large $x$. The right hand side converges to $0$ as $x\rightarrow\infty$, hence
$$\lim_{x\rightarrow\infty} \int_u^{u+h} \frac{\Psi(x^s, x)}{x^s} ds = \int_u^{u+h} \rho(s) ds.$$
It then follows that
\begin{align*}
\delta(u) &= \limsup_x \int_u^\infty \frac{\Psi(x^s, x)}{x^s} ds\\[5pt]
&= \limsup_x \int_{u+h}^\infty \frac{\Psi(x^s, x)}{x^s} ds + \int_u^{u+h} \rho(s) ds\\[5pt]
&= \delta(u+h) + \int_u^{u+h} \rho(s) ds,
\end{align*}
%(by additivity of $\limsup$ when one sequence converges); this 
which implies $\delta'(u) = -\rho(u)$.

In order to see that $\delta(1) = e^\gamma - 1$, fix $x>0$ and let $n = \pi(x)$, so that  $\nx_n$ is the set of positive integers with no prime factors larger than $x$ . Consider the sequence $b_m = 1$ if $m\in \nx_n$ and $0$ if $m\notin \nx_n$. Then  $B(y) = \sum_{1\leq m \leq y} b_m = \Psi(y,x)$ and Abel's summation formula \eqref{eqn:abelsum} with $f(x) = 1/x$  gives
$$\sum_{1\leq m \leq y} \frac{b_m}{m} = \int_1^y \frac{\Psi(t,x)}{t^2} dt + \frac{\Psi(y,x)}{y}.$$
Taking limits as  $y\rightarrow\infty$ we see that
$$\sum_{m\in \nx_n} \frac{1}{m} = \int_1^\infty \frac{\Psi(t,x)}{t^2} dt,$$
which gives 
\begin{equation*}
\frac{1}{\log(x)}\int_{x}^\infty \frac{\Psi(t,x)}{t^2} dt = \frac{1}{\log(x)} \Big(\sum_{m\in \nx_n} \frac{1}{m} - \sum_{1\leq m\leq x}\frac{1}{m} - 1\Big).
%& = \frac{1}{\log(x)} \left( \frac{1}{\prod_{p\in\primes_n}(1-p\inv)}- H(\lfloor x \rfloor) - 1\right),
\end{equation*}
The right-hand side converges to $e^\gamma - 1$ as $x\rightarrow\infty$ because of \eqref{eqn:mertens3new} and the asymptotic formula for the harmonic numbers $H_n:= \sum_{m=1}^{n}\frac{1}{m} \approx \log(n) + \gamma$.
\end{proof}

\begin{proof}[Proof of Proposition \ref{prop:summation}]
Assume that $0\leq a_m\leq 1$ for every $m$. Then, for fixed $u\geq 1$,
\begin{equation}\label{eqn:inequality1}
\frac{1}{\log(n)}\sum_{m\in\nx_n} \frac{a_m}{m}\leq \frac{1}{\log(n)}\sum_{1\leq m\leq p_n^u} \frac{a_m}{m} + \frac{1}{\log(n)}\sum_{\substack{m\in \nx_n\\ m>p_n^u  }} \frac{1}{m} \qquad (n\geq1).
\end{equation}
The first summand on the right of \eqref{eqn:inequality1} converges to 0 as $n\rightarrow \infty$ by assumption \eqref{eqn:propassumption}, since $\log p_n^u\sim u \log n$. For the second summand, Abel's summation formula gives 
$$\frac{1}{\log(n)} \sum_{\substack{m\in\nx_n\\ m> p_n^u}}\frac{1}{m} = \frac{1}{\log(n)}\int_{p_n^u}^\infty \frac{\Psi(t, p_n)}{t^2} dt  - \frac{1}{\log(n)}\frac{\Psi(p_n^u, p_n)}{p_n^u}.$$
The first term is bounded above by $\delta(u) + \ve$ for each $\ve>0$ and sufficiently large $n$, while the second term converges to $0$ as $n\rightarrow\infty$. Thus, we have the following bound for each $u\geq 1$,
$$\limsup_{n\rightarrow\infty} \frac{1}{\log(n)}\sum_{m\in\nx_n}\frac{a_m}{m} \leq \delta(u).$$
Since $\inf \delta(u) = 0$, it follows that 
\[
\lim_{n\rightarrow\infty} \frac{1}{\log(n)}\sum_{m\in\nx_n}\frac{a_m}{m}  =0.
\]
Finally, by \eqref{eqn:mertens3new} we may put $\prod_{p\in\primes_n} \left(1-p^{-1}\right)$ in place of $\frac{1}{\log(n)}$ above (the factor $e^{\gamma}$ is irrelevant), which yields \eqref{eqn:mainapp}, as required.
\end{proof}

Recall that a subset $J\subseteq \N$ has natural density 0 if 
$\lim_{n\to\infty}\#\{ j \in J : j\leq n\}/n = 0$.
Next we see that if a subset $J\subseteq\N$ has natural density 0, then it has \emph{multiplicative density 0}.
\begin{cor}\label{cor:density-0}
If $J\subseteq \nx$ is a set of natural density $0$, then
$$\lim_{n\rightarrow\infty} \Big(\prod_{p\in \primes_n} (1 - p^{-1}) \Big)\sum_{m\in \nx_n\cap J} \frac{1}{m} = 0.$$
\end{cor}

\begin{proof}
Let $b_n = 1$ if $n\in J$ and $b_n = 0$ if $n\notin J$, and define $B(x) := \sum_{1\leq m \leq x} b_m$.
Abel's summation formula \eqref{eqn:abelsum} gives
$$\sum_{1\leq m \leq x} \frac{b_m}{m} = \int_1^x \frac{B(t)}{t^2} dt + \frac{B(x)}{x}.$$
Since $J$ has natural density 0, the function ${B(t)}/{t}$ converges to $0$ as $t\rightarrow\infty$, so for each $\ve>0$ we may choose  $T\geq 1$ such that ${B(t)}/{t}<\ve$ for all $t>T$. Hence 
\begin{align*}
\frac{1}{\log(x)}\left(\int_1^x \frac{B(t)}{t^2} dt + \frac{B(x)}{x}\right) < \frac{1}{\log(x)} \left(\int_1^T \frac{B(t)}{t^2} dt + \frac{B(x)}{x}\right) + \ve\left(1 - \frac{\log T}{\log x}\right).
\end{align*}
The first term on the right converges to 0 as $x\rightarrow\infty$, while the second term converges to $\ve$. Since $\ve>0$ was arbitrary, 
\[\lim_{n\to \infty}\frac{1}{\log(n)}\sum_{m=1}^n \frac{b_m}{m} = 0, \]
and the result follows from \proref{prop:summation}.
\end{proof}

\section{Uniqueness of nonatomic subconformal measure}\label{sec:nonatomic}

We now turn our efforts to showing that   \thmref{thm:main} gives a complete list of extremal \kmsb states for each $\beta \in (0,1]$. After \thmref{thm:atomic}, all that remains to show is that Haar measure is the only nonatomic $\beta$-subconformal measure on $\T$. Our argument is inspired by \cite{nes} for the Bost-Connes system: we show that a certain dilation of a given nonatomic $\beta$-subconformal measure is ergodic and, hence, that there is a unique such measure.

For each $B\subseteq \primes$ the subset $\{V_a f V_a^*: a\in\nx_B, f\in C(\T)\} \subset \mfd $ is self-adjoint  and \eqref{eqn:productxy} shows it is closed under multiplication. Hence
\[
\mfd_B := \clsp \{V_a f V_a^*: a\in\nx_B, f\in C(\T)\} = \varinjlim (\mfd_a, \iota_{a,b})_{a\in\nx_B}
\]
 is a unital C*-subalgebra of $\mfd$. The inclusion $\iota_B: \mfd_B\hookrightarrow \mfd$ induces a surjective continuous map of spaces $\iota_B^*:X= \Spec \mfd \rightarrow X^B:= \Spec\mfd_B$ and also a continuous linear map taking a measure $\tau$ on $X$ to the measure $\iota^*_B(\tau)$ on $X^B$. 

Now let $\nu$ be a $\beta$-subconformal probability measure on $\T$ and $\psi_{\beta,\nu}$ the corresponding \kmsb state from \thmref{thm:KMScharact}. For any subset $B\subseteq \primes$, the restriction $\psi_{\beta,\nu}|_{\mfd_B}$ is a state on $\mfd_B$ that, according to the Riesz-Markov-Kakutani representation theorem, is represented by integration against a measure which we call $\nu_{\beta,B}$. Thus, $\nu_{\beta,\primes}$ is the measure on $X = X^{\primes}$ corresponding to $\psi_{\nu,\beta}$, and, at the other extreme, with $B =\emptyset$, $\nu_{\beta, \emptyset} = \nu$ is the measure on $X_\emptyset = \Spec V_1 C(\T) V_1^*\cong \T$ that already appeared in  \eqref{eqn:nufromvarphi}. These measures satisfy $\iota_B^*(\nu_{\beta,\primes}) = \nu_{\beta,B}$.

\begin{lemma}\label{lem:local-GNS}
Suppose $\nu$ is a $\beta$-subconformal probability  on $\T$ and $B\Subset \primes$ a finite subset of primes.
 Then there is a representation $\rho_B$
 of $\mfd_B$ on $\HH_B := \bigoplus_{n\in\nx_B} L^2(\T, n^{-\beta} A_{\beta,B} \nu)$
such that for each $V_a f V_a^* \in \mfd_B$ and each $n\in \nx_B$, 
\[
\rho_B(V_a f V_a^*) g_n = 
\begin{cases}
(f\circ \omega_{n/a}) g_n &\text{if $a|n$},\\
0&\text{otherwise,}
\end{cases} \qquad  \ g_n \in L^2(\T, n^{-\beta} A_{\beta,B} \nu).
\]
The vector $\Omega_B := (1_\T)_{n\in \nx_B}$ is cyclic for $\rho_B$ and $(\HH_B, \rho_B,\Omega_B)$ is canonically unitarily equivalent to the GNS representation of the restriction to $\mfd_B$ of the \kmsb state ${\psi}_{\beta,\nu}$ from \lemref{lem:prime-positivity}.

Consequently, if  $\nu_{\beta,B}$ denotes the measure on $X_B$ representing $\psi_{\nu,\beta}|_{\mfd_B}$ and $(L^2(X^B, \nu_{\beta,B}), \lambda_B, \Omega_0)$ its GNS triple, there is a (unique) unitary intertwiner $T:\HH\rightarrow L^2(X^B,\nu_{\beta,B})$ for the representations $\rho_B$ and $\lambda_B$ that satisfies $T\Omega = \Omega_0$.

\end{lemma}

\begin{proof}
The operators $\rho_B(V_dfV_d^*)$ satisfy the product formula \eqref{eqn:productxy} so that $\rho_B$ defines a representation of $\mfd_a$ for each $a\in\nx_B$, cf. \corref{cor:system-spectra}. These representations are coherent for the inductive system $(\mfd_a,\iota_{a,b})_{a\in\nx_B}$, and therefore determine a representation of the limit $\mfd_B$. The vector $\Omega$ is cyclic  because  
$$\rho_B(\sum_{d\in\nx_B} \mu(d) V_{kd} f\circ\omega_d V_{kd}^*)\Omega = (\delta_{n,k} f),\qquad f\in C(\T)$$
and these vectors span a dense subspace of $\HH_B$. Lastly, the vector state of $\Omega$ yields
\begin{align*}
\<{\rho_B(V_afV_a^*)\Omega, \Omega} &= a^{-\beta}\sum_{n\in\nx_B} n^{-\beta} \int_\T f\circ \omega_n d(A_{\beta,B}\nu)\\[5pt]
&= a^{-\beta} \sum_{n\in\nx_B} n^{-\beta} \sum_{d\in \nx_B} \mu(d)d^{-\beta} \int_\T f\circ \omega_{nd} d\nu\\[5pt]
&= a^{-\beta} \sum_{n\in\nx_B} n^{-\beta} \int_\T f\circ \omega_n d\nu \sum_{d|n} \mu(d)\\[5pt]
&= a^{-\beta} \int_\T f d\nu = \int_{X_B} (V_a f V_a^*) d\nu_{\beta,B}.\qedhere
\end{align*}
\end{proof}

\begin{lemma}\label{lem:coisometries}
Suppose $B\subseteq \primes$ (not necessarily finite) and consider $(L^2(X^B,\nu_{\beta,B}), \lambda_B, \Omega_0)$, the usual GNS representation of $\mfd_B$ for the state corresponding to $\nu_{\beta,B}$. Then for $a\in \nx$, the map
$$S_a:L^2(X^B,\nu_{\beta,B})\rightarrow L^2(X^B, \nu_{\beta,B}),\qquad \lambda_B(x)\Omega_0\mapsto \lambda_B(V_a^* x V_a)\Omega_0$$
defines a bounded operator and the map $a\mapsto S_a$ is multiplicative. For finite $B\Subset \primes$ and $T:\HH_B\rightarrow L^2(X^B,\nu_{\beta,B})$ the unitary intertwiner of \lemref{lem:local-GNS}, the operator $T^* S_a T$ can be described explicitly by the formulae
\[
T^*S_aT(g_n)_{n\in\nx_B} =\left\{\begin{array}{ll} (g_{an})_{n\in\nx_B} &\text{if $a\in \nx_B$}\\[2pt]
(g_n\circ \omega_a)_{n\in\nx_B} & \text{if $a\in \nx_{\primes\backslash B}$}
\end{array}\right.
\]
and by prime factorization for general $a\in\nx$. 
\end{lemma}

\begin{proof}
Let $\psi_{\nu,\beta}$ denote the \kmsb state associated to $\nu$ in equation \eqref{eqn:KMScharactBoundQuot}. Then for $x\in \mfd_B$,
$$\|\lambda_B(V_a^* x V_a)\Omega_0\|^2 = \psi_{\nu,\beta} ( V_a ^*x^* V_a V_a^* x V_a) = a^\beta \psi_{\nu,\beta}(x^* V_a V_a^* x) \leq a^\beta\psi_{\nu,\beta}(x^* x) = a^\beta\|\lambda_B(x)\Omega_0\|^2,$$
where we have used the fact that $x$ commutes with $V_a V_a^*$. That $a\mapsto S_a$ is multiplicative follows from commutativity of the isometries $\{V_a: a\in\nx\}$.

If $a\in\nx_B$, $b\in \nx_B$, and $f\in C(\T)$, then letting 
$$(f_n)_{n\in\nx_B} = \rho_B(V_a f V_a^*) \Omega,$$
and $a' = \frac{\lcm(a,b)}{b}$, $b' = \frac{\lcm(a,b)}{a}$,
$$ T^*S_aT \rho_B(V_b f V_b^*)\Omega = \rho_B(V_a^* V_b f V_b^* V_a) \Omega = \rho_B(V_{b'} f\circ\omega_{a'} V_{b'}^*)\Omega = (f_n')_{n\in\nx_B},$$
where 
$$f_n' = \left\{\begin{array}{ll}
f\circ \omega_{a'n/b'}&\text{if $b'|n$},\\[2pt]
0 & \text{otherwise}
\end{array}\right.$$
Now if $b'| n$, then $b | \lcm(a,b) | an$; conversely, if $b| an$, then $b' | a'n$, which implies $b'|n$ since $a'$ and $b'$ are relatively prime. Since $\frac{a'n}{b'} = \frac{an}{b}$ when $b'|n$, it follows that $f_n' = f_{an}$. 

Lastly, if $a\in \nx_{\primes\backslash B}$, then $a$ and $b$ are relatively prime, and we have $T^*S_aT\rho_B(V_b f V_b^*)\Omega = \rho_B (V_b f\circ \omega_a V_b^*)\Omega.$
\end{proof}

Recall  the {\em periodic zeta function} defined  by the series 
\[
F(\beta,z) := \sum_{n\in\nx} m^{-\beta} z^m, \qquad  \Re(\beta)>1, \ z\in\T,
\]
(where we have chosen to deviate slightly from standard  practice by using  $z = \exp(2\pi i n \alpha)\in\T$ instead of $\alpha \in \R$ for the second variable).
For each finite $B\Subset \primes$ define also  a partial periodic zeta function by the partial series
\[
F_B(\beta,z) := \sum_{m\in\nx_B} m^{-\beta} z^m ,
\qquad  \Re(\beta)>0, \ z\in \T,\]
where convergence is absolute. Write $\primes_n= \{2,3,5, ...,p_n\}$ for the set of the first $n$ primes, $\nx_n := \nx_{\primes_n}$ for the monoid generated by $\primes_n$, and, accordingly,
  \[
  \zeta_n(\beta):= \zeta_{\primes_n}(\beta) = \sum_{m\in\nx_n} \frac{1}{m^\beta} \qquad \text{and} \qquad
  F_n(\beta,z) := F_{\primes_n}(\beta,z) = \sum_{m\in\nx_n}  \frac{z^m}{m^\beta}
  \]
We will need the following consequence of Wiener's Lemma; see \cite[Theorem III.24]{wie} for the original statement and \cite[Theorem 1.1]{CEF19} for the version of the lemma that we use here. Our proof relies on  \corref{cor:density-0}.

\begin{prop}\label{pro:leb-extreme}
Suppose $\nu$ is a nonatomic probability measure  on $\T$ and fix $B\Subset \primes$.  
Then
\begin{equation}\label{eqn:cont-measure}
\lim_{n\rightarrow\infty} \frac{1}{\zeta_n(1)} \sum_{m\in \nx_{\primes_n\backslash B}} \frac{1}{m}\wh{\nu}(\ell m + k ) 
= 0\quad \forall \ell\in \Z \setminus \{0\},\ k\in\Z.
\end{equation}
\end{prop}

\begin{proof}
Since $\nu$ is a probability measure, $\wh{\nu}(-m) = \overline{\wh{\nu}(m)}$ and $|\wh{\nu}(m)|\leq 1$ for each $m\in\Z$, and since $\nu$ is assumed to be nonatomic,  \cite[Theorem 1.1]{CEF19} implies that 
\[
\lim_{N\to \infty} \frac{1}{N} \sum_{m =1}^N |\wh{\nu}(m)|^2 = 0.
\]
By \cite[Lemma 2.1]{CEF19},  the sequence $\wh{\nu}(m)$ converges in density to $0$ as $m \to \infty$; that is, there exists a set $J\subset \N$ with natural density $0$ such that $\lim_{n\in \N \backslash J} \wh{\nu}(n) = 0$. Since $\nx_B$ has natural density $0$, we may assume without loss of generality that $\nx_B\subseteq J$. 
If we now let 
\[
a_m = \begin{cases} \wh{\nu}(\ell m + k) & \text{if } |\ell m + k|\in \N\backslash J \\
0& \text{if } |\ell m + k|\in J,
\end{cases}
\]
 then  $a_m \to 0$ as $m\to \infty$, and
 \begin{equation}\label{eqn:twoterms}
\frac{1}{\zeta_n(1)} \left|\sum_{m\in \nx_{\primes_n\backslash B}} \frac{1}{m}\wh{\nu}(\ell m + k)\right| \ \leq \ \frac{1}{\zeta_n(1)}\left|\sum_{m\in \nx_n} \frac{a_m}{m}\right| + \frac{1}{\zeta_n(1)} \sum_{m\in \nx_n\cap J} \frac{1}{m}.
\end{equation}
Using the Euler product formula for the first $n$ primes,
\[
\frac{1}{\zeta_n(\beta)} = \prod_{p\in \primes_n} (1 - p^{-\beta}), \qquad \beta >0,
\]  
 with $\beta =1$ and applying  \corref{cor:density-0}, we see that the second term on the right hand side of \eqref{eqn:twoterms} converges to $0$ as $n\to \infty$. That the first term also converges to $0$ is a consequence of the following general observation.
Suppose  that $\alpha_m$ is a sequence converging to $0$ and let $\ve>0$. Choose  $N$ such that $|\alpha_m|<\ve$ for all $m>p_N$. Then, for each $n\geq N$,
\begin{align*}
\frac{1}{\zeta_{n}(1)} \left|\sum_{m\in \nx_n} \frac{\alpha_m}{m}\right| \leq \frac{1}{\zeta_n(1)} \sum_{m\leq p_N} \frac{|\alpha_m|}{m} + \frac{1}{\zeta_{n}(1)} \sum_{\substack{m\in \nx_n \\  m > p_N}} \frac{|\alpha_m|}{m} 
< \frac{1}{\zeta_n(1)} \sum_{m\leq p_N} \frac{|\alpha_m|}{m} + \ve.
\end{align*}
Since ${\zeta_{n}(1)} \to \infty$, the right hand side tends to $\ve$ as $n\rightarrow\infty$, and since $\ve$ is arbitrary, the left hand side tends to 0.
This shows that $\frac{1}{\zeta_n(1)} \left|\sum_{m\in \nx_n} \frac{a_m}{m}\right|$ converges to $0$ and completes the proof of \eqref{eqn:cont-measure}.
\end{proof}

As explained in Appendix A, the action $\alpha$ of $\nx$ on $\mfd$ of \proref{pro:semigroupaction} can be dilated to an action $\wt{\alpha}$ of $\qx$ on a commutative C*-algebra $\wt{\mfd}$ so that $\nx\ltimes \mfd$ embeds as a full-corner in $\qx\ltimes \wt{\mfd}$. In this dilation, the spectrum $X$ of $\mfd$ can be realized as a compact open subset of the spectrum  $\wt{X}$ of $\wt{\mfd}$. \lemref{lem:dilationmeasure} shows that if $\nu$ is a $\beta$-subconformal measure on $\T$, then there is a unique Radon measure $\wt{\nu}_{\beta,\primes}$ on $\wt{X}$  that extends $\nu_{\beta,\primes}$ and satisfies rescaling: $\wt{\alpha}_{a*} \wt\nu_{\beta,\primes} = a^{-\beta} \wt\nu_{\beta,\primes}$ for $a\in \qx$.

\begin{prop}\label{pro:ergodic}
Suppose  $\nu$ is a nonatomic $1$-subconformal measure on $\T$, and let $\wt{\nu}_{1,\primes}$ be the dilated measure on $\tilde X$ from \lemref{lem:dilationmeasure}.
 Then the action of $\qx$ on $(\tilde X,\wt{\nu}_{1,\primes})$ is ergodic.
\end{prop}

\begin{proof}
We argue along the lines set out in \cite{nes} for the Bost--Connes system; the idea is to show that the subspace of $\qx$-invariant functions in $L^2(\tilde X,\wt{\nu}_{1,\primes})$ consists only of constant functions. 
 Since  $\tilde X = \bigcup_{a\in\nx}(\tilde\alpha_a)_*\inv(X)$ by minimality of the dilation,  every $\qx$-invariant function on $\tilde X$ is determined by its restriction to $X$. Using the left inverse for $\alpha_a$ of \proref{pro:semigroupaction}, for $f\in L^2(\wt{X},\wt{\nu}_{1,\primes})$, we have $\wt{\alpha}_a^{-1}(f)|_X = V_a^* f|_X V_a = S_a f|_X$, where $S_a$ are the bounded operators on $L^2(X,\nu_{1,\primes})$ from \lemref{lem:coisometries}.
Thus it suffices to  show that the subspace
\[
H := \{f\in L^2(X,\nu_{1,\primes}): S_a(f) = f,\ \forall a\in\nx\} 
\]
consists only of $\nu_{1,\primes}$-a.e. constant functions. We denote the projection of $L^2(X,\nu_{1,\primes})$ onto $H$ by $P$.

We make two approximations using finite subsets of $\primes$. First, since $\mfd = \overline{\bigcup_{B\Subset\primes} \mfd_B}$, the union of the subspaces $\iota_B^*(L^2(X^B, \nu_{1,B}))$ over all finite $B\Subset \primes$ is dense in $L^2(X,\nu_{1,\primes})$; moreover, the subspaces $\iota_B^*(L^2(X^B, \nu_{1,B}))$ are invariant under the action of $S_a$ and $S_a \iota_B^* = \iota_B^* S_a$. Thus, in order to conclude that $H$ consists only of $\nu_{1,\primes}$-a.e. constant functions, it suffices to show that  $Pf$  is constant for $f\in L^2(X^B,\nu_{1,B})$. Second, for each finite subset $A\Subset \primes$  let $H_A$ denote the space of $\nx_A$-invariant functions in $L^2(X,\nu_{1,\primes})$ and $P_A$ denote the projection onto $H_A$. Then $P_A\rightarrow P$ in the strong operator topology as $A\nearrow \primes$. 

 For each $f\in L^2(X, \nu_{1,\primes})$  we use the decomposition from \lemref{lem:Xdecomposition}(4) to define an $\nx_A$-invariant function $f_A$ on $X$ by  setting
\[
f_A(w) = f_A (m\cdot x) = f_A (x) :=\frac{1}{\zeta_A(1)}\sum_{n\in\nx_A} n^{-1} S_n(f)(x),\qquad \text{if }w=m\cdot x \in \nx_A \cdot W_A,
\]
and letting $f_A (w) =0$ for $w \in X\setminus \bigsqcup_m  m\cdot W_A$, which is a $\nu_{1,\primes}$-null set  by \lemref{lem:Xdecomposition}. If $g\in H_A$, then 
\begin{align*}
\<{f,g} &= \sum_{n\in\nx_A}\int_{n\cdot W_A} f(x) \overline{g(x)} d\wt{\nu}_{1,\primes}(x) =
\sum_{n\in\nx_A} n^{-1}\int_{W_A} S_n(f)(x) \overline{g(x)} d\wt{\nu}_{1,\primes}(x)\\
  & =   \int_{W_A} \sum_{n\in\nx_A} n^{-1}S_n(f)(x) \overline{g(x)} d\wt{\nu}_{1,\primes}(x)   
  =\int_{W_A} \zeta_A(1)  f_A(x) \overline{g(x)} d\wt{\nu}_{1,\primes}(x) \\
& = \Big(\sum_{n\in\nx_A} n^{-1}\Big) \int_{W_A}  f_A(x) \overline{g(x)} d\wt{\nu}_{1,\primes}(x)= \sum_{n\in\nx_A} \int_{n\cdot W_A} f_A(x) \overline{g(x)} d\wt{\nu}_{1,\primes}(x) = \<{f_A, g}.
\end{align*}
Since this holds for every $g\in H_A$, we conclude that $P_A f = f_A$ $\nu_{1,\primes}$-a.e..

Now for $B\subseteq A \Subset \primes$, let $T:\HH_B\rightarrow \ell^2(X_B,\nu_{1,B})$ denote the unitary intertwiner for the GNS representation of \lemref{lem:local-GNS}, cf. \lemref{lem:coisometries}. For $\ell\in \Z$, let $\chi_\ell$ denote the function $T(\zed^\ell \delta_{n,1})_{n\in\nx_B}$ where $\zed:\T\rightarrow \C$ is the inclusion function. Then, by \lemref{lem:coisometries}:
\begin{align*}
P_A \chi_\ell(x) &= \frac{1}{\zeta_A(1)} \sum_{n\in\nx_A}n^{-1} S_n(\chi_\ell)(x)\\[5pt]
&= \frac{1}{\zeta_A(1)} \sum_{(n,m)\in \nx_B\times\nx_{A\backslash B}} (nm)^{-1} T(\zed^{\ell m}\delta_{1,nk})_{k\in\nx_B}(x)\\[5pt]
&= T \left(\delta_{1,k}\frac{1}{\zeta_A(1)}\sum_{m\in\nx_{A\backslash B}} m^{-1} z^{\ell m}\right)_{k\in\nx_B}(x)\\[5pt]
&= T \left(\delta_{1,k} \frac{F_{A\backslash B}(1,z^\ell)}{\zeta_A(1)}\right)_{k\in\nx_B}(x).
\end{align*}
In $L^2(\T, A_{1,B}\nu)$, this function satisfies
\begin{align*}
\<{F_{A\backslash B}(1,z^\ell)\zeta_A(1)^{-1}, \zed^k} &= \sum_{n\in \nx_B} \mu(n)n^{-1} \int_\T \frac{F_{A\backslash B}(1, z^{n\ell})}{\zeta_A(1)} z^{-nk}d\nu\\[5pt]
&= \sum_{n\in \nx_B} \mu(n)n^{-1} \frac{1}{\zeta_A(1)} \sum_{m\in \nx_{A\backslash B}} m^{-1} \wh{\nu}(n(\ell m + k)).
\end{align*}
By \proref{pro:leb-extreme}, the sequence $P_{\primes_n} \chi_\ell$ converges weakly to 0 as $n\rightarrow \infty$ for every $\ell\neq 0$; more generally, since $\{\primes_n\}_{n=1}^\infty$ is cofinal in $\primes$ and $P_A\leq P_{A'}$ when $A'\subseteq A$, the net $P_{A} \chi_\ell$ converges weakly to 0. This implies that $P_A\chi_\ell$ converges to 0 in norm (e.g. $\| P_A\chi_\ell\|^2 = \<{P_A\chi_\ell, \chi_\ell}\rightarrow 0$), and hence $P\chi_\ell = 0$ for $\ell\neq 0$.
Since the functions $\{\alpha_n(\chi_\ell):  \ell\in\Z, \ n \in\nx_B \}$ span a dense subspace of $L^2(X^B,\nu_{1,B})$,
  we conclude that $P$ is the projection onto the subspace of constant functions.
\end{proof}

\begin{theorem}\label{thm:leb-extreme}
For each $\beta \in [0,1]$, normalized Lebesgue measure $\lambda$ on $\T$ is the only nonatomic  $\beta$-subconformal  probability measure on $\T$, and the corresponding \kmsb state $\psi_{\beta,\lambda}$  is a factor state.
 \end{theorem}
\begin{proof}

We first consider the case $\beta=1$. Let $\nu$ be a nonatomic  $1$-subconformal  probability measure on $\T$.
Then  the action of $\qx$ on $(\tilde X,\wt{\nu}_{1,\primes})$ is ergodic by \proref{pro:ergodic}. This implies that the group-measure space von Neumann algebra of $(\wt{X}, \wt{\nu}_{1,\primes},\qx)$ is a factor; it is also the von Neumann factor generated by $\wt{\mfd}\rtimes \qx$ in the GNS representation for the weight $\wt{\psi}_{1,\nu}$ (cf. Appendix A).
Since $\tnxxz\cong \mfd \rtimes_\alpha \nx$ is a full-corner in $\wt{\mfd}\rtimes \qx$ we conclude that the GNS representation of $(\tnxxz, \psi_{1,\nu})$ is a factor representation, and this implies that $\psi_{1,\nu}$ is an extremal  \KMS{1} state. 
By \thmref{thm:KMScharact}, this says that $\nu$ is extremal among $1$-subconformal probability measures.

For uniqueness, observe that if $\nu$ is a $1$-subconformal nonatomic probability measure, then the  barycenter  $(\nu +\lambda)/2$   is 
$1$-subconformal and nonatomic, hence the associated \KMS{1} state $(\psi_{1,\nu}+\psi_{1,\lambda})/2$ is a factor. But this implies $\psi_{1,\nu} =\psi_{1,\lambda}$, and thus $\nu = \lambda$.

Suppose now  that  $\nu$ is $\beta$-subconformal for some $\beta\in[0,1)$. Since subconformality improves as $\beta$ increases by \cite[Corollary 9.5]{ALN20},  it follows that $\nu$ is also $1$-subconformal. The result then follows from uniqueness of the $1$-subconformal nonatomic measure.
\end{proof}

\section{KMS states at supercritical temperature}\label{sec:proofofmainthm}
The results of the previous three sections  can now be combined to prove one of our main goals, namely the parametrization of the \kmsb states at each fixed inverse temperature $\beta \in [0,1]$ by measures on $\T$ as stated in the introduction.
\begin{prop}[\thmref{thm:main} without the type $\mathrm{III}_1$ assertion]\label{pro:kms-parametrized}
Suppose $\beta \in (0,1]$. For each $n\in \nx$  let $\psi_{\beta,n} := \psi_{\beta, \nu_{\beta,n}}$ be the \kmsb state arising from the atomic extremal $\beta$-subconformal measure $\nu_{\beta,n}$ defined in \eqref{eqn:definitionnubetan}, and let $\psi_{\beta,\infty} := \psi_{\beta,\lambda}$ be the \kmsb state arising from normalized Lebesgue measure $\lambda$. Then the states $\psi_{\beta,n}$ satisfy the characterizations given in \thmref{thm:main} (a) and (b), and the map  $n \mapsto \psi_{\beta,n}$  is a homeomorphism of the one-point compactification $\nx \sqcup\{\infty\}$ to the extremal boundary $\partial_e K_\beta$. 

Suppose $\beta =0$. Then there are exactly two extremal \KMS{0} states $\psi_{0,1}$ and $\psi_{0,\infty}$; they are given by
\[
\psi_{0,1} (V_a U^k V_b^*) = \delta_{a,b} \quad \text{and} \quad \psi_{0,\infty} (V_a U^k V_b^*) =   \delta_{a,b} \delta_{k,0}.
\]
\end{prop}

\begin{proof}
By \proref{pro:bsubcequivalence}(5), an extremal $\beta$-subconformal probability measure on $\T$ is either atomic, in which case it is one of the measures $\nu_{\beta,n}$ by \thmref{thm:atomic}, or else nonatomic, in which case it is Lebesgue measure by \thmref{thm:leb-extreme}. By \thmref{thm:KMScharact}, the extremal boundary of the simplex of \kmsb states of $\tnxxz$ is  $\{\psi_{\beta,n}: n \in \nx \sqcup \{\infty\}\}$.

Next we show  that the states $\psi_{\beta,n} :=\psi_{\beta, \nu_{\beta,n}}$ are indeed  characterized  by  the formula given in equation \eqref{eqn:finite-states}. The Fourier coefficients of $\nu_{\beta,n}$ are given by
\begin{align*}
\int_\T \zed^k d\nu_{\beta,n} &= \int_\T \zed d\omega_{k*}\nu_{\beta,n} \\
&=\int_\T \zed d\nu_{\beta,\frac{n}{\gcd(n,k)}} \\
&= \big(\tfrac{n}{\gcd(n,k)}\big) ^{-\beta}\sum_{d|\frac{n}{\gcd(n,k)}} \frac{\euler_\beta(d)}{\euler(d)} \sum_{\ord(z) = d} z\\
&=\big(\tfrac{n}{\gcd(n,k)}\big)^{-\beta} \sum_{d|\frac{n}{\gcd(n,k)}}\mu(d) \frac{\euler_\beta(d)}{\euler(d)}.
\end{align*}
where the second equality holds by \lemref{lem:div} and the last one holds because the sum of the primitive roots of unity of order $d$ is the M\"obius function $\mu(d)$.
Using \eqref{eqn:psiformula}, we see that  $\psi_{\beta,n} $ satisfies \eqref{eqn:finite-states}. That $\psi_{\beta,\infty}$ satisfies \eqref{eqn:infinite-state} is immediate from the Fourier coefficients of $\lambda$.

Since the measures $\{\nu_{\beta,n} : n \in \nx\}$ are atomic and have distinct  supports, the set $\{\psi_{\beta,n} : n \in \nx\}$ is discrete in the subspace weak*-topology of $\partial_e K_\beta$. Next set $k\neq 0$ and observe that
$$|\psi_{\beta,n}(V_aU^kV_a^*)| = a^{-\beta}\left(\frac{n}{\gcd(n,k)}\right)^{-\beta} \prod_{p|\frac{n}{\gcd(n,k)}} 1- \frac{\euler_\beta(p)}{\euler(p)} < a^{-\beta} \left(\frac{n}{\gcd(n,k)}\right)^{-\beta},$$
because $\euler_\beta(p) \leq \euler(p)$ for $\beta\leq 1$ and hence $0\leq 1 - \frac{\euler_\beta(p)}{\euler(p)} < 1$. Since the right hand side converges to $ 0 = \psi_{\beta,\infty}(V_a U^k V_a^*) $ as $n\rightarrow \infty$ we conclude that $\psi_{\beta,n} \to \psi_{\beta,\infty}$ in the weak*-topology.

Now suppose that $\psi$ is an extremal \KMS{0} state and let $\nu = \nu_\psi$ be the corresponding extremal $0$-subconformal measure on $\T$. If $\nu$ is atomic, then \proref{pro:finitesupport} implies that $\nu=\delta_1$, while if $\nu$ is nonatomic, then \thmref{thm:leb-extreme} implies that $\nu =\lambda$. Therefore, the measures $\delta_1$ and $\lambda$ are the only extremal $0$-subconformal measures and we see that $\psi_{0,\delta_1} = \psi_{0,1}$ and $\psi_{0,\lambda} = \psi_{0,\infty}$. By \thmref{thm:KMScharact}, these are all extremal \KMS{0} states.
\end{proof}

\begin{remark}
The proof of surjectivity of the parameterization of \kmsb states for the system studied in \cite{LR-advmath} made use of the  \emph{reconstruction formula} \cite[Lemma 10.1]{LR-advmath}. One may ask how (or indeed, if) the reconstruction formula is related to surjectivity for \kmsb states of our system.

Suppose $\psi$ is a \kmsb state of $\tnxxz$ and let $F \Subset \primes$.
  If $\zeta_F(\beta) <\infty$ there is a conditional state $\psi_{e_F}$ defined by
  $\psi_{e_F} (\cdot) := \zeta_F(\beta) \psi(e_F \ \cdot \ e_F) $
  and $\psi$ can be reconstructed from $\psi_{e_F}$ via
  \[
  \psi(T) = \sum_{a\in \nx_F} \frac{a^{-\beta}}{\zeta_F(\beta)} \psi_{e_F} (V_a^* T V_a ).
  \]
  In particular, if  we set $T = U^n$ and $\psi = \psi_{\beta,\nu}$, then 
   \[
  \psi(U^n ) = \sum_{a\in \nx_F} \frac{a^{-\beta}}{\zeta_F(\beta)} \psi_{e_F} (U^{an} )= \sum_{a\in\nx_F} \frac{a^{-\beta}}{\zeta_F(\beta)} \int_\T \zed^{an} dA_{\beta,F} \nu = \int_\T \frac{\zeta_F(\beta,z^n)}{\zeta_F(\beta)} dA_{\beta,F} \nu.
  \]
Here we see a key difference between the reconstruction formulas for $\tzxnx$ and $\tnxxz$: the one in \cite{LR-advmath} is a sum of ratios $\frac{n}{a}$, for $a|n$, while the one obtained above is a sum of multiples $an$. That the set of divisors is finite when $n\neq 0$ leads to uniqueness of the \kmsb state on $\tzxnx$ for $\beta\in[1,2]$, while for the \kmsb states of $\tnxxz$ for $\beta\in[0,1]$, we need an estimate (such as that appearing in \proref{pro:leb-extreme}) for uniqueness of the \kmsb state corresponding to a nonatomic $\beta$-subconformal measure.
\end{remark}

 We now turn our attention to an interesting question that arises from the parametrization.  What happens to a \kmsb state when the temperature is varied? This phenomenon, sometimes referred to as {\em persistence} of a phase, has been discussed, for the very low-temperature range $\beta \to \infty$ by Connes, Consani, and Marcolli \cite{CCM}, see also \cite{BLT}. We wish to explore here what happens when the system  passes from super- to subcritical temperature and back.

Recall that for $\beta>1$, the operator defined by $T_\beta := \frac{1}{\zeta(\beta)}\sum_{n=1}^\infty n^{-\beta}\omega_{n*}$ is an affine isomorphism between $\mt_1^+$ and the simplex of $\beta$-subconformal measures on $\T$.

\begin{proposition}
For $n\in \nx$ and $\beta\in[0,\infty)$, let $\nu_{\beta,n}$ be the $\beta$-subconformal probability measure given by \eqref{eqn:definitionnubetan}. Then the map $\beta\mapsto \nu_{\beta,n}$ is continuous and, for $\beta>1$, $\nu_{\beta,n} = T_\beta \ve_n$.
\end{proposition}

\begin{proof}
For any $f\in C(\T)$, $\|f\|\leq 1$ and $\beta,\beta'\geq0$, we have
\begin{align*}\left\| \int_\T f(z) d(\nu_{\beta,n} - \nu_{\beta',n})\right\| &= \left|n^{-\beta} \sum_{d|n} \frac{\euler_\beta(d)}{\euler(d)} \sum_{z\in Z_d^*} f(z) - n^{-\beta'} \sum_{d|n} \frac{\euler_{\beta'}(d)}{\euler(d)} \sum_{z\in Z_d^*} f(z)\right|\\
&\leq \sum_{d|n} \left|n^{-\beta} \euler_\beta(d) - n^{-\beta'} \euler_{\beta'}(d)\right|
\end{align*}
Since the function $\beta\mapsto n^{-\beta} \euler_\beta(d)$ is uniformly continuous for $\beta\geq 0$, it follows that $\beta\mapsto \nu_{\beta,n}$ is continuous in the norm topology. 

Now suppose that $\beta>1$ and let $B$ denote the set of primes dividing $n$. Recall that $\ve_n$ denotes the atomic probability measure on $\T$ that is equidistributed on the primitive $n^{\mathrm{th}}$ roots of unity. The measure $\ve_n$ is invariant under $\omega_{c*}$ when $\gcd(n,c) =1$, so it follows that
\begin{align*}
T_\beta \ve_n &= \frac{1}{\zeta(\beta)}\sum_{c\in \nx} c^{-\beta} \omega_{c*} \ve_n = \frac{1}{\zeta(\beta)}\left(\sum_{k\in \nx_{\primes\backslash B}} k^{-\beta}\right)\sum_{c\in \nx_B} c^{-\beta} \omega_{c*} \ve_n \\
&= \frac{1}{\zeta_B(\beta)} \sum_{c\in\nx_B} c^{-\beta} \omega_{c*} \ve_n = \prod_{p|n} (1-p^{-\beta})A_{\beta,n}^{-1}\ve_n.
\end{align*}
This is the definition of $\nu_{\beta,n}$ from \eqref{eqn:definitionnubetan}.
\end{proof}

\begin{theorem}\label{theorem:limits-at-1}
The weak* limit $T_1\delta_z:= \lim_{\beta\rightarrow 1^+} T_\beta \delta_z$ exists for every $z\in\T$. If $z$ is a primitive $n^\mathrm{th}$ root of unity, then $T_1\delta_z$ is the measure $\nu_{n,1}$ from \thmref{thm:atomic}. If $z$ has infinite order in $\T$ (i.e. is `irrational'), then $T_1\delta_z$ is  normalized Haar measure on $\T$.
\end{theorem}

\begin{proof}

 Let $\zeta(\beta,a) = \sum_{n=0}^\infty (n+a)^{-\beta}$ be the Hurwitz zeta function with parameter $a\in(0,1]$. If $z$ is a primitive $n^\mathrm{th}$ root of unity, then
\[
  T_\beta \delta_z = \frac{1}{\zeta(\beta)} \sum_{c=1}^\infty c^{-\beta} \delta_{z^c}= \frac{1}{\zeta(\beta)} \sum_{k=1}^n \sum_{\substack{c=1\\ n | (c-k)}}^\infty c^{-\beta} \delta_{z^k} = \frac{n^{-\beta}}{\zeta(\beta)} \sum_{k=1}^n \zeta\left(\beta, \frac{k}{n}\right)\delta_{z^k}.
\]
The functions $\zeta(\beta,a)$ and $\zeta(\beta)$  have simple poles with residue $1$ at $\beta = 1$, \cite[Section 12.5]{Apo-76}, so 
\[
 \lim_{\beta\rightarrow 1^+} T_\beta \delta_z = \frac{1}{n}\sum_{k=1}^n \delta_{z^k},
 \]
 which is equal to $\nu_{1,n}$ by \eqref{eqn:nunbdefinition}.

Now suppose that $z$ has infinite order in the group $\T$ and let $n\in\Z$; then
$$T_\beta \delta_z (\zed^n) = \frac{1}{\zeta(\beta)}\sum_{c=1}^\infty c^{-\beta} \delta_{z^c}(\zed^n)= \frac{1}{\zeta(\beta)} \sum_{c=1}^\infty c^{-\beta}z^{cn} = \frac{F(\beta,z^n)}{\zeta(\beta)}.$$
If $n\neq 0$, then  $z^n\neq 1$ so that $F(\beta,z^n)$ is conditionally convergent for $\beta>0$ (cf. \cite[Section 12.7]{Apo-76}). Thus  $\lim_{\beta \to 1^+}T_\beta \delta_z(\zed^n) = 0$. If $n=0$, then $F(\beta,1) = \zeta(\beta)$ and the ratio is constant equal to $1$. Therefore, we conclude that the w*-limit $T_1\delta_z$ exists and is equal to normalized Haar measure.
\end{proof}

\section{Equivariant quotients}\label{sec:eqv-quot}

In this section, we consider a family of $\sigma$-invariant ideals associated to closed subsets of extremal \kmsb states. For $\beta\in (0,1]$, the space of extremal \kmsb states $\partial_e K_\beta$ is described in \thmref{thm:main} (cf. \proref{pro:kms-parametrized}): a closed subset either contains $\psi_{\beta,\infty}$ or is a finite subset of $\{\psi_{\beta,n}:n\in\nx\}$. To each closed set $F\subseteq \partial_e K_\beta$, we associate the ideal $J_F = \bigcap_{\psi\in F} \ker\pi_\psi$, where $\pi_\psi$ is the GNS representation of $\psi$. Since $J_F$ is $\sigma$-invariant, the quotient map $q:A\rightarrow A/J_F$ is $\sigma$-equivariant and determines an embedding of simplices $q^*:K_\beta(A/J_F)\rightarrow K_\beta(A)$. This embedding takes extreme points to extreme points; the image of $\partial_e K_\beta(A/J_F)$ contains $F$ but may be larger, in general.

If $F = \{\psi\}$ is a singleton set, then $J_F$ is simply $\ker \pi_\psi$. In \proref{prop:bc-subalg} we will characterize the quotient $\tnxxz/\ker\pi_\psi$ for $\psi = \psi_{\beta,n}$. In this section, we consider the distinguished family of closed subsets  $F_{\beta,n} = \{\psi_{\beta,d}: d|n\}$ using the multiplicative order on $\nx$.

\begin{prop}\label{prop:ideals}
Let $\beta\in (0,\infty)$. For each $n\in \nx$ let $\pi_{\beta,n}$ denote the GNS representation of $\psi_{\beta,n}$. Then $\ker\pi_{\beta,n}$ does not depend on $\beta$ and 
$$J_{F_{\beta,n}} = \bigcap_{d|n} \ker\pi_{\beta,d} = \<{U^{n} - 1}.$$
The GNS representation of $\psi_{\beta,\infty}$ is faithful. 
\end{prop}

We will show that the quotient of $\tnxxz$ by $\<{U^n-1}$ has another interpretation as the Toeplitz algebra of the monoid  $\nxx{n}$ consisting of pairs $(a,[x])$, where $a \in \nx$ and $[x]$ is the class of $x\in \Z$ modulo $n\Z$, with the binary operation
\[
(a,[x])\cdot(b,[y]) = (ab, [bx+y]).
\]

\begin{prop}\label{prop:nxx{n}-toeplitz}
The monoid $\nxx{n}$ is left-cancellative and right LCM. The quotient map $\nxxz\rightarrow \nxx{n}$ induces a surjective *-homomorphism $q_n:\tnxxz\rightarrow \tnxx{n}$ that satisfies
$$q_n:V_a\mapsto T_{(a,[0])},\quad U\mapsto T_{(1,[1])}$$
and which is  equivariant for the dynamics $\sigma_t(T_{(a,[x])}) = a^{it} T_{(a,[x])}$. 
\end{prop}

\begin{proof}
To see that $\nxx{n}$ is left-cancellative, suppose that $(b,[y])\cdot(a,[x]) = (b,[y])\cdot(c,[z])$ for some $a,b,c\in\nx$, $x,y,z\in\Z$; that is,
\begin{align*}
(ab, [ay + x]) &= (cb, [cy + z]).
\end{align*}
Then $a = c$ by cancellation in $\nx$, and thus $[x] = [z + cy -ay] = [z]$. Therefore, $(a,[x]) = (c,[z])$. (Notice that $\nxx{n}$ is not right-cancellative because, for instance,   $(1,[1])\cdot(n,[0]) = (n, [0]) = (1,[0])\cdot(n,[0])$. )

The right-multiples of an element $(a,[x])$ are of the form $(ac, [cx + z])$ for $c\in \nx,$ $z\in \Z$, or equivalently, of the form $(ac, [z])$. We conclude that the common multiples of $(a,[x])$ and $(b,[y])$ are of the form $(\lcm(a,b)c, [z])$, which are the right-multiples of $(\lcm(a,b), [0])$. This shows that $\nxx{n}$ is right LCM.

That there is an equivariant homomorphism of C*-algebras follows easily on noticing that the generators $T_{(a,[0])}$ and $ T_{(1,[1])}$
of $\tnxx{n}$ satisfy the relations (AB0)--(AB3) from \proref{pro:presentationtnxxz} and that the dynamics match on corresponding generators.
\end{proof}

\begin{prop}\label{prop:relations-modn}
For each $n\in\nx$ the Toeplitz C*-algebra $\tnxx{n}$ is generated by elements $R := T_{(1,[1])}$ and $V_a := T_{(a,[0])}$ for  $a\in\nx$, which satisfy the following conditions:
\begin{enumerate}
\smallskip\item[\textup{(N0)}] \ $V_a^* V_a = 1 = R R^* = R^* R$,
\smallskip\item[\textup{(N1)}] \ $R V_a = V_a R^a$,
\smallskip\item[\textup{(N2)}] \ $V_a V_b = V_{ab}$,
\smallskip\item[\textup{(N3)}] \ $V_a V_b^* = V_{b}^* V_a$ when $\gcd(a,b) = 1$,
\smallskip\item[\textup{(N4)}] \ $R^n = 1$.
\end{enumerate} 
Moreover, the relations \textup{(N0)--(N4)} constitute a presentation of $\tnxx{n}$. In particular, the following sequence is exact,
\begin{center}
% https://tikzcd.yichuanshen.de/#N4Igdg9gJgpgziAXAbVABwnAlgFyxMJZABgBpiBdUkANwEMAbAVxiRGJAF9T1Nd9CKAIzkqtRizYAdKQB5gAVQB6YALRDOXHiAzY8BIgCZR1es1aIQMnGAAetgF5beegUQDMJ8eelSb94DBNbhd+AxQAFi8zSUsOTjEYKABzeCJQADMAJwgAWyQyEBwIJA1tbLzS6mKkQxCQCvzEYyKSxE9vWJAARwB9QnrGpA6axAiEziA
\begin{tikzcd}
0 \arrow[r] & \<{U^n-1} \arrow[r] & \tnxxz \arrow[r, "q_n"] & \tnxx{n} \arrow[r] & 0.
\end{tikzcd}
\end{center}
\end{prop}

\begin{proof}
It is easily verified that $R$ and $V_a$ generate $\tnxx{n}$ and satisfy (N0)--(N4), so we turn our attention to the second claim. Let $C^*(r,v)$ be the universal C*-algebra generated by elements $r$ and $\{v_a: a\in\nx\}$ satisfying the relations {\rm (N0)--(N4)}. By the universal property, there is a surjective homomorphism $\pi:C^*(r,v)\rightarrow \tnxx{n}$; we need to show it is also injective.

Relations (N0)--(N4) imply that the collection $\{v_ar^x v_b^*: a,b\in\nx, x\in\Z/n\Z\}$ 
is closed under multiplication and adjoints and contains the generating elements, whence 
\[C^*(r,v) = \clsp\{v_ar^x v_b^*: a,b\in\nx, x\in\Z/n\Z\};
\]
similarly
\[
\tnxx{n} = \clsp\{V_a R^x V_b^*: a,b\in\nx, x\in\Z/n\Z\}.
\]

For each character $\chi\in\qd$, the elements $r$ and $\{\chi(a) v_a: a\in\nx\}$ satisfy conditions (N0)--(N4), so there exists $\theta_\chi\in\Aut(C^*(r,v))$ satisfying
$$\theta_\chi(r) = r,\quad \theta_\chi(v_a) = \chi(a) v_a,$$
and $\theta_\chi\circ\theta_{\chi_1} = \theta_{\chi\chi_1}$ for any $\chi_1\in\qd$. Further, an approximation argument shows that the map $\chi \mapsto \theta_\chi(a)$ is continuous for each $a\in C^*(r,v)$. By integration with respect to normalized Haar measure on  $\qd$ we get a faithful conditional expectation $E:C^*(r,v)\rightarrow C^*(r,v)^\theta=:\mfc$ onto the fixed point algebra,  determined by 
$$E(v_a r^x v_b^*) = \delta_{a,b} v_a r^x v_a^*.$$
Since $E$ is contractive, $$\mfc = \clsp\{v_a r^x v_a^*: a\in\nx, x\in\Z/n\Z\}.$$

As customary, a spatial argument gives the analogous result  at the level of the reduced Toeplitz algebra. Specifically, for each $\chi\in\qd$, there is a unitary $Q_\chi$ on $L^2(\nx\!\times(\Z/n\Z))$ defined on the canonical basis by $Q_\chi \delta_{(a,[x])} = \chi(a)\delta_{(a,[x])}$. Then $\tnxx{n}$ is invariant under conjugation by $Q_\chi$, which yields a representation $\theta^\circ_\chi\in \Aut(\tnxx{n})$. Integrating over $\qd$ yields a faithful conditional expectation $E^\circ:\tnxx{n}\rightarrow \tnxx{n}^{\theta^\circ} =:\mfc^\circ$ and evaluating this on monomials yields
$$E^\circ(V_a R^x V_b^*) = \delta_{a,b} V_a R^x V_b^*.$$
Thus 
$$\mfc^\circ = \clsp \{V_a R^x V_a^*: a\in\nx, x\in\Z/n\Z\}$$
and, moreover, $\pi\circ E = E^\circ \circ \pi$. Since $E$ and $E^\circ$ are faithful (as positive maps), a standard argument shows that $\pi$ is faithful if and only if $\pi|_{\mfc}$ is faithful.

Recall that for each $a\in\nx$ the set of divisors of $a$ is denoted by $\Delta_a$. Consider the  subalgebras of $C^*(r,v)$ defined by
\[
\mfc_a := \clsp\left\{v_b r^x v_b^*: b\in\Delta_a, x\in\Z/n\Z\right\}
\]
and their upper-case analogues $\mfc_a^\circ$. 
Then 
\begin{align*}
\mfc = \lim_{a\in\nx} \mfc_a\qquad\text{and}\qquad \mfc^\circ = \lim_{a\in\nx} \mfc^\circ_{a},
\end{align*}
with $\nx$ ordered by division. Thus, faithfulness of $\pi|_\mfc$ is equivalent to faithfulness of $\pi|_{\mfc_{a}}$ for each $a\in\nx$. Since $\pi$ maps generators of $\mfc_{a}$  bijectively to generators of $ \mfc_{a}^\circ$, and since the latter are linearly independent (which can be verified by computing on 
$\ell^2(\nxx{n})$), the restriction $\pi|_{\mfc_{a}}$ is indeed an isomorphism of $\mfc_{a}$ onto $ \mfc_{a}^\circ$.

Compared to \proref{pro:presentationtnxxz}, we see that conditions (N0)--(N3) and (AB0)--(AB3) are identical under the identification $U \equiv R$. The quotient map $q_n:\tnxxz\rightarrow\tnxx{n}$ is then determined by the last relation (N4), so that $\ker(q_n) = \<{U^n-1}$.
\end{proof}

\begin{lemma}\label{lem:finite-pull-back}
For $\beta\in (0,\infty)$, let $\nu$ be a $\beta$-conformal measure. For each $n\in\nx$, $\psi_{\beta,\nu}$ factors through $q_n$ if and only if $\nu$ is supported on the set of $n^{\mathrm{th}}$ roots of unity $Z_n$.
\end{lemma}

\begin{proof}
By \proref{prop:relations-modn}, we have 
\[
\ker(q_n) = \<{U^{n} - 1} = \clsp \{A(U^n-1)B \mid A,B \in \tnxxz\}.
\]
Approximating the elements $A$ and $B$  with linear combinations of the spanning monomials $V_a U^y V_b^*$ and then using \eqref{eqn:D-prod} to reduce the products, we see that
\[
\ker(q_n) = \clsp\{V_a (U^{xn+y} -U^y)V_b^*: a,b\in\nx,\ x,y\in\Z\}.
\]
If $\nu$ is supported on $Z_n$, then computing  the  \kmsb state $\psi_{\beta,\nu}$ at the spanning elements of $\ker(q_n)$ using  \eqref{eqn:psiformula} gives
\begin{equation}\label{eqn:KMSkernelbasis}
\psi_{\beta,\nu} (V_a (U^{xn+y}-U^y )V_b^*) = \delta_{a,b}a^{-\beta} \int_\T \zed^y( \zed^{nx} - 1) d\nu.
\end{equation}
Since $\nu$ is supported on $Z_n$, the right hand side of \eqref{eqn:KMSkernelbasis} vanishes because $\zed^{nk} - 1 \equiv 0$ on $Z_n$. This shows that $\psi_{\beta,\nu}$ vanishes on $\ker(q_n)$, as desired.

Suppose conversely that $\phi_{\beta,\nu}$ vanishes on $\ker q_n$; then the left hand side of  \eqref{eqn:KMSkernelbasis} vanishes, and setting $x=1$ gives
\[
 \int_\T \zed^y( \zed^{n} - 1) d\nu  = 0 \qquad \forall y\in \Z.
\] 
This says that  all the Fourier coefficients of the complex measure $( \zed^{n} - 1) d\nu$ vanish, which implies $\supp (\nu) \subset Z_n$.
\end{proof}

For the proof of \proref{prop:ideals}, we require the following lemma on inductive limits of commutative C*-algebras

\begin{lemma}\label{lem:inductive-limit-order}
Let $A$ be a unital commutative C*-algebra and let $A_i,$ $i\in I,$ be unital subalgebras indexed by a directed set $I$ such that $A = \varinjlim A_i$. If $x\in A_+$ is nonzero, then there exists $i\in I$ and a nonzero $x_0\in A_{i+}$ such that $x_0\leq x$. If $\psi$ is a state on $A$ and $\psi(x)>0$, then $x_0$ can be chosen so that $\psi(x_0)>0$.
\end{lemma}

\begin{proof}
Let $X = \Spec A$, $X_i = \Spec A_i$. The homomorphisms $A_i\rightarrow A$ induce a unital surjective homomorphism $\bigotimes_{i\in I} A_i\rightarrow A$ which is dual to an embedding $\iota:X\rightarrow \prod X_i$. If $x\in A_+$ is nonzero, let $U_j$, $j=1,2$ be nonempty basic open subsets of $\prod X_i$ such that $U_j\cap \iota(X) \subseteq x^{-1}((\ve_j, \infty))$ for some $0<\ve_1<\ve_2$; if $\psi$ is a state on $A$ and $\psi(x)>0$, we also impose $\iota_*\nu_\psi(U_2)>0$ where $\nu_\psi$ is the probability measure on $X$ corresponding to $\psi$. Let $B\Subset I$ and $U_{i,j}\subset X_i$, $i\in B$, such that
$$U_j = \prod_{i\in B} U_{i,j} \times \prod_{i\in I\backslash B} X_i.$$
We write $p_B:\prod X_i\rightarrow \prod_{i\in B} X_i$ for the projection. Let $\bar{x}_0$ be a positive function on $\prod_{i\in B} X_i$ that is $0$ on $p_B(U_1^c)$ and $\ve_1$ on $p_B(\overline{U}_2)$, and let $x_0 = \iota^*(p_B^*(x_0))$. Then $x_0\in A_{i+}$ for $i = \bigvee B$ and $x_0\leq x$; additionally, $\psi(x_0) \geq \ve_1 \iota_*\nu_\psi(U_2) >0$, as desired.
\end{proof}

%%%%%%%%%%%%%%%%%%%%%%%%%%%%%%%%%%%%%%%%%%%%%%

\begin{proof}[Proof of \proref{prop:ideals}]
We start with faithfulness of $\psi_{\beta,\infty}$. Let $\Gamma_a:C(X_a)\cong \mfd_a$ be the isomorphism of \corref{cor:system-spectra}; for $b|a$ and nonzero $f\in C(\T)_+$, just as in the proof of \lemref{lem:prime-positivity}, we have
$$\psi_{\beta,\infty}(\Gamma_a(f^b)) = b^{-\beta} \int_\T f d\left(\sum_{d|\frac{a}{b}} \mu(d) d^{-\beta} \omega_{d*} \lambda\right).$$
Since $\omega_{d^*}\lambda = \lambda$, this is $b^{-\beta}\prod_{p|\frac{a}{b}} (1-p^{-\beta})\int_\T f d\lambda$, which is greater than 0 when $f$ is nonzero. Therefore, $\psi_{\beta,\infty}$ is faithful on $\mfd_a$ for every $a\in\nx$. By \lemref{lem:inductive-limit-order}, we conclude that $\psi_{\beta,\infty}$ is faithful on $\mfd = \varinjlim \mfd_a$.

Next we show that $\ker\pi_{\beta,n}$ does not depend on $\beta$. For $n\in \nx$, suppose that $x\notin \ker  \pi_{\beta,n}$, or equivalently, there exists some $u\in \tnxxz$ such that $\psi_{\beta,n}(u^* x^* x u) >0$. By \lemref{lem:inductive-limit-order}, there exists some $a\in \nx$ and nonzero $y\in \mfd_{a+}$ such that $y\leq E( u^* x^* x u)$ and $\psi_{n,\beta}(y) >0$. The measures $\{\nu_{n,\beta}: \beta\in (0,\infty)\}$ are equivalent on $\T$, so it follows that $0<\psi_{n,\beta'}(y)\leq \psi_{n,\beta'}(u^* x^* x u)$ for every $\beta'\in (0,\infty)$. Therefore, $x\notin \ker\pi_{\beta',n}$ and hence, the ideals are equal. This also implies that $J_{F_{\beta,n}}$ does not depend on $\beta$, so we denote it here by $J_n$.

The equality $\ker(q_n) = \<{U^n-1}$ was shown in \proref{prop:relations-modn}. For each $d|n$, the state $\psi_{\beta,d}$ vanishes on $\ker(q_n)$ by \lemref{lem:finite-pull-back}; since $\ker(q_n)$ is a two-sided ideal, this implies that $\pi_{\beta,d}$ vanishes on $\ker(q_n)$, so we have $\ker(q_n)\subseteq J_n$. It only remains to show that $J_n\subseteq \ker(q_n)$.

For $\beta\in (0,\infty)$, define $\psi = \frac{1}{d(n)} \sum_{d|n} \psi_{\beta,d}$. Since $\ker q_n\subseteq J_n$, there exists a state $\bar{\psi}$ on $\tnxx{n}$ satisfying $\psi = \bar{\psi}\circ q_n$; we will argue that $\bar{\psi}$ is faithful. Consider $I = \ker q_n \cap \mfd$ and $I_a = \ker q_n \cap \mfd_a$ for $a\in \nx$. Since $\ker q_n = \<{U^{n} - 1}$, clearly 
$$\clsp\{V_b (U^{xn + y} - U^y) V_b^*: b\in\nx,x,y\in\Z\}\subseteq I.$$
On the other hand, if $\{x_i\}_i$ is a net in $\lsp\{V_a (U^{xn + y} - U^y)V_b^*: a,b\in \nx,x,y\in\Z\}$ converging to $x\in \mfd$, then, since the expectation $E:\tnxxz\rightarrow \mfd$ from \proref{pro:condexptheta} is contractive, $E(x_i)$ is a net in $\lsp\{V_b (U^{xn + y} - U^y) V_b^*: b\in\nx,x,y\in\Z\}$ converging $E(x) = x$. Thus, 
$$I = \clsp\{V_b (U^{xn + y} - U^y) V_b^*: b\in\nx,x,y\in\Z\}.$$
It follows that $I = \varinjlim I_a$, and $\mfd/I = \varinjlim \mfd_a/I_a$. 

The isomorphism $\Gamma_a$ of \corref{cor:system-spectra} identifies $I_a$ with an ideal of $\bigoplus_{b|a} C(\T)$ which, by the Fourier transform, is given by $\bigoplus_{b|a} K_n$, where $K_n = \{f\in C(\T): f|_{Z_n} = 0\}.$ If $q_n(\Gamma_a(f^b)) \neq 0$ for some $f\in C(\T)_+$ and $b|a$ (that is, $f\notin K_n$), then  
\begin{align*}
\psi(\Gamma_a(f^b)) &= \frac{b^{-\beta}}{d(n)}\sum_{d|n}  \prod_{p|d} (1-p^{-\beta}) \int_\T fd\left(A_{\beta,\frac{a}{b}} A_{\beta,d}^{-1} \ve_d\right)\\
&=\frac{b^{-\beta}}{d(n)} \sum_{d|n}  \prod_{p|\frac{ad}{b}} (1-p^{-\beta}) \int_\T f d\left( A_{\beta,d'}^{-1} \ve_d\right)\\
&=\frac{b^{-\beta}}{d(n)} \int_\T f d\left( \sum_{d|n}  \prod_{p|\frac{ad}{b}} (1-p^{-\beta})A_{\beta,d'}^{-1} \ve_d\right)
\end{align*}
where $d'$ is the largest factor of $d$ that is relatively prime to $\frac{a}{b}$. Then $\ve_d$ is absolutely continuous with respect to $A_{\beta,d'}^{-1}\ve_d$, so uniform measure on $Z_n$ is absolutely continuous with respect to the measure of integration, and hence, $\psi(\Gamma_a(f^b))\neq 0$. Consequently, if $x\in \mfd_{a+}$ and $q_n(x)\neq 0$, then $\psi(x)\neq 0$, so that $\bar{\psi}$ is faithful on $\mfd_a/I_a$. By \lemref{lem:inductive-limit-order}, for each $x\in \mfd_+$, there is some $a\in\nx$ and $x_a\in \mfd_{a+}$ with $x_a\leq x$, so $\bar{\psi}$ is faithful on $\mfd/I$. Now for $x\in \tnxxz_+$ and $q_n(x)\neq 0$, since $q_n$ is equivariant for the action of $\qd$, we have $q_n(E(x)) = E(q_n(x)) \neq 0$. Thus, $\bar{\psi}\circ q_n(x) = \bar{\psi}\circ q_n(E(x)) \neq 0$, and therefore $\bar{\psi}$ is faithful on $\tnxx{n}$.

Finally, if $x\geq 0$ and $x\notin \ker(q_n)$, then $\psi(x) \neq 0$, so $x\notin J_n$. We conclude that $J_n\subseteq \ker(q_n)$.
\end{proof}

\begin{theorem}\label{thm:finite-kms}
Fix $n\in \nx$ and let $\sigma$
be the dynamics on $\tnxx n$ determined by  $\sigma_t(V_a) = a^{it} V_a$ and $\sigma_t(R) = R$. 
\begin{enumerate}%[(1)]
\item Suppose $\beta\in(1,\infty)$. For each $n^\mathrm{th}$ root of unity $z$, there is an extremal \kmsb state $\bar{\psi}_{\beta,z}$ of $\tnxx n$ determined by
\begin{equation}\label{eqn:finite-subcrit}
\bar{\psi}_{\beta,z}(V_a R^k V_b^*) = \delta_{a,b} \frac{a^{-\beta}}{\zeta(\beta)} \sum_{c\in\nx} c^{-\beta}z^{ck}.
\end{equation}
The map $z\mapsto \bar{\psi}_{\beta,z}$ extends to an affine w*-homeomorphism of the simplex of probability measures on $Z_n$ onto the simplex of \kmsb states of $(\tnxx n,\sigma)$.
\item Suppose $\beta\in(0,1]$. For each divisor $m$ of $n$, there is an extremal \kmsb state $\bar{\psi}_{\beta,m}$ of $\tnxx n$ determined by
\begin{equation}\label{eqn:finite-supcrit}
\bar{\psi}_{\beta,m}(V_a R^k V_b^*) = \delta_{a,b} a^{-\beta} \left(\frac{m}{\gcd(m,k)}\right)^{-\beta} \sum_{d|\frac{m}{\gcd(m,k)}} \mu(d) \frac{\euler_\beta(d)}{\euler(d)}.
\end{equation}
The map $m\mapsto \bar{\psi}_{\beta,m}$ extends to an affine w*-homeomorphism of the simplex of probability measures on $\Delta_n := \{m\in\nx: m|n\}$ onto the simplex of \kmsb states of $(\tnxx n,\sigma)$.
\end{enumerate}
\end{theorem}

\begin{proof}
Since $q_n$ is an equivariant surjective homomorphism, the map $\psi\mapsto \psi\circ q_n$ is an  injective continuous affine map from the simplex of \kmsb states on $\tnxx n$ to the \kmsb states in $\tnxxz$. By \thmref{thm:KMScharact} and \lemref{lem:finite-pull-back}, the range of this map is the finite-dimensional simplex of states of the form $\psi_{\beta,\nu} \circ\Phi,$ for $\nu$ a measure on $Z_n$ satisfying \eqref{eqn:positivitycondition}. Formula \eqref{eqn:finite-subcrit} for the extreme points when $\beta>1$ follows from \cite[Theorem 8.1]{aHLR21}, and formula \eqref{eqn:finite-supcrit} for the extreme points when $\beta\leq 1$ follows from our \thmref{thm:main}.
\end{proof}

\section{Symmetry and type}\label{sec:bc-symm}

Recall from \cite[Section 4]{bos-con} that the \emph{Bost--Connes C*-algebra} $\cq$ is canonically isomorphic to the universal C*-algebra with generators $\bcmu_a$ for $a\in\nx$ and $e(x)$  for $x\in\Q/\Z$ subject to the relations
\begin{enumerate}
\smallskip\item[(BC0)] \ $\bcmu_a^*\bcmu_a = 1$,
\smallskip\item[(BC1)] \ $e(x)\bcmu_a = \bcmu_a e(ax)$,
\smallskip\item[(BC2)] \ $\bcmu_a\bcmu_b = \bcmu_{ab}$,
\smallskip\item[(BC3)] \ $\bcmu_a\bcmu_b^*= \bcmu_b^* \bcmu_a$ when $\gcd(a,b)=1$,
\smallskip\item[(BC4)] \ $e(0) = 1$, $e(x)e(y) = e(x+y)$ and $e(x)^* = e(-x)$,
\smallskip\item[(BC5)] \ $\bcmu_a e(x)\bcmu_a^* = \frac{1}{a}\sum_{ay = x} e(y)$.
\end{enumerate}
In \cite[Lemma 2.7]{bcalg} it is shown that the two relations (BC1) and (BC3) are consequences of the other four, but it is important for us to keep the whole list here because it allows us to establish the connection with $\tnxxz$. The dynamics on $\cq$ are determined by
$$\sigma_t(e(x)) = e(x)\qquad\text{and}\qquad \sigma_t(\bcmu_a) = a^{it} \bcmu_a.$$
The system $(\cq,\sigma)$ has a unique \kmsb state for $\beta\in(0,1]$, which we denote by $\psi_\beta$.

\begin{prop}\label{pro:bc-hom}
For $x\in \Q/\Z$ and $n = \ord(x)$, there is an equivariant *-homomorphism $\phi_x:\tnxxz\rightarrow \cq$ determined by
\begin{equation}\label{eqn:bc-hom}
\phi_x(U) = e(x)\qquad\text{and}\qquad \phi_x(V_a) = \bcmu_a
\end{equation}
and $\phi_x$ factors through $q_n$; that is,  there is a (unique) map $\bar{\phi}_x:\tnxx{n}\rightarrow \cq$ such that the diagram commutes,
\begin{center}
\begin{tikzcd}
\tnxxz \arrow[rr, "\phi_x"] \arrow[rd, "q_n"] &                                     & \cq \\
                                              & \tnxx{n} \arrow[ru, "\bar{\phi}_x"] &    
\end{tikzcd}
\end{center}
For every $\beta\in(0,1]$, $\ker(\phi_x) = \ker(\pi_{\beta,n})$ and $\psi_{\beta,n}$ is the only \kmsb state of $\tnxxz$ that factors through $\phi_x$. Moreover, $\psi_{\beta,n} = \psi_\beta\circ \phi_x$. 
\end{prop}

\begin{proof}
Let $v_a$ and $e(x)$ denote the universal generators of $\cq$ and fix $x\in \Q/\Z$. Because of the relations  (BC0)--(BC4), the unitary $e(x)$ and the isometries $\{\bcmu_a: a\in\nx\}$ satisfy the relations (AB0)--(AB3) of  \proref{pro:presentationtnxxz}, so that there is a *-homomorphism $\phi_x$ such that \eqref{eqn:bc-hom} holds. Similarly, the $v_a$ and $e(x)$ also satisfy the relations (N0)--(N4) of \proref{prop:relations-modn}, so that there is a map $\bar{\phi}_x$. By universality of $q_n$, we have $\phi_x = \bar{\phi}_x\circ q_n$.

For $\beta\in (0,1]$, let $\psi_\beta$ be the unique \kmsb state of $(\cq,\sigma)$. Comparing \cite[Section 6 (14)]{bos-con} and \thmref{thm:main} \eqref{eqn:finite-states}, we have $\psi_\beta\circ \phi_x = \psi_{\beta,n}$ so that $\ker\phi_x\subseteq \ker\pi_{\beta,n}$. Since $\psi_\beta$ is faithful on $\cq$, we also have $\ker\pi_{\beta,n}\subseteq \ker\phi_x$.

From $\phi_x = \bar{\phi}_x\circ q_n$, the image of $\phi_{x*}$ is contained in the image of $q_{n*}$, so that by \thmref{thm:finite-kms} (2), $\psi_{\beta,m}$, $m\nmid n$ does not factor through $\phi_x$. For $m|n$, $m\neq n$, we will show that $\psi_{\beta,m}$ does not factor through (BC5), that is, $\psi_{\beta,m}(V_n V_n^*) \neq \psi_{\beta,m}\left(\frac{1}{n}\sum_{k=0}^{n-1} U^k\right)$. We have $\psi_{\beta,m}(V_n V_n^*) = n^{-\beta}$ and
\begin{align*}
\psi_{\beta,m}\left(\frac{1}{n}\sum_{k=0}^{n-1}U^k\right) &= \frac{m}{n} \sum_{k=0}^{\frac{n}{m}-1} \psi_{\beta,m}\left(\frac{1}{m}\sum_{\ell=0}^{m-1} U^{\ell+mk}\right) = \frac{m}{n}\sum_{k=0}^{\frac{n}{m}-1} m^{-\beta} = m^{-\beta} \neq n^{-\beta}.
\end{align*}
Therefore, $\psi_{\beta,m}$ does not factor through $\phi_x$ for $m\neq n$. Since $\phi_{x*}$ maps extremal \kmsb states to extremal \kmsb states, it follows that $\psi_{\beta,n}$ is the only \kmsb state that factors through $\phi_x$.
\end{proof}

An important feature of the Bost-Connes system is its \emph{group of symmetries}, which remarkably is isomorphic to the Galois group of the cyclotomic extension $\Gal(\Q^{\rm cycl}/\Q)$. We describe this action for completeness. The field $\Q^{\rm cycl}$ is the direct limit of the fields $\Q(e^{2\pi i /n})$; the Galois group of $\Q(e^{2\pi i /n})$ is isomorphic to $\Aut(n^{-1}\Z/\Z)$, which acts by $u\cdot e^{2\pi i k/n} = e^{2\pi i u(k/n)}$. We then identify $$\Gal(\Q^{\rm cycl}/\Q) = \varprojlim \Gal(\Q(e^{2\pi i /n})/\Q)\cong \varprojlim \Aut(n^{-1} \Z/\Z) = \Aut(\Q/\Z).$$
The action of $\Aut(\Q/\Z)$ on $\cq$ is determined by
\begin{equation}\label{eqn:bc-symm}
\theta_u(e(x)) = e(u(x))\qquad\text{and}\qquad \theta_u(\bcmu_a) = \bcmu_a.
\end{equation}
One verifies that $\{e(u(x)):x\in \Q/\Z\}$ and $\{\bcmu_a: a\in \nx\}$ satisfy (BC0)--(BC5), so that \eqref{eqn:bc-symm} determines an equivariant *-endomorphism of $\cq$ with inverse $\theta_{u^{-1}}$. 

In the spirit of \eqref{eqn:bc-symm}, we define a semigroup of endomorphisms on $\tnxxz$ as follows. For $b\in \N$, define
\begin{equation}\label{eqn:tnxxz-symm}
\kappa_b(U) = U^b\qquad\text{and}\qquad \kappa_b(V_a) = V_a.
\end{equation}
Then $U^b$ and $\{V_a:a\in \nx\}$ satisfy (AB0)--(AB3), so there exists a unique *-endomorphism of $\tnxxz$ satisfying \eqref{eqn:tnxxz-symm} that commutes with the dynamics $\sigma_t$. Since $\kappa_a\kappa_b = \kappa_{ab}$, for $a,b\in \N$, the map $b\mapsto \kappa_b$ is indeed a semigroup action of the multiplicative monoid $\N$ on $\tnxxz$. 

Similarly, for $n\in \nx$, the formulas
$$\bar{\kappa}_b(R) = R^b\qquad\text{and}\qquad \bar{\kappa}_b(V_a) = V_a$$
determine a *-endomorphism $\bar{\kappa}_b$ of $\tnxx{n}$. Obviously $\bar{\kappa}_b = \bar{\kappa}_{b'}$ if $b\equiv b'$ mod $n$, so in particular, we have an action of the unit group $(\Z/n\Z)^*$ on $\tnxx{n}$ by automorphisms that commute with the dynamics.

If $n$ and $b$ are relatively prime, then $b$ determines an automorphism of $n^{-1}\Z/\Z$ by $u_b([x]) = [bx]$. This gives an identification $(\Z/n\Z)^*\cong \Aut(n^{-1} \Z/\Z)$. Using the action $\bar{\kappa}$ and the projection $\Aut(\Q/\Z)\rightarrow \Aut(n^{-1}\Z/\Z)$ given by restriction, we obtain an action of $\Aut(\Q/\Z)$ on $\tnxx{n}$, which we also denote by $\bar{\kappa}$.

\begin{lemma}\label{lem:tnxxz-symm}
For $x\in \Q/\Z$, $n = \ord(x)$ and $b\in\N$, let $\phi_x$ and $\bar{\phi}_x$ be as in \proref{pro:bc-hom} and let $\kappa_b$ and $\bar{\kappa}_b$ be as in \eqref{eqn:tnxxz-symm}; for $u\in \Aut(\Q/\Z)$, let $\theta_u$ be as in \eqref{eqn:bc-symm}. Then:
\begin{enumerate}[label=(\roman*)]
\item $\phi_x\circ\kappa_b = \phi_{bx}$ and $\theta_u\circ\phi_x = \phi_{u(x)}$;
\item $\bar{\phi}_x$ is $\Aut(\Q/\Z)$-equivariant, i.e. for any $u\in \Aut(\Q/\Z)$, $\bar{\phi}_x\circ \bar{\kappa}_{u} = \theta_u\circ \bar{\phi}_x$;
\item $\psi_{\beta,n} \circ \kappa_b = \psi_{\beta,\frac{n}{\gcd(n,b)}}$.
\end{enumerate}
In particular, if $n$ and $b$ are relatively prime, then $\psi_{\beta,n}\circ\kappa_b = \psi_{\beta,n}$ and if $b=0$, then $\psi_{\beta,n}\circ\kappa_b = \psi_{\beta,1}$. 
\end{lemma}

\begin{proof}
(i) We have $\phi_x\circ \kappa_b(U) = e(bx) = \phi_{bx}$ and $\theta_u\circ \phi_x(U) = e(u(x)) = \phi_{u(x)}(U)$. By universality, the result follows.

(ii) For $u\in \Aut(\Q/\Z)$, let $\bar{u} = u_b$ be the image of $u$ in $\Aut(n^{-1}\Z/\Z)$. Then  $u(x) = bx$ and $\bar{\kappa}_u = \bar{\kappa}_b$, by definition. From (i) and universality of $q_n$, we have $\bar{\phi}_x\circ\bar{\kappa}_u = \bar{\phi}_{bx} = \theta_u\circ\bar{\phi}_x$, as desired.

(iii) Let $\psi_\beta$ be the unique \kmsb state on $\cq$. Then, by (i) and \proref{pro:bc-hom}, $\psi_{\beta,n}\circ\kappa_b = \psi_\beta\circ\phi_x\circ\kappa_b = \psi_\beta\circ\phi_{bx} = \psi_{\beta,m}$, where $m= \ord(bx) = \frac{n}{\gcd(b,n)}$. 
\end{proof}

\begin{prop}\label{prop:bc-subalg}
For $x\in \Q/\Z$ and $n = \ord(x)$, let $G_n$ be the kernel of the homomorphism $\Aut(\Q/\Z)\rightarrow (\Z/n\Z)^*$. Then the image of $\phi_x$ is $\cq^{G_n}$; consequently, the following sequence is exact,
\begin{center}
% https://tikzcd.yichuanshen.de/#N4Igdg9gJgpgziAXAbVABwnAlgFyxMJZABgBpiBdUkANwEMAbAVxiRGJAF9T1Nd9CKAIzkqtRizYAdKQGsYAJxlosAfWAyARjBx1SYTlx4gM2PASIAmUdXrNWiEDJxgAHq4BeR3mYFEAzDbi9tJSAMYAjgB6wADiqgbeJnzmgsgALEF2ko4cnGIwUADm8ESgAGYKEAC2SGQgOBBIQtwVVbWIIg1NiJatIJU1SNbdSIHBOU5SaAAWaq5Jgx3jjUjp+ZxAA
\begin{tikzcd}
0 \arrow[r] & {\ker\pi_{\beta,n}} \arrow[r] & \tnxxz \arrow[r, "\phi_x"] & \cq^{G_n} \arrow[r] & 0.
\end{tikzcd}
\end{center}
\end{prop}

\begin{comment}
Diagram including $\tnxx{n}$ and $\cq^{G_n}$:

% https://tikzcd.yichuanshen.de/#N4Igdg9gJgpgziAXAbVABwnAlgFyxMJZABgBpiBdUkANwEMAbAVxiRGJAF9T1Nd9CKMgEYqtRizYduvbHgJFh5MfWatEIADqaAPMACqAPTABaYZy48QGOQMWlR1VZI3aA1jABO2tFgD6wNoARjA4dKRgFjLWfPKCyABMyk4S6lqaOGAAHlkAXpay-AooSY7iamzamTn50TZF8QDMyeUu6dVZwJEFMbbFyAAsLc5p0lb1cURDZSNSPRN2KM0zqZWaAMYAjobAAOJ+3ZxiMFAA5vBEoABmnhAAtkhkIDgQSAnRN-dv1C9IAx+3B6IIbPV6IABsAK+EJ+YIA7FCgUpQUhGojUbCkABWdGILGYxAATlxhIJAA5cSDfnjqAwsGA0lAIEwggxWNQABYwOhQNiQBkgH50LAMPkEVi4pIoxDNZ7C0UaDkQCBuHqfIHggmkkBcnlsHAAdwguqgCCOnCAA
\begin{tikzcd}
0 \arrow[r] & \<{U^n-1} \arrow[r] \arrow[d, hook] & \tnxxz \arrow[r] \arrow[d, no head, Rightarrow] & \tnxx{n} \arrow[r] \arrow[d, two heads] & 0 \\
0 \arrow[r] & {\ker\pi_{\beta,n}} \arrow[r]       & \tnxxz \arrow[r]                                & \cq^{G_n} \arrow[r]                     & 0
\end{tikzcd}
\end{comment}

\noindent In terms of the identification $\Aut(\Q/\Z)\cong \Gal(\Q^{\rm cycl}/\Q)$, $G_n$ is identified with $\Gal(\Q^{\rm cycl}/\Q(\sqrt[n]{1}))$.

\begin{proof}
The containment $\mathrm{im}(\phi_x) = \mathrm{im}(\bar{\phi}_x)\subseteq \cq^{G_n}$ follows from \lemref{lem:tnxxz-symm} (ii). The subgroup $G_n$ is compact, so there is a conditional expectation associated to the action of $G_n$,
\begin{equation}\label{eqn:bc-expect}
\Theta_n:\cq\rightarrow \cq^{G_n},\qquad \Theta_n = \int_{G_n} \theta_u du;
\end{equation}
we will show that the image of $\Theta_n$ is contained in the image of $\phi_x$. 

For $y\in \Q/\Z$, let $m = \ord(y)$ and $d = \gcd(n,m)$. Then, for $u\in \Aut(\Q/\Z)$, we have $uy = by$ for some $b\in (\Z/m\Z)^*$; moreover, $b\equiv 1$ mod $d$ when $u\in G_n$. Conversely, for any $b\in (\Z/m\Z)^*$ with $b\equiv 1$ mod $d$, by the Chinese remainder theorem, there exists $b_1\in (\Z/\frac{mn}{d}\Z)^*$ such that $b_1\equiv b$ mod $m$ and $b_1\equiv 1$ mod $n$. Since the projection from $\Aut(\Q/\Z)$ is surjective, there exists $u\in G_n$ such that $uy = by$. It follows that
$$\Theta_n(e(y)) = \frac{\euler(d)}{\euler(m)} \sum_{\substack{b\in (\Z/m\Z)^*\\b\equiv 1 \mod d}} e(by) = \frac{\euler(d)}{\euler(m)} \sum_{\substack{k=0\\\gcd(m,dk+1)=1}}^{\frac{m}{d}-1} e(y(dk+1)).$$
Using the formula $\sum_{d|n} \mu(d) = \delta_{n,1}$, this can be written
\begin{align*}
\frac{\euler(m)}{\euler(d)}\cdot \Theta_n(e(y)) &= \sum_{k=0}^{\frac{m}{d}-1} e(y(dk+1)) \sum_{c|\gcd(m,dk+1)} \mu(c)\\
&= \sum_{\substack{c|m\\ \gcd(c, d)=1}} \mu(c) \sum_{\substack{k=0\\ c|(dk+1)}}^{\frac{m}{d}-1} e(y(dk+1))
\end{align*}
If $c|m$ is relatively prime to $d$, then there exists some $0\leq \ell_c <d$ such that $c\ell_c \equiv 1$ mod $d$. The set of elements of $\{dk+1: 0\leq k<\frac{m}{d}\}$ that are divisible by $c$ is then given by $\{c(dk+\ell_c): 0\leq k<\frac{m}{cd}\}$, considered mod $m$. Moreover, since $m=\ord(y)$, the function $k\mapsto (ycd)k$ is a bijection between $\Z/\frac{m}{cd}\Z$ and $\frac{cd}{m}\Z/\Z$, so by (BC5) we have
\begin{align*}
\frac{\euler(m)}{\euler(d)}\cdot \Theta_n(e(y)) &= \sum_{\substack{c|m\\\gcd(c,d) = 1}}\mu(c) \sum_{k=0}^{\frac{m}{cd}-1} e(yc(dk+\ell_c)) = \sum_{\substack{c|m\\\gcd(c,d) = 1}} \mu(c)\frac{m}{cd} \cdot \bcmu_{\frac{m}{cd}} e\left(\frac{my\ell_c}{d}\right) \bcmu_{\frac{m}{cd}}^*.
\end{align*}
Now $\frac{my}{d}$ has order $d|n$, so it is a multiple of $x$ in $\Q/\Z$. Thus, we have shown that $\Theta_n(e(y))\in \im(\phi_x)$, and this extends to monomials by $\Theta_n(\bcmu_a e(y) \bcmu_b^*) = \bcmu_a \Theta_n(e(y)) \bcmu_b^*$. Since $\Theta_n$ is contractive, we conclude that $\im(\Theta_n)\subseteq \im(\phi_x)$, as desired.
\end{proof}

\begin{prop}\label{pro:finite-index-subfactor}
For $n\in\nx$, $\beta\in(0,1]$, let $(\HH_n,\pi_{\beta,n},\Omega)$ be the GNS triple for $\psi_{\beta,n}$. Then, for every $b\in\nx$ relatively prime to $n$, there is a unitary $W_b$ on $\HH_n$ such that: $W_a W_b = W_{ab}$, \ $W_b = W_{b'}$ whenever $b\equiv b'$ mod $n$, and $\Ad_{W_b}\circ\ \pi_{n,\beta} = \pi_{n,\beta}\circ \kappa_b$. In particular, the map $b\mapsto W_b$ defines a representation of $(\Z/n\Z)^*$ on $\HH_n$. The fixed-point von Neumann algebra $(\pi_{\beta,n}(\tnxxz)'')^W$ is a finite-index subfactor of $\pi_{\beta,n}(\tnxxz)''$ 
of type $\mathrm{III}_1$.
\end{prop}

\begin{proof}
It is convenient to instead consider the GNS triple $(\HH_n, \bar{\pi}_{\beta,n}, \Omega)$ for $\bar{\psi}_{\beta,n}$ the state on $\tnxx{n}$, where $\pi_{\beta,n} = \bar{\pi}_{\beta,n}\circ q_n$. For $b$ relatively prime to $n$, define $W_b\bar{\pi}_{\beta,n}(x)\Omega = \bar{\pi}_{\beta,n}(\bar{\kappa}_b(x))\Omega$. By \lemref{lem:tnxxz-symm} (iii), this gives a well-defined unitary representation of $(\Z/n\Z)^*$ on $\HH_n$. For $b\in \nx$ relatively prime to $n$, we have $\pi_{\beta,n}\circ \kappa_b = \bar{\pi}_{\beta,n} \circ \bar{\kappa}_b\circ q_n  = \Ad_{W_b}\circ\ \bar{\pi}_{\beta,n}\circ q_n = \Ad_{W_b}\circ\ \pi_{\beta,n}.$

Recall that a conditional expectation $E:M\rightarrow N$ of von Neumann algebras is said to have \emph{finite index} if there exists  $K>0$ such that $K\cdot E - \mathrm{Id}_M$ is a positive operator on $M$ (the smallest such $K$ is the index of~$E$). In particular, 
$$E^{(n)} =  \frac{1}{\euler(n)} \sum_{b\in (\Z/n\Z)^*} \Ad_{W_b}$$
is a conditional expectation with index at most $\euler(n)$.

Now let $(\HH,\pi_\beta,\Omega)$ be the GNS representation for the \kmsb state $\psi_\beta$ on $\cq$. Define a unitary representation of $\Aut(\Q/\Z)$ on $\HH$ by $T_u \pi_{\beta}(a) \Omega = \pi_{\beta}(\theta_u(a))\Omega$. By \cite[Proposition 21 (b)]{bos-con}, $(\pi_\beta(\cq)'')^T$ is the von Neumann algebra generated by the elements $\pi_\beta(\bcmu_a)$, which is a type $\mathrm{III}_1$ factor by \cite[Proposition 8]{bos-con}; we will argue that it is isomorphic to $(\pi_{\beta,n}(\tnxxz)'')^W$.

By \proref{pro:bc-hom}, for $x\in \Q/\Z$, $\ord(x) = n$, the map $\bar{\phi}_x$ defines an isometry
$$L_x:\HH_n\rightarrow \HH,\qquad L_x(\bar{\pi}_{\beta,n}(a))\Omega = \pi_{\beta}(\bar{\phi}_x(a)) \Omega.$$
Let $P_n = L_x L_x^*$ be the projection of $\HH$ onto $L_x(\HH_n)$, which only depends on the order of $x$. The map $\pi_\beta(a)\Omega\mapsto \pi_\beta(\Theta_n(a))\Omega$ defines a projection on $\HH$, and by \proref{prop:bc-subalg} we have
$$P_n\pi_\beta(a) \Omega = \pi_\beta(\Theta_n(a))\Omega = \int_{G_n} T_u\pi_\beta(a)\Omega du.$$
so $P_n = \int_{G_n} T_u du$ commutes with $(\pi_{\beta}(\cq)'')^T$. By \lemref{lem:tnxxz-symm} (ii) we have $P_nT_uP_n = L_x W_b L_x^*$, where $b$ is the image of $u$ in $(\Z/n\Z)^*$. Thus $L_x$ gives a spatial isomorphism between $(\pi_{\beta,n}(\tnxxz)'')^W$ and $P_n (\pi_\beta(\cq)'')^{T} P_n$. Since $P_n$ is in the commutant of the type \III{1} factor $(\pi_{\beta}(\cq)'')^T$, it follows that $(\pi_{\beta,n}(\tnxxz)'')^W$ is a type \III{1} factor.
\end{proof}

\begin{cor}[\thmref{thm:main}, type $\mathrm{III}_1$ assertion]\label{cor:maintype}
For $n\in \nx$ and $\beta\in (0,1]$, the state $\psi_{\beta,n}$ is a factor state of type \III{1}.
\end{cor}

\begin{proof}
 By \cite[Theorem 2.7]{loi}, if $M$ is a finite-index subfactor of a factor $N$ and $M$ is of type \III{1}, then $N$ is also of type \III{1}. Thus, the result is immediate from \proref{pro:finite-index-subfactor}.
\end{proof}

\section{The Toeplitz system of $\nxxqz$}\label{sec:nxxqz}

In this section we describe how our analysis of the \kmsb states on $\tnxx{n}$ can be extended to $\tnxxqz$.
The monoid $\nxx{n}$ of the previous section is naturally isomorphic to $\nxxf{n}$ by the map $(a,[x])\mapsto (a,[x/n])$. The latter form a nested system of monoids over $n\in\nx$ with union $\nxxqz$. We begin by describing a presentation of the Toepliz algebra $\tnxxqz$ in terms of generators and relations analogous to the one obtained for $\tnxx{n}$ in the preceding section.

\begin{prop}
The monoid $\nxxqz$ is left-cancellative and right LCM. The Toeplitz algebra $\tnxxqz$ is generated by elements $R_x := T_{(1,x)}$ for $x\in \Q/\Z$ and $V_a: =T_{(a,[0])}$ for $a\in\nx$, which satisfy 
\begin{enumerate}
\smallskip\item[\textup{(Q0)}] \ $V_a^* V_a = 1 = R_x R_x^* = R_x^* R_x$,
\smallskip\item[\textup{(Q1)}] \ $R_x V_a = V_a R_{ax}$,
\smallskip\item[\textup{(Q2)}] \ $V_a V_b = V_{ab}$,
\smallskip\item[\textup{(Q3)}] \ $V_a V_b^* = V_{b}^* V_a$ when $\gcd(a,b) = 1$,
\smallskip\item[\textup{(Q4)}] \ $R_x R_y = R_{x+y}$.
\end{enumerate}
Moreover, the relations {\rm (Q0)--(Q4)} constitute a presentation of $\tnxxqz$.
\end{prop}

\noindent The proof is essentially the same as in Propositions \ref{prop:nxx{n}-toeplitz} and \ref{prop:relations-modn}.

Under the identification $\nxx{n}\cong \nxxf{n}$ above, the inclusions $\nxxf{n}\subseteq \nxxf{m}$ for $n|m$ give an inductive system of *-homomorphisms 
\begin{equation}\label{eqn:tnxx-inclusion}
\iota_{n,m}:\tnxx{n}\rightarrow \tnxx{m} \quad \text{ such that } \quad \iota_{n,m}:R\mapsto R^{\frac{m}{n}},\qquad V_a\mapsto V_a.
\end{equation}

\begin{prop}\label{prop:tnxxqz-limit}
For $n\in\nx$, there are injective *-homomorphisms $\iota_n:\tnxx{n}\rightarrow \tnxxqz$ determined by
\begin{equation}
\iota_n:R\mapsto R_{\frac{1}{n}},\qquad V_a\mapsto V_a,
\end{equation}
and $(\tnxxqz, \iota_n)_{n\in\nx}$ is the limit of the system $(\tnxx{n},\iota_{n,m})_{n\in\nx}$.

In particular, the inclusion $\nxxf{n}\subseteq \nxxqz$ induces an inclusion of Toeplitz algebras $\tnxxf{n}\subseteq \tnxxqz$ such that the following diagram commutes:
\begin{center}
\begin{tikzcd}
\tnxxz \arrow[r, "q_n"] & \tnxx{n} \arrow[rd, "\iota_n"] \arrow[d, "\cong"'] &         \\
                        & \tnxxf{n} \arrow[r, "\subseteq"]                   & \tnxxqz
\end{tikzcd}
\end{center}
\end{prop}

We give $\tnxxqz$ the dynamics $\sigma_t(T_{(a,[x])}) = a^{it} T_{(a,x)}$. The maps $\iota_{n,m}$ and $\iota_n$ are equivariant for the dynamics, so \proref{prop:tnxxqz-limit} gives us a first description of the \kmsb simplex of $\tnxxqz$, as the projective limit of the system $(K_\beta^{n}, \iota_{n,m}^*)_{n\in\nx}$, where $K_\beta^n$ is the \kmsb simplex of $\tnxx{n}$. We will describe this limit in the cases $\beta\in (1,\infty)$ and $\beta\in (0,1]$ with the help of the following interpretations of the KMS states of $\tnxxz$ and $\tnxx{n}$.

The group of characters of $\Z$ is isomorphic to $\T$ and the group of characters of $\Z/n\Z$ is isomorphic to $Z_n$, the $n^{\mathrm{th}}$ roots of unity. Thus, the extremal \kmsb states of $\tnxxz$ and $\tnxx{n}$ for $\beta\in(1,\infty)$ can be reparameterized by the character groups of $\Z$ and $\Z/n\Z$, respectively (cf. \cite[Theorem 8.1 (2)]{aHLR21} and \thmref{thm:finite-kms}). For a character $\chi$ on $\Z$, the reparametrization for states of $\tnxxz$ is given by
$$\psi_{\beta,\chi}(V_a U^k V_b^*) = \delta_{a,b} \frac{a^{-\beta}}{\zeta(\beta)}\sum_{c=1}^\infty c^{-\beta} \chi(c k).$$

Recall that the Chabauty topology is the topology on the set $\subg(G)$  of closed subgroups of $G$ with a basis of neighborhoods for $C\in \subg(G)$ given by
$$V_C(K,U) = \{D\in \subg(G): D\cap K\subseteq CU\}$$
for $K\subseteq G$ compact and $U\subseteq G$ an open neighborhood of identity, cf. \cite{Ch50}. The (closed) subgroups of $\Z$ are parameterized by the set $\nx\cup\{\infty\}$, where $H_n = n\Z$ and $H_\infty = 0\Z$, and this identifies $\subg(\Z)$ with the one-point compactification of $\nx$. The subgroups of $\Z/n\Z$ are parameterized by divisors of $n$, $H_d = d\Z/n\Z$. Thus, by Theorems \ref{thm:main}(1) and \ref{thm:finite-kms}, the \kmsb states of $\tnxxz$ and $\tnxx{n}$ for $\beta\in (0,1]$ can be reparametrized by the spaces $\subg(\Z)$ and $\subg(\Z/n\Z)$, respectively. Note that for $n\in \nx\cup\{\infty\}$, $\ord_{\Z/H_n}(k) = \frac{n}{\gcd(n,k)}$. Then, for a subgroup $H\subseteq \Z$, the reparametrization of states of $\tnxxz$ gives
$$\psi_{\beta,H} (V_a U^k V_b^*) = \delta_{a,b} a^{-\beta} \ord_{\Z/H}(k)^{-\beta}\sum_{d|\ord_{\Z/H}(k)} \mu(d)\frac{\euler_\beta(d)}{\euler(d)},$$
with the convention that $(\infty)^{-\beta} = 0$.

Since we are considering states on a family of C*-algebras indexed by $n\in\nx$, we write $\bar{\psi}_{\beta}^n$ (rather than $\bar{\psi}_\beta$) for a \kmsb state on $\tnxx{n}$ and $\psi_{\beta}^\infty$ for a \kmsb state on $\tnxxqz$, in order to clarify the domain of the state. The maps $\iota_{n,m}^*$ are readily computed in terms of characters and subgroups of $\tfrac{1}{n}\Z/\Z\cong \Z/n\Z$. For $\beta\in (1,\infty)$ and $\chi\in (\frac{1}{m}\Z/\Z)\wh{\ }$\ , we have $$\iota_{n,m}^*(\bar{\psi}^m_{\beta,\chi}) = \bar{\psi}^m_{\beta,\chi}\circ\iota_{n,m} = \bar{\psi}^n_{\beta,\chi|_{\frac{1}{n}\Z/\Z}},$$ 
and for $\beta\in (0,1]$ and a subgroup $H\subseteq \frac{1}{m}\Z/\Z$, we have 
$$\iota_{n,m}^*(\bar{\psi}_{m,\beta,H}) = \bar{\psi}^m_{\beta,H}\circ\iota_{n,m} = \bar{\psi}^n_{\beta,H\cap \frac{1}{n}\Z/\Z}.$$

\begin{lemma}\label{lem:proj-limits}
Let $m, n\in\nx$ and assume $n|m$, so that $\tfrac{1}{n} \Z/\Z \subset \tfrac{1}{m}\Z/\Z$.
\begin{enumerate}
\item Suppose $\beta\in(1,\infty)$ and let  $\eta \in (\tfrac{1}{n} \Z/\Z)\wh{\ }$. Then $\eta = \chi|_{\tfrac{1}{n} \Z/\Z}$ for a character  $\chi\in (\tfrac{1}{m}\Z/\Z)\wh{\ }$ if and only if $\bar{\psi}^m_{\beta,\chi}\circ \iota_{n,m} = \bar{\psi}^n_{\beta,\eta}.$
\smallskip
\item Suppose $\beta\in (0,1]$ and let  $K$ be a subgroup of $\tfrac{1}{n} \Z/\Z$. Then $K = H\cap (\tfrac{1}{n} \Z/\Z)$ 
for a subgroup $H$ of $\tfrac{1}{m}\Z/\Z$ if and only if  $\bar{\psi}^m_{\beta,H}\circ\iota_{n,m} = \bar{\psi}^n_{\beta,K}.$
\end{enumerate}
\end{lemma}

\begin{proof}
(1) If the character  $\eta \in (\tfrac{1}{m} \Z/\Z)\wh{\ }$ is the restriction  of a character $\chi\in (\tfrac{1}{n}\Z/\Z)\wh{\ }$, in the sense that 
$\eta(x) = \chi(x)$ for every $x\in \tfrac{1}{n} \Z/\Z$, then
 a computation using formula \eqref{eqn:finite-subcrit} shows that
 \[
\bar{\psi}^m_{\beta,\chi}\circ\iota_{n,m}(V_a R_x V_b^*) =  \bar{\psi}^n_{\beta,\eta} (V_a R_x V_b^*) \qquad a,b\in \nx, \ x\in \tfrac{1}{n} \Z/\Z.
 \]

Conversely, suppose that $\bar{\psi}^m_{\beta,\chi}\circ\iota_{n,m} = \bar{\psi}^n_{\beta,\eta}$ for characters 
$\chi\in (\tfrac{1}{m}\Z/\Z)\wh{\ }$  and $\eta\in (\tfrac{1}{n}\Z/\Z)\wh{\ }$.
For every $x\in \tfrac{1}{n}\Z/\Z$, we have
\[
\zeta(\beta)\sum_{k\in\nx}\mu(k)k^{-\beta}\bar{\psi}^m_{\beta,\chi}(R_{x}^k) = \sum_{k\in\nx} \sum_{c\in\nx} \mu(k) (ck)^{-\beta} \chi(x)^{ck} = \chi(x).
\]
hence $\chi(x) = \eta(x)$ for $x\in \tfrac{1}{n} \Z/\Z$, as desired.

(2) If the subgroup $K \leq  \tfrac{1}{n}\Z/\Z$ is the intersection $K = H \cap \tfrac{1}{n}\Z/\Z$ for  a subgroup $H \leq  \tfrac{1}{m}\Z/\Z$, then 
\[
\ord_{(\tfrac{1}{m}\Z/\Z)/H}(x + H) = \ord_{(\tfrac{1}{n}\Z/\Z)/K}(x + K) \qquad x\in \tfrac{1}{n} \Z/\Z,
\]
 hence  formula \eqref{eqn:finite-supcrit} gives $$\bar{\psi}^m_{\beta,H}\circ\iota_{n,m}(V_a R_x V_b^*) = \bar{\psi}^n_{ \beta,K}(V_a R_x V_b^*)\qquad a,b\in \nx,\ x\in\tfrac{1}{n}\Z/\Z.$$

Conversely, suppose that $\bar{\psi}^m_{\beta,H}\circ\iota_{n,m} = \bar{\psi}^n_{\beta,K}$.
Setting $a = b = 1$ in \eqref{eqn:finite-supcrit}, we see that  
\[
\bar{\psi}^m_{\beta,H}(R_x) =   \ord(x+H)^{-\beta} \sum_{d|\ord(x+H)} \mu(d) \frac{\euler_\beta(d)}{\euler(d)}
\]
only depends on $x$  through  the value $c:=\ord(x+H) $. Motivated by this, we define an arithmetic function $h$ by 
\[
h(c) := c^{-\beta} \sum_{d|c} \mu(d) \frac{\euler_\beta(d)}{\euler(d)} = \sum_{d|c} \Big(\mu(d) d^{-1}\prod_{p|d} \frac{1-p^{-\beta}}{1-p^{-1}}\Big) \left(\frac{c}{d}\right)^{-\beta},
\]
which shows that 
$h $ is the Dirichlet convolution $f*g$ of the arithmetic functions $f$ and $g$ defined by
\[
f(c) = \mu(c)c^{-1}\prod_{p|c} \frac{1-p^{-\beta}}{1-p^{-1}},\qquad g(c) = c^{-\beta}  \qquad \text{for }c\in \nx.
\]
If $c$ divides $[(\tfrac{1}{m}\Z/\Z):H]$, then $c = \ord(x+H)$ for some $x\in \tfrac{1}{m}\Z/\Z$, and 
 $h(c) = \bar{\psi}^m_{\beta,H}(R_x)$ for every such $x$.
Since $f(1) = 1 $, the function $f$ is invertible in the Dirichlet ring of arithmetic functions. Hence
$$c^{-\beta} = g(c) = (f^{-1}*h)(c) = \sum_{d|c} f^{-1}(d) h(c/d) $$
If $c$ divides $[(\tfrac{1}{m}\Z/\Z):H]$, this is $ \sum_{d|c} f^{-1}(d) \bar{\psi}^m_{\beta,H}(R_x^d)$ for $c = \ord(x+H)$.

Since $\bar{\psi}^m_{\beta,H}(R_x) = \bar{\psi}^n_{\beta,K}(R_x)$ for every $x\in \tfrac{1}{n} \Z/\Z$ by assumption, it follows that 
\[
\ord(x + (H\cap (\tfrac{1}{n} \Z/\Z)))  = \ord(x + H) = \ord(x + K).
\]
Thus, $x\in H\cap \tfrac{1}{n} \Z/\Z$ if and only if $x\in K$, so the two subgroups are equal.
\end{proof}

For each $H\in \subg(\frac{1}{m}\Z/\Z)$ and $n|m$, we define $h_{n,m}(H) := H\cap \frac{1}{n}\Z/\Z $. This gives a projective system $(\subg(\frac{1}{n}\Z/\Z), h_{n,m})_{n \in\nx}$ of finite spaces, and there is a homeomorphism
\begin{align*}
h:\subg(\Q/\Z) &\stackrel{\cong}{\longrightarrow}\varprojlim (\subg(\tfrac{1}{n}\Z/\Z), h_{n,m})_{n\in\nx},\\
H &\longmapsto (H\cap \tfrac{1}{n}\Z/\Z)_{n\in\nx},
\end{align*}
where the left hand side is endowed with the Chabauty topology and the right hand side has the 
profinite topology.
Indeed, the map $h$ is injective because
 $\bigcup_n (H\cap \frac{1}{n}\Z/\Z) =  H \cap \bigcup_n \frac{1}{n}\Z/\Z = H$, and it is surjective because, for every net  $(H_n)_{n\in\nx}$ of subgroups satisfying $H_m\cap \frac{1}{n}\Z/\Z = H_n$ whenever $n|m$,  the set $H = \bigcup H_n$ is a subgroup of $\Q/\Z$ such that $h(H) =  (H_n)_{n\in\nx}$. 

 \begin{theorem}\label{thm:limit-kms}
Let $\sigma:\R\rightarrow \operatorname{Aut} \tnxxqz$ be the dynamics on $\tnxxqz$ determined by $\sigma_t(R_x) = R_x$ and $\sigma_t(V_a) = a^{it} V_a$. 
\begin{enumerate}%[(1)]
\item For  each $\beta\in (1,\infty)$ and character $\chi\in(\Q/\Z)\,\wh{\ }$, there is a unique extremal \kmsb state of $\tnxxqz$ such that
\begin{equation}\label{eqn:limit-subcrit}
\psi^\infty_{\beta,\chi}(V_a R_x V_b^*) = \delta_{a,b} \frac{a^{-\beta}}{\zeta(\beta)} \sum_{c\in\nx} c^{-\beta} \chi(x)^c.
\end{equation}
The map $\chi\mapsto \psi^{\infty}_{\beta,\chi}$ extends to an affine w*-homeomorphism of the simplex of probability measures on $(\Q/\Z)\wh{\ }$ onto the simplex of \kmsb states of $(\tnxxqz,\sigma)$. For each $n\in \nx$, $$\psi^\infty_{\beta,\chi}\circ \iota_n = \bar{\psi}^n_{\beta,\chi_n},\quad \chi_n = \chi|_{\frac{1}{n}\Z/\Z}.$$
\smallskip\item For each $\beta\in(0,1]$ and subgroup $H\subseteq \Q/\Z$, there is a unique extremal \kmsb state of $\tnxxqz$ such that
\begin{equation}\label{eqn:limit-supcrit}
\psi^\infty_{\beta,H}(V_a R_x V_b^*) = \delta_{a,b} a^{-\beta} \ord\left(x + H\right)^{-\beta} \sum_{d|\ord\left(x+H\right)} \mu(d) \frac{\euler_\beta(d)}{\euler(d)}.
\end{equation}
The map $H\mapsto \psi^\infty_{\beta,H}$ extends to an affine w*-homeomorphism of the simplex of probability measures on $\subg(\Q/\Z)$ onto the simplex of \kmsb states of $(\tnxxqz,\sigma)$. For each $n\in \nx$,
$$\psi^\infty_{\beta,H}\circ \iota_n = \bar{\psi}^n_{\beta,H_n},\quad H_n = H\cap (\tfrac{1}{n}\Z/\Z).$$
\end{enumerate}
\end{theorem}

\begin{proof}

It follows from \proref{prop:tnxxqz-limit} that the simplex of \kmsb states of $(\tnxxqz,\sigma_t)$ is canonically homeomorphic to the projective limit of the simplices $K_{n,\beta}$ of \kmsb states under the maps $\iota_{n,m}^*$. We will show that the extreme points of this simplex are determined by either \eqref{eqn:limit-subcrit} or \eqref{eqn:limit-supcrit} on monomials (depending on if $\beta\in (1,\infty)$ or $(0,1]$), whence the result follows.

(1) For $\beta\in (1,\infty)$, by \thmref{thm:finite-kms}(1) and \lemref{lem:proj-limits}(1), the system $(\partial_e K_{n,\beta}, \iota_{n,m}^*)_{n\in \nx}$ is naturally isomorphic to $((\frac{1}{n}\Z/\Z)\wh{\ }, r_{n,m})_{n\in\nx}$, where $r_{n,m}(\chi) = \chi|_{\frac{1}{n}\Z/\Z}$. It follows from $\Q/\Z = \bigcup \frac{1}{n}\Z/\Z$ that $((\Q/\Z)\wh{\ }, r_n)$ is the limit of $((\frac{1}{n}\Z/\Z)\wh{\ }, r_{n,m})$. For $\chi \in (\Q/\Z)\wh{\ }$, if $\chi_n = r_n(\chi)$, then $(\bar{\psi}^n_{\beta,\chi_n})_{n\in \nx}$ is a coherent family of \kmsb states that determines a state $\psi$ on $\tnxxqz$. Since $\chi(x) = \chi_n(x)$ for $x\in \frac{1}{n}\Z/\Z$, formulas \eqref{eqn:finite-subcrit} and \eqref{eqn:limit-subcrit} imply that $\bar{\psi}^n_{\beta,\chi_n} = \psi^\infty_{\beta,\chi} \circ \iota_n$ on monomials, so $\psi = \psi^\infty_{\beta,\chi}$.

(2) For $\beta\in(0,1]$, by \thmref{thm:finite-kms}(2) and \lemref{lem:proj-limits}(2), the system $(\partial_e K_{n,\beta}, \iota_{n,m}^*)_{n\in \nx}$ is naturally isomorphic to $(\subg(\frac{1}{n}\Z/\Z), h_{n,m})_{n\in\nx}$, where $h_{n,m}(H) = H\cap \frac{1}{n}\Z/\Z.$ Then $\subg(\Q/\Z)\cong \varprojlim (\subg(\frac{1}{n}\Z/\Z), h_{n,m})_{n\in\nx}$, so for every $H\in \subg(\Q/
\Z)$ and $H_n = h_n(H)$, $(\bar{\psi}^n_{\beta,H_n})_{n\in\nx}$ is a coherent family of \kmsb states that determines a state $\psi$ on $\tnxxqz$. Since $\ord_{(\Q/\Z)/H}(x) = \ord_{(\frac{1}{n}\Z/\Z)/H_n}(x)$ for every $x\in \frac{1}{n}\Z/\Z$, formulas \eqref{eqn:finite-supcrit} and \eqref{eqn:limit-supcrit} imply that $\bar{\psi}^n_{\beta,H_n} = \psi^\infty_{\beta,H}\circ\iota_n $ on monomials, so $\psi = \psi^\infty_{\beta,H}$. 
\end{proof}

\begin{comment}
\begin{remark} For $\beta>1$ it is also possible to determine the  \kmsb states of $\tnxx{n}$ and of $\tnxxqz$ directly via similar arguments to those of \cite{aHLR21}.
Indeed, our \secref{sec:atomic} can be viewed as a computation of the \kmsb states of $\tnxx{n}$, while \thmref{thm:atomic} shows that composing these states with $q_n$ provide all extremal \kmsb states on $\tnxxz$ that correspond to atomic subconformal measures.
\end{remark}
\end{comment}

\appendix
\label{sec:appendixA}
\section{Dilation/extension results} 
 We prove a dilation/extension of the semigroup action of $\nx$ on $\mfd$ and show that every $\beta$-subconformal measure on $\Spec \mfd$ extends to the dilated system. 

The semigroup action of $\nx$ by injective endomorphisms of $\mfd$  from \proref{eqn:semiXprod} satisfies the dilation/extension conditions of Theorem 2.1 and Theorem 2.4 of \cite{Laca-JLMS}. Hence there exist a C*-algebra $\tilde{\mfd}$, an embedding $i:\mfd\rightarrow \tilde{\mfd}$, and  an action $\tilde{\alpha}: \qx \to \Aut(\tilde{\mfd})$,   such that
\begin{enumerate}
\smallskip\item $\tilde{\alpha}_a$ dilates $\alpha_a$, that is, $i\circ\alpha_a = \tilde{\alpha}_a\circ i$ for $a\in\nx$;
\smallskip\item $\tilde{\mfd}$ is minimal with respect to  $\tilde\alpha$, that is, $\bigcup_{a\in\nx} \tilde{\alpha}_a^{-1}(i(\mfd))$ is dense in $\tilde{\mfd}$;
\smallskip\item $\nx\ltimes_{\alpha} \mfd$ is the full corner in $\qx\ltimes_{\tilde\alpha} \tilde{\mfd}$ corresponding to the projection $i(1_{\mfd})$.
\end{enumerate}
Explicitly, $\wt\mfd$ is the direct limit of the system $(\mfb_a, \alpha_{b,a})_{a|b}$, in which the C*-algebras are $\mfb_a = \mfd$ for all $a$,  and the connecting maps $\alpha_{b,a}:\mfb_a\rightarrow \mfb_b$ are given by $\alpha_{b,a}(x) = \alpha_{\frac{b}{a}}(x)$ when $a \mid b$. We write $i_a$ for the canonical embedding $i_a:\mfb_a\rightarrow \mfd$. 

Both $\mfd$ and $\tilde\mfd$ are commutative C*-algebras, and we let $X =\Spec \mfd$ and $\tilde X = \Spec \tilde{\mfd}$. The image  $i(1_\mfd)$ of the identity of $\mfd$ 
is a full  projection in $\tilde{\mfd}$ and there is a homeomorphism $i_*:X\rightarrow \supp i(1_\mfd)$ identifying $X$ with a compact open subset  of $\tilde X$; more generally, there is a homeomorphism $i_{a*}:X\rightarrow \supp i_a(1_\mfd)$ for each $a\in \nx$ whose image is the translate of $i_*(X)$ under $\wt{\alpha}_a^*$. By the minimality condition (2), the union of the $i_{a*}(X)$ is $\tilde X$. In particular, this says that $C_c(\wt{X})$ is equal to $\bigcup_{a\in \nx} i_a(\mfb_a)$: any function belonging to $\bigcup_{a\in \nx} i_a(\mfb_a)$ is compactly supported, and conversely, if $f$ is compactly supported, then there exists a subset $F\Subset \nx$ such that $\supp(f)\subseteq \bigcup_{a\in F}i_{a*}(X)\subseteq i_{\lcm F*}(X)$.

Next we wish to show that  states on $\mfd$ that satisfy a rescaling condition with respect to $\alpha$ extend to densely defined locally finite weights on $\tilde\mfd$. 
This  has been used to study KMS states of the Bost-Connes system and its generalizations, see e.g. \cite{diri,nes}. We provide the details here for completeness, 
formulating things in terms of Radon measures on $\tilde X$, which represent positive linear functionals on compactly supported functions in $\tilde\mfd$ by the Riesz--Markov--Kakutani Theorem. 

\begin{lemma} \label{lem:rescalingextension}
If $\phi$ is a state of $\mfd$ such that  $\phi \circ \alpha_a = a^{-\beta} \phi$ for all $a\in \nx$, then there is  a unique  positive linear functional  $\tilde \phi$ on $C_c(\tilde X) \subset \wt\mfd$ such that 
$ \tilde\phi (i (f) )= \phi(f)$  for every  $ f \in \mfd$    and  $\tilde\phi \circ \tilde\alpha_r = r^{-\beta} \tilde \phi$ for every  $ r\in \qx$.
Conversely,  if   $ \eta$ is a positive linear functional on $C_c(\wt{X})$ normalized so that $\eta (i (1_{\mfd}) ) =1$ and satisfying $\eta \circ \tilde\alpha_r = r^{-\beta}
 \eta$ for every  $ r\in \qx$, then the restriction $ \phi:= \eta \circ i$ to $\mfd$ is a state such that  $\phi \circ \alpha_a = a^{-\beta} \phi$ for all $a\in \nx$; moreover, $\tilde\phi = \eta$.

\begin{proof} Suppose first that $\phi$ is a state of $\mfd$ such that  $\phi \circ \alpha_a = a^{-\beta} \phi$. For all $a\in \nx$
define a  positive linear functional $\wt{\phi}_a = a^\beta \phi$ on $\mfb_a := \mfd$. Then 
\[
\wt{\phi}_b\circ \tilde{\alpha}_{b,a} =  b^{\beta}\left(\frac{b}{a}\right)^{-\beta} \phi = \wt{\phi}_a,
\]
so that $(\wt{\phi}_a)_{a\in\nx}$ is a coherent family for the inductive system $(\mfb_a,\alpha_{a,b})_{a\in\nx}$. Since $C_c(\wt{X})$ is equal to $\bigcup_{a\in\nx} i_a(\mfb_a)$, there is a unique positive linear functional $\wt{\phi}$ on $C_c(\wt{X})$ that agrees with $\wt{\phi}_a$ when restricted to $\mfb_a$. This implies that $\wt{\phi}\circ i = \phi$. 
Additionally, if $x \in \mfb_b$, then for every $a\in\nx$,
\[
\wt{\phi}(\tilde{\alpha}_a(x)) = b^{\beta}\phi(V_a x V_a^*) = a^{-\beta} \wt{\phi}(x),
\]
which implies that $\wt{\phi}\circ \tilde{\alpha}_a = a^{-\beta} \wt{\phi}$ for  $a\in\qx = (\nx)\inv \nx$.

Suppose now $\eta$ is a linear functional on $C_c(\wt{X})$ normalized to $\eta(i(1_\mfd)) = 1$ and satisfying $\eta\circ \tilde{\alpha}_a = a^{-\beta}\eta$. By minimality of the dilation/extension, property (2) above, it follows that $\eta$ is determined by its restrictions to the subspaces $\tilde{\alpha}_a^{-1}(i(\mfd))$ with $a\in\nx$. Then $ \phi:= \eta\circ i$ is a state on $\mfd$ because of the normalization assumption, and it satisfies
\[
\phi\circ \alpha_a (x)  = \eta( i(\alpha_a(x))) = \eta( \wt{\alpha}_a(i(x)))   =a^{-\beta} \eta(i(x)) = a^{-\beta} \phi(x).
\]
Therefore $\eta = \tilde\phi$, proving that the extension is unique and $\phi \mapsto \tilde\phi$  is a bijection . 
\end{proof}
\end{lemma}
\begin{lemma}\label{lem:dilationmeasure} Suppose $\nu$ is a $\beta$-subconformal probability measure on $\T$
and   $\phi_{\beta,\nu}$ is the
\kmsb state  associated to $\nu$ in equation \eqref{eqn:KMScharactBoundQuot}.
Then the restriction of $\phi_{\beta,\nu}$ to $\mfd$ has a unique extension to a positive linear functional $\tilde \phi_{\beta,\nu}$ on $C_c(\tilde X) \subset \wt\mfd$ such that
$\wt{\phi}_{\beta,\nu}\circ i = \phi_{\beta,\nu}$ and $\wt{\phi}_{\beta,\nu}\circ\tilde{\alpha}_a = a^{-\beta} \wt{\phi}_{\beta,\nu}$.  By the same token, the measure $\wt\nu_\beta$ on $\tilde X =\Spec \wt\mfd$ representing $\wt{\phi}_{\beta,\nu}$ is the unique extension of the measure  $\nu_\beta$ on $X =\Spec \mfd$ representing ${\phi}_{\beta,\nu}$  that satisfies rescaling: $\alpha_{a*} \wt\nu_\beta = a^{-\beta} \wt\nu_\beta$. 
\end{lemma}
\begin{proof}  The restriction of a \kmsb state to  $\mfd$ satisfies  rescaling, so \lemref{lem:rescalingextension} applies.
\end{proof}

\end{document}